\documentclass[12pt,a4paper,reqno]{article}
\usepackage{amsmath}
\usepackage{amsthm}
\usepackage{enumerate}
\usepackage{amssymb}

\usepackage{verbatim}
\usepackage{url}
\usepackage[latin1]{inputenc}

\usepackage{caption}
\usepackage{graphicx,color}
\def\be{\begin{equation}}
\def\ee{\end{equation}}
\def\bea{\begin{equation*}}
\def\eea{\end{equation*}}
\def\begs{\begin{split}}
\def\ends{\end{split}}

\newtheorem{thm}{Theorem}
\newtheorem{lma}[thm]{Lemma}
\newtheorem{cor}[thm]{Corollary}
\newtheorem{prop}[thm]{Proposition}
\newtheorem{claim}{Claim}
\newtheorem*{question}{Question}

\newtheorem*{thm*}{Theorem}
\newtheorem*{prop*}{Proposition}

\newtheorem{df}[thm]{Definition}

\theoremstyle{remark}
\newtheorem{preremark}[thm]{Remark}
\newtheorem{preex}[thm]{Example}

\newenvironment{acknow}{{\bf Acknowledgements:}}{}

\title{On the chemical distance in critical percolation}
\author{Michael Damron \thanks{The research of M. D. is supported by NSF grant DMS-0901534.} \\ \small{Indiana University, Bloomington}  \and Jack Hanson \\ \small{Indiana University, Bloomington} \and Philippe Sosoe\thanks{The research of P. S. is supported by the Center for Mathematical Sciences and Applications at Harvard University.} \\ \small{CMSA, Harvard}}

\begin{document}

\maketitle

\abstract{We consider two-dimensional critical bond  percolation. Conditioned on the existence of an open circuit in an annulus, we show that the ratio of the expected size of the shortest open circuit to the expected size of the innermost circuit tends to zero as the side length of the annulus tends to infinity, the aspect ratio remaining fixed. The same proof yields a similar result for the lowest open crossing of a rectangle. In this last case, we answer a question of Kesten and Zhang by showing in addition that the ratio of the length of the shortest crossing to the length of the lowest tends to zero in probability. This suggests that the chemical distance in critical percolation is given by an exponent strictly smaller than that of the lowest path.}

\section{Introduction}

The object of this paper is to prove a result concerning the chemical distance inside large open clusters in critical independent bond percolation on $\mathbb{Z}^2$. The chemical distance between two sets $A$ and $B$ is the minimum number of edges in any lattice path of open edges joining $A$ to $B$. 

Distances inside the infinite cluster in supercritical percolation are known to be comparable to the Euclidean distance on $\mathbb{Z}^d$, through the work of G. Grimmett and J. Marstrand \cite[Section 5 (g)]{grimmettmarstrand}. P. Antal and A. Pisztora \cite{antalpisz} give exponential bounds for the probability of deviation from this linear behavior. 

By contrast, little is known in the critical case. The most complete results are available in high dimensions ($d\ge 19$). Using techniques of G. Kozma and A. Nachmias \cite{kozmanachmias1, kozmanachmias2}, R. van der Hofstad and A. Sapozhnikov \cite[Theorem 1.5]{vdhs} have shown that, conditioned on the existence of an open path to Euclidean distance $n$, the chemical distance from the origin to the boundary of a Euclidean box of side length $n$ is at least of order $\epsilon n^2$ with probability at least $1-C\sqrt{\epsilon}$. The matching upper bound follows directly from the work of Kozma and Nachmias (see also \cite[Theorem 2.8]{hhh} for a more general result, which applies also to long-range percolation). These estimates presumably hold for any dimension above the critical dimension $d=6$, but the current proofs rely on results derived from the lace expansion. To the best of our knowledge, there is currently no rigorous work addressing the chemical distance in percolation for $2<d<19$.

Despite the remarkable progress in the study of planar critical percolation in the last 15 years, the question of the chemical distance has remained mysterious. As observed by Pisztora \cite{pisztora}, the work of M. Aizenman and A. Burchard \cite{aizenman} implies that distances in planar critical percolation are bounded below by a power greater than one of the Euclidean distance, with high probability. Letting $B(n)=[-n,n]^2$, there is an $\epsilon>0$ such that, for any $\kappa>0$,
\begin{equation} \label{boundwithtext}
\mathbf{P}(\exists \text{ an open crossing of } B(n) \text{ with cardinality } \le n^{1+\epsilon}\mid \exists \text{ an open crossing})\le C_\kappa n^{-\kappa}.
\end{equation}
For definiteness, we consider horizontal crossings of $B(n)$. Pisztora treats the ``near-critical'' case, when the percolation parameter $p$ is sufficiently close to $p_c=\frac{1}{2}$ and obtains essentially the same result as long as $n$ is below the correlation length for $p$. H. Kesten and Y. Zhang \cite{kestenzhang} had previously outlined a proof of an estimate analogous to \eqref{boundwithtext} for some fixed $\kappa$, for the size of the \emph{lowest} open crossing in $B(n)$. 

We know of no explicit estimate for $\epsilon$ in \eqref{boundwithtext}. In principle, such an estimate could be obtained from careful examination of the proof in \cite{aizenman}, but the resulting value would be exceedingly small, and it is not likely to correspond to the true typical length of crossings.

In this work we will be concerned with upper, rather than lower bounds for the chemical distance. Conditioned on the existence of a crossing, the obvious approach is to identify a distinguished crossing of $B(n)$ whose size can be estimated. This provides an upper bound for the shortest crossing.

The lowest open crossing of $B(n)$ has a well-known characterization: an edge $e\in B(n)$ lies on the lowest open crossing if and only if it is connected to the left and right sides of $B(n)$ by disjoint open paths, and the dual edge $e^*$ is connected to the bottom side of $B(n)^*$, the dual to $B(n)$. (For precise definitions, see Section \ref{sec: defs}.) G. J. Morrow and Zhang have used this fact to show that if $\tilde{L}_n$ is the size of the lowest open crossing of $B(n)$, then for each positive integer $k$,
\begin{equation}\label{eqn: zhangmorrow}
C_{1,k}n^{2k}(\pi_3(n))^k \le \mathbf{E}\tilde{L}_n^k \le C_{2,k} n^{2k}(\pi_3(n))^k,
\end{equation}
with $\pi_3(n)$ denoting the ``three-arm'' probability (see \eqref{eqn: a3n}). 
On the triangular lattice, the existence and asymptotic value of the three-arm exponent are known \cite{schrammwerner}, and \eqref{eqn: zhangmorrow} becomes
\[\mathbf{E}\tilde{L}_n^k=n^{4k/3+o(1)}.\]
It is natural to ask whether this is also the correct order of magnitude for the shortest crossing of $B(n)$. This question was asked by Kesten and Zhang in \cite{kestenzhang}:
\begin{question}[H. Kesten and Y. Zhang, 1992]
Let $H_n$ be the event that there is an open horizontal crossing of $[-n,n]^2$. Let $\tilde{S}_n$ be the number of edges in the crossing of $[-n,n]^2$ of minimal length. Is it the case that
\begin{equation}\label{eqn: kzquestion}
\tilde{S}_n/ \tilde{L}_n \rightarrow 0,
\end{equation}
\end{question} 
\noindent in probability, conditionally on $H_n$? From \cite[p. 603]{kestenzhang}: ``It is not clear that $\tilde{S}_n/\tilde{L}_n\rightarrow 0$ in probability.'' In this paper, we give a positive answer to this question. (See Corollary \ref{cor: probratio}.)

We present our result on the chemical distance in terms of circuits in annuli. The same proof, with minor modifications, applies to the case of horizontal crossings. Let $A(n) = B(3n)\setminus B(n)$. By Russo-Seymour-Welsh (RSW) \cite{russo, seymourwelsh} estimates, the probability that there is an open circuit around $B(n)$ in $A(n)$ is bounded below by a positive number independent of $n$. Conditioned on the existence of such a circuit, one defines the innermost open circuit $\gamma_n$ as the circuit with minimal interior surrounding $B(n)$ inside $A(n)$. As in the case of the lowest path, one can show that if $L_n$ is the size of $\gamma_n$, then for some $C>0$
\[(1/C) n^2\pi_3(n) \le \mathbf{E}L_n \le C n^2\pi_3(n).\]
Let $S_n$ be the number of edges on the shortest open circuit around $B(n)$ in $A(n)$ (defined to be zero when there is no circuit). Our main result, Theorem \ref{thm: mainthm}, is the following. 
\begin{thm*}
As $n\rightarrow \infty$,
\begin{equation}\label{eqn: probresult} 
\frac{\mathbf{E}S_n}{n^2\pi_3(n)} \rightarrow 0.
\end{equation}
\end{thm*}
This shows that in an averaged sense, $S_n$ is much shorter than the typical size of $L_n$. The formulation \eqref{eqn: probresult} in terms of circuits in annuli serves as an illustration of the fractal nature of percolation clusters. If macroscopic open paths were smooth, in the sense that they had no small-scale features, one would not expect the shortest circuit to be much shorter than the innermost, since the latter encloses a smaller area.

\subsection{Conjectures in the literature}
Here we make a few brief remarks and give additional references to the literature on the subject of the chemical distance in critical percolation.

 Physicists expect that there exists an exponent $d_{min}$ such that
\begin{equation}
\tilde{S}_n \sim n^{d_{min}},
\end{equation}
where the precise meaning of the equivalence $\sim$ remains to be specified. O. Schramm included the determination of $d_{min}$ in a list of open problems on conformally invariant scaling limits \cite{schrammsurvey}, noting that the question does not ``seem accessible to SLE methods.'' Even the existence claim has so far not been substantiated.
 
Following Schramm and Kesten-Zhang, we have formulated the problem in terms of crossings of large boxes. More generally, $d_{min}$ is predicted to govern the chemical distance between any two points inside the same critical percolation cluster in the sense that if $x,y\in \mathbb{Z}^2$ are connected by an open path and $\|x-y\|_1= n$, then
\begin{equation} \label{eqn: point-to-point}
\mathrm{dist}_{\text{chemical}}(x,y)\sim n^{d_{min}}.
\end{equation}
%To our knowledge, there is currently no rigorous, non-trivial upper bound for the point-to-point chemical distance. 
It follows from the results of Aizenman and Burchard that if $x$ and $y$ are at Euclidean distance of order $n$, then with high probability, the chemical distance between $x$ and $y$ is greater than $n^{\eta}$ for $\eta>1$.  One might expect, based on \eqref{eqn: zhangmorrow}, that the average point-to-point chemical distance can be bounded by $n^2\pi_3(n)$, but this bound does not follow directly from the method of Morrow and Zhang. Our main result and numerical simulations suggest that a sharp upper bound would involve a quantity smaller than $n^2\pi_3(n)$ by a power of $n$.

Simulations have yielded the approximation $d_{min}\approx 1.130\ldots$ \cite{grassberger,hermann-stanley,zydz}. In contrast to other critical exponents, there is no agreement on an exact value for $d_{min}$, and several proposed values seem inconsistent with each other, and with numerical results. See the introduction and bibliography in \cite{poseetal} for a more extensive review of these questions. In that article, the authors use the formula of V. Beffara \cite{beffara} for the dimension of SLE curves
\[d_{\mathrm{SLE}(\kappa)}=\mathrm{min}\left(1+\frac{\kappa}{8},2\right)\]
along with a conjectured value for $d_{min}$ to compare, based on simulations, the behavior of SLE($\kappa$) with the shortest path accross a domain.

\subsection{Outline of the proof}
Our approach is guided by the following simple consideration: given any circuit $L$ in $A(n)$, the event that the innermost circuit $\gamma_n$ in $A(n)$ coincides with $L$ depends only on the edges in $A(n)$ which also lie in the interior and on $\gamma_n$. Fixing any edge $e$ on $L$ which is far from the boundary, RSW estimates imply that in several concentric annuli around $e$, there is a positive probability to find a ``detour'': an open arc lying outside $L$, but with its endpoints on $L$. Given such an arc, we can form a new open circuit in $A(n)$ by replacing a portion of $L$ by the detouring open arc. Provided the resulting curve still surrounds $B(n)$, we obtain a candidate for a circuit which could be shorter than $L$. 

Given the abundance of such detours everywhere on $L$, guaranteed by the logarithmic in $n$ number of scales, we might expect that some of them contain many fewer edges than the portion of $L$ which they circumvent, for typical values of $L$. Indeed, the innermost circuit $\gamma_n$ is constrained to remain ``as close as possible'' to the inner boundary $B(n)$ of the annulus, while the detour paths are merely required to be open. 
%(One can ensure that the detours are made of three arm points, and hence that their fractal structure on smaller scales is similar to that of $L$.) 

The idea is then to construct, for $\epsilon>0$ fixed but arbitrarily small, an open circuit $\sigma_n$, which consists of portions of the innermost open circuit $\gamma_n$ in $A(n)$, with a number of detours attached. The detours are required to have total length smaller than $\epsilon$ times that of the corresponding portions of $\gamma_n$ which they replace. If most of $\gamma_n$ can be covered by detours in this manner, one might hope that $\#\sigma_n \le (\epsilon +o(1))\#\gamma_n$ with high probability.

In trying to implement this basic strategy, we are faced with a number of problems:
\begin{enumerate}
\item Multiple detours around different edges might intersect. A systematic method is needed to keep track of how much of $\gamma_n$ we have replaced by detours.
\item We lack prior knowledge about the size of open paths in the critical cluster. It is thus not obvious that one of the many detours around each edge will have length smaller than $\epsilon$ times that of the detoured path.
\item The orientation and rough geometry of $\gamma_n$ could make it difficult to carry out the percolation estimates required to construct detours. In particular, in our argument, we do not condition on the value $L$ of $\gamma_n$ at any point.
\end{enumerate}

We address the first point by considering ``shielded'' detours: short detours which are also covered by a closed dual arc; see Definition \ref{def: shieldeddetours}. Two shielded detours are either equal or disjoint, and this allows us to estimate the total contribution of the detours to the circuit $\sigma_n$.

To address the second point, we must show that very short shielded detour paths exist with positive probability in every annulus. The only tool that we have to upper bound the length of paths is the result of Morrow and Zhang, which gives asymptotics for the length of the lowest crossing (innermost circuit). We use the fact that the fractal structure of this innermost circuit of an annulus implies that it can be made much smaller than its expected size, by forcing it to lie in a very thin region. This observation, applied to outermost partial circuits within shields, allows us to construct short detours as in Definition~\ref{def: shieldeddetours}, by constraining them to be in thin annuli. As an illustration of this idea, we give the following proposition: 
\begin{prop*}
Let $\tilde{L}_n$ be the length of the lowest horizontal crossing in $[-n,n]^2$. For any $\epsilon>0$, there is $C(\epsilon)>0$ such that
\begin{equation}
\label{eqn: aizenman-sugg}
\mathbf{P}(0<\tilde{L}_n < \epsilon \mathbf{E}\tilde{L}_n \mid  \text{there is an open crossing of } [-n,n]^2) \ge C(\epsilon),
\end{equation}
for all $n$ large enough.
\end{prop*}
\begin{proof}[Sketch of proof]
We only provide an outline of the proof here. For a more detailed argument, see the proof of Lemma \ref{lma: thin}. The size of the lowest open crossing of $[-n,n]\times [-n,-(1-\alpha)n]$ is of order $\alpha n^2\pi_3(\alpha n)$. Using quasimultiplicativity \cite[Proposition 12.2]{nolin}, and the fact that the three arm-exponent is $<1$, this is smaller than $\alpha^{1-\eta}n^2\pi_3(n)$ for some $\eta<1$. Choosing $\alpha$ small enough, the result follows.
\end{proof}
%The same technique as in the previous proof allows us to construct short shielded detours as in Definition \ref{def: shieldeddetours}.

Rather than attempting to construct detours conditioned on the innermost circuit, and showing (uniformly in this conditioning) that most of $\gamma_n$ can be covered by detours, we show in Section \ref{eqn: Svest} (see equation \eqref{eqn: decay}), that for most edges $e\in \gamma_n$, the probability of $e$ having no shielded detour around it is small, conditioned on $e$ lying on $\gamma_n$:
\begin{equation}
\label{eqn: po1}
\limsup_{n\rightarrow \infty} \mathbf{P}(\text{no detour around } e \mid e\in \gamma_n) = 0,
\end{equation}
for edges $e$ away from the boundary of $A(n)$ and $\epsilon>0$ arbitrary.

We then estimate $S_n$ by considering separately the contributions to $\# \sigma_n$ of the union $\Pi$ of all the short detours, and the edges on $\gamma_n \setminus \hat \Pi$, where $\hat Pi$ is the union of the ``detoured'' portions of the innermost circuit:
\begin{align*}
\mathbf{E}S_n &\le \mathbf{E}\#\Pi  + \mathbf{E}\#(\gamma_n\setminus \hat \Pi)\\
&\le \epsilon \mathbf{E} \#\gamma_n + \mathbf{E}\#\{e\in \gamma_n : \text{there is no detour around } e\}\\
&\le \epsilon \mathbf{E} \#\gamma_n + \sum_e\mathbf{P}(\text{no detour around } e \mid e\in \gamma_n)\mathbf{P}(e\in \gamma_n).
\end{align*}
Using \eqref{eqn: po1}, this gives
\begin{equation}
\label{eqn: expectation}
\mathbf{E}S_n \le (\epsilon+o(1))\cdot \mathbf{E}\#\gamma_n.
\end{equation}

The proof of Corollary \ref{cor: lowestcros} concerning the expected size of the lowest crossing is identical to the argument for the innermost circuit. To obtain the statement of convergence in probability in Corollary \ref{cor: probratio}, we need an additional argument. Essentially, it remains to prove that the lowest crossing of $[-n,n]^2$ cannot be smaller than $o(1)\mathbf{E}\tilde{L}_n$ with positive probability. The basic idea for our proof comes from Kesten's lower bound for the number of pivotals in a box \cite[(2.46)]{kestencrit}, but the requirement to find a large (of order $n^2\pi_3(n)$) points rather than one at each scale introduces substantial new technical difficulties. See Section \ref{sec: lowertail}.

For clarity, we have ignored the edges very close to the boundary in this rough sketch of our proof; for such edges, no estimate like \eqref{eqn: po1} holds.

To obtain the estimate \eqref{eqn: po1}, we define a sequence $E_k(e)$, $k\ge 1$ of events which depend on edges inside concentric annuli around $e$, and whose occurrence implies the existence of a shielded detour (in the sense of Definition \ref{def: shieldeddetours}) if $e\in \gamma_n$. The definition and construction of $E_k(e)$ are given in Section \ref{sec: eksec}, where it is also proved that
\begin{equation}
\label{eqn: eknaked}
\mathbf{P}(E_k(e))\ge c_1
\end{equation}
uniformly in $k\ge k_0$ for some $c_1>0$. A schematic representation of the event $E_k(e)$ appears in Figure \ref{consopic}; see also the accompanying description at the beginning of Section \ref{sec: eksec}. We use closed dual circuits with defects to force the lowest crossing to traverse certain regions inside the annulus where $E_k(e)$ is defined, regardless of the ``local orientation'' of the innermost circuit outside. To connect the innermost circuit, the detour path and its shielding closed dual path, we use five-arm points (see Section \ref{sec: shielded}), avoiding any conditioning on the realization of the lowest path.

To pass from \eqref{eqn: eknaked} to \eqref{eqn: po1}, we show in Section \ref{sec: conditioning3} that
\begin{enumerate}
\item The estimate \eqref{eqn: eknaked} remains true (with a different, but still $n$-independent constant) when we condition on $e$ lying in the innermost circuit. See \eqref{eqn: condek} in Section \ref{sec: eksec} and Proposition \ref{prop: switch}.
\item Although the $E_k(e)$'s are no longer independent under the conditional measure $\mathbf{P}(\cdot\mid e\in \gamma_n)$, the dependence is weak enough to obtain an estimate on the event that none of the $E_k$'s occur; see Proposition \ref{eqn: a3separation}. Here we use arm separation tools which appeared in \cite{DS} (which we state as Lemma \ref{lem: DS}).
\end{enumerate} 

\section{Notation and results}\label{sec: defs}
On the square lattice $(\mathbb{Z}^2,\mathcal{E}^2)$, let $\mathbf{P}$ be the critical bond percolation measure $\prod_{e \in \mathcal{E}^2} \frac{1}{2}(\delta_0 + \delta_1)$ on $\Omega = \{0,1\}^{\mathcal{E}^2}$. 

A lattice path is a sequence $v_0,e_1,v_1, \ldots, v_{N-1},e_N,v_N$ such that for all $k=1, \ldots, N$, $\|v_{k-1}-v_k\|_1=1$ and $e_k = \{v_{k-1},v_k\}$. A circuit is a path with $v_0=v_N$. For such paths we denote $\# \gamma = N$, the number of edges in $\gamma$. If $V \subset \mathbb{Z}^2$ then we say that $\gamma \in V$ if $v_k\in V$ for $k=0, \ldots, N$. 

A path $\gamma$ is said to be (vertex) self-avoiding if $v_i=v_j$ implies $i=j$ and a circuit is (vertex) self-avoiding if $v_i=v_j$ implies $i=j$ whenever $0 \notin \{i,j\}$. Given $\omega \in \Omega$, we say that $\gamma$ is open in $\omega$ if $\omega(e_k)=1$ for $k=1, \ldots, N$. Any self-avoiding circuit $\gamma$ can be viewed as a Jordan curve and therefore has an interior $\text{int }\gamma$ and exterior $\text{ext }\gamma$ (component of the complement that is unbounded). In this way, $\mathbb{Z}^2$ is the disjoint union $\text{int } \gamma  \cup \text{ext }\gamma \cup \gamma$. We say a self-avoiding circuit surrounds a vertex $v$ if $v \in \text{int }\gamma$.

The dual lattice is written $((\mathbb{Z}^2)^*,(\mathcal{E}^2)^*)$, where $(\mathbb{Z}^2)^* = \mathbb{Z}^2 + (1/2)(\mathbf{e}_1+\mathbf{e}_2)$ with its nearest-neighbor edges. Here, we have denoted by $\mathbf{e}_i$ the coordinate vectors:
\[ \mathbf{e}_1 = (1,0), \ \mathbf{e}_2 = (0,1).\]
Given $\omega \in \Omega$, we obtain $\omega^* \in \Omega^* = \{0,1\}^{(\mathcal{E}^2)^*}$ by the relation $\omega^*(e^*) = \omega(e)$, where $e^*$ is the dual edge that shares a midpoint with $e$. We blur the distinction between $\omega$ and $\omega^*$ and say, for example, that $e^*$ is open in $\omega$. For any $V \subset \mathbb{Z}^2$ we write $V^* \subset (\mathbb{Z}^2)^*$ for $V + (1/2)(\mathbf{e}_1+\mathbf{e}_2)$. For two subsets $X$ and $Y$ of the plane, we denote by dist$(X,Y)$ the Euclidean distance from $X$ to $Y$.

The symbols $C$, $c$ will denote positive constants whose value may change between occurrences, but is independent of any parameters. Dependence on parameters is indicated by an argument, as in $C(\alpha)$, and we have numbered some recurring constants using subscript for clarity.

\subsection{Circuits in annuli}

For $n \geq 1$, let $B(n)$ be the box of side-length $2n$,
\[
B(n) = \{x \in \mathbb{Z}^2 : \|x\|_\infty \leq n\} \text{ for } n \geq 1\ ,
\]
and $A(n)$ the annulus 
\[
A(n) = B(3n) \setminus B(n)\ .
\]
For $n \geq 1$, let $\partial B(n) = \{x \in \mathbb{Z}^2 : \|x\|_\infty = n\}$. 

Let $\mathcal{C}(n)$ be the collection of all self-avoiding circuits in $A(n)$ that surround the origin and, given $\omega$, let $\Xi(n) = \Xi(n)(\omega)$ be the sub-collection of $\mathcal{C}(n)$ of open circuits.

We will be interested in the event
\[
\Omega_n = \{\Xi(n) \neq \emptyset\}\ ,
\]
which we know has $0 < \inf_n \mathbf{P}(\Omega_n) \leq \sup_n \mathbf{P}(\Omega_n)<1$ by RSW arguments \cite{russo, seymourwelsh}. On $\Omega_n$ we may define $\gamma_n$, the innermost element of $\Xi(n)$, as the unique $\gamma \in \Xi(n)$ which has $\text{int }\gamma \subset \text{int }\sigma$ for all $\sigma \in \Xi(n)$. This allows us to define the random variable
\[
L_n = L_n(\omega) = \begin{cases}
\# \gamma_n &\text{ for } \omega \in \Omega_n \\
0 & \text{ for } \omega \notin \Omega_n
\end{cases}\ .
\]
This is the length of the innermost open circuit. 
%Two properties of the construction will be important for us. 
%\begin{enumerate}
%\item For $\gamma \in \mathcal{C}(n)$, the event $L(\gamma) = \{\gamma_n = \gamma\}$ depends only on edges contained in $\gamma \cup \text{int }\gamma$.
%\item For distinct $\gamma, \gamma' \in \mathcal{C}(n)$, $L(\gamma) \cap L(\gamma') = \emptyset$. 
%\end{enumerate}

The expected length of the innermost open circuit can be estimated using arm events. Let $A_3(n)$ be the ``three-arm'' event that 
\begin{enumerate}
\item The edge $(0,\mathbf{e}_1)$ is connected to $\partial B(n)$ by two open vertex disjoint paths and
\item $(1/2)(\mathbf{e}_1-\mathbf{e}_2)$ is connected to $\partial B(n)^*$ by a closed dual path.
\end{enumerate}
In later sections, we use arm events centered at vertices other than the origin. We define them now. For $v\in \mathbb{Z}^2$, $A_3(v,n)$ denotes the event that $A_3(n)$ occurs in the configuration $\omega$ shifted by $-v$. For an edge $e=(v_1,v_2)\in\mathcal{E}^2$, $A_3(e,n)$ denotes the event that
\begin{enumerate}
\item $e$ is connected to $\partial B(e,n) := \partial B((v_1+v_2)/2,n)$ by two disjoint open paths and
\item The dual edge $e^*$ is connected to $\partial B((v_1+v_2)/2,n)^*$ by a closed dual path.
\end{enumerate}
In item 2, we view the boundary as a subset of $\mathbb{R}^2$ and say that $e^*$ is connected to it if there is a closed dual path from $e^*$ which (when viewed as a subset of $\mathbb{R}^2$, touches it.

Denoting
\begin{equation}\label{eqn: a3n}
\pi_3(n) = \mathbf{P}(A_3(n)),
\end{equation}
we have the following simple adaptation of the result of Morrow and Zhang:
\begin{thm*}\label{thm: MZ}
There exist $C_1,C_2>0$ such that
\begin{equation}\label{eqn: MZestimate}
C_1 n^2 \pi_3(n) \leq \mathbf{E}L_n \leq C_2 n^2 \pi_3(n) \text{ for all } n \geq 1\ .
\end{equation}
\end{thm*}

The characterization of the innermost circuit (based on Morrow and Zhang) we will use throughout the paper is as follows. An edge $e \subset A(n)$ is in the innermost circuit if and only if the following occurs: $e^*$ is connected to $B(n)^*$ by a closed dual path, and $e$ is in an open circuit surrounding $B(n)$ such that if we remove $e$ from this circuit, then it becomes a self-avoiding path (it is no longer a closed curve). One way to say this is that $e$ has three disjoint arms (two open and one closed), with the closed arm connected to $B(n)^*$ and the open ones connecting into a circuit around $B(n)$.

We can further define the length of the shortest open circuit. That is, set
\[
S_n = S_n(\omega) = \begin{cases}
\min \{\# \gamma : \gamma \in \Xi(n)\} & \text{ for } \omega \in \Omega_n \\
0 & \text{ for } \omega \notin \Omega_n
\end{cases}\ .
\]

Our main result for circuits is
\begin{thm}\label{thm: mainthm}
As $n\to \infty$,
\begin{equation}\label{eq: to_show}
\mathbf{E}S_n = o(n^2 \pi_3(n))\ .
\end{equation}
\end{thm}
\subsection{The lowest crossing and the question of Kesten-Zhang}
The proof of Theorem \ref{thm: mainthm} applies equally well to the length $\tilde{L}_n$ of the lowest crossing of $B(n)$:
\begin{cor} \label{cor: lowestcros}
Let $\tilde{S}_n$ be the minimal number of edges in any open horizontal crossing of $B(n)$ ($\tilde{S}_n=0$ if there is no such crossing). Then
\begin{equation}
\mathbf{E}\tilde{S}_n =o(n^2\pi_3(n))\ . 
\end{equation}
\end{cor}
To address the question of Kesten and Zhang stated in the introduction and obtain the  result \eqref{eqn: probresult} on convergence in probability, we combine the preceding corollary with \eqref{eqn: MZestimate} (the version for $\tilde{L}_n$ in place of $L_n$) and an auxiliary estimate for the lower tail of $\tilde{L}_n$ (see Section \ref{sec: lowertail}). Let $H_n$ be the event that there is an open horizontal crossing of $B(n)$, and $\tilde{L}_n$ the number of edges in the lowest open crossing of $B(n)$ (with $\tilde{L}_n=0$ on $H_n^c$).
\begin{cor}\label{cor: probratio}
Conditionally on $H_n$, we have the convergence in probability:
\begin{equation}\label{eqn: probratio}
\tilde{S}_n/\tilde{L}_n \rightarrow 0\ .
\end{equation}
\end{cor}
The proof of Corollary \ref{cor: probratio} will be found in Section \ref{sec: ESn}.

\section{Short detours} \label{sec: short}
On the event $\Omega_n$, we will find another another circuit $\sigma_n \in \Xi(n)$ such that $\# \sigma_n = o(\# \hat \gamma_n)$, where $\hat{\gamma}_n$ is a truncated version of the innermost circuit $\gamma_n$ (see equation \eqref{eqn: hatdef}) in
%For $e \in \gamma_n$, we find a self-avoiding open ``detour'' path $\pi(e)$ around $e$. We will not do this for all $e$, only those 
%in 
a slightly thinned version of $A(n)$. To define this annulus, we note the following:
\begin{lma}\label{lem: AB}
For some $C_3>0$ and $C_4 \in (0,1)$, $C_3 (m/n)^{1+C_4} \leq (n/m)^2 \pi_3(m,n)$, or
\begin{equation}
\pi_3(m,n) \geq C_3 (n/m)^{C_4-1} \text{ for all } 1 \leq m \leq n\ .
\end{equation}
Here, $\pi_3(m,n)$ is the probability that there are two disjoint open paths connecting $B(m)$ to $\partial B(n)$ and one closed dual path connecting $B(m)^*$ to $\partial B(n)^*$.
\end{lma}
\begin{proof}
We use the result of Aizenman and Burchard  \cite{aizenman}.
 
Consider the event $\Lambda_1$ that 
\begin{enumerate}
\item there is a closed dual crossing of $([-n,n]\times [-n/2,-n/4])^*$, connected to the bottom of $B(n)^*$ by a closed dual crossing,
\item there is a closed dual crossing of $([-n,n]\times [n/4,n/2])^*$,
\item there is an open left-right crossing of $[-n,n]\times[-n/4,n/4]$.
\end{enumerate}
Then, by RSW, $\mathbf{P}(\Lambda_1)\ge c$ for some $c>0$. Note that on $\Lambda_1$, the lowest open crossing of $B(n)$ contains an open crossing of $B(n/2)$.

Tile the box $B(n/2)=[-n/2,n/2]^2$ by boxes of size $(1/10)m\times (1/10)m$, and let $L(m,n)$ be the number of these boxes that intersect the lowest crossing of $B(n)$.
\begin{align*}
\mathbf{E}L(m,n) &\le \sum_{B}\mathbf{P}(B\cap \tilde{L}_n\neq \emptyset)\\
&\le C\left(\frac{n}{m}\right)^2\pi_3(m,n),
\end{align*}
where the sum is over boxes $B$ of side-length $m/10$ in the tiling of $B(n/2)$.

Critical percolation in $\frac{1}{n}\mathbb{Z}^2\cap (B(1)=[-1,1]^2)$ satisfies ``Hypothesis H2'' in that paper. For $\ell>0$ and $\mathcal{C}$ a curve formed by a self-avoiding concatenation of open edges in $B(1)$, let $N(\mathcal{C},\ell)$ be the minimal number of sets of diameter $\ell$ required to cover $\mathcal{C}$. 

By \cite[Theorem 1.3]{aizenman} and \cite[Equation (1.21)]{aizenman}, there exists $C_4>0$ such that for any $\epsilon>0$:
\begin{equation}
\label{eqn: ABlwrbd}
\mathbf{P}_{\frac{1}{n}\mathbb{Z}^2\cap  B(1)}\left( \inf_{\mathrm{diam}(\mathcal{C}) \ge 1/10} N(\mathcal{C},\ell) \le C(\epsilon)\ell^{-1-C_4} \right)\le \epsilon,
\end{equation}
uniformly in $n$ sufficiently large and $\ell\ge 1/n$. Choosing $\ell=m/n$, $\epsilon$ sufficiently small and letting $\Lambda_2$ be the event that there is an open crossing of $B(n/2)$ with fewer than $C(\epsilon)n^{1+C_4}$ edges, we have by \eqref{eqn: MZestimate} and \eqref{eqn: ABlwrbd}:
\[\left(\frac{n}{m}\right)^2\pi_3(m,n) \ge (1/C)\mathbf{E}L(m,n) \ge (1/C)\mathbf{E}[L(m,n), \Lambda_1\cap \Lambda_2^c] \ge C_3(\epsilon) (n/m)^{1+C_4}\ .\]
\end{proof}
Now define the annulus
\[
\hat A(n) = B\left( \lfloor 3n - n^{C_4/2} \rfloor \right) \setminus B\left( \lceil n + n^{C_4/2} \rceil \right)
\]
and the inner portion $\hat \gamma_n$ of the innermost open circuit, defined as the union of all edges $e\in \gamma_n$ which lie entirely inside $\hat A(n)$:
\begin{equation}
\label{eqn: hatdef}
\hat \gamma_n = \{e\in \gamma_n: e\subset \hat A(n)\}.
\end{equation}

\subsection{Definition of shielded detours}\label{sec: detours}

In this section, we define the central objects of our construction, the \emph{shielded detour paths} $\pi(e)$, $e \in \hat \gamma_n$.

\begin{df}\label{def: shieldeddetours}Given $\omega \in \Omega_n$, $\epsilon \in (0,1)$, and any $e \in \hat \gamma_n$, we define the set $\mathcal{S} (e)$ of $\epsilon$-\emph{shielded detours} around $e$ as follows. An element of $\mathcal{S}(e)$ is a self-avoiding open path $P$ with vertex set $w_0, w_1, \ldots, w_M$ such that the following hold:
\begin{enumerate}
\item for $i=1, \ldots, M-1$, $w_i \in (A(n) \cap \text{ext }\gamma_n)$,
\item the edges $\{w_0,w_0+\mathbf{e}_1\}, \{w_0-\mathbf{e}_1,w_0\}, \{w_M,w_M+\mathbf{e}_1\}$ and $\{w_M-\mathbf{e}_1,w_M\}$ are in $\gamma_n$ and $w_1=w_0+\mathbf{e}_2$, $w_{M-1}=w_M+\mathbf{e}_2$.
\item writing $Q$ for the subpath of $\gamma_n$ from $w_0$ to $w_M$ that contains $e$, the circuit $Q \cup P$ does not surround the origin,
\item the points $w_0+ (1/2)(-\mathbf{e}_1+\mathbf{e}_2)$ and $w_M+(1/2)(\mathbf{e}_1+\mathbf{e}_2)$ are connected by a dual closed self-avoiding path $R$ whose first and last edges are vertical (translates of $\{0,\mathbf{e}_2\}$) and is such that the curve formed by the union of $R$, the line segments from the endpoints of $R$ to $w_1$ and $w_{M-1}$, and $P$ does not enclose the origin, and
\item $\#P \leq \epsilon \#Q$.
\end{enumerate}
Now fix a deterministic ordering of all finite lattice paths and define $\pi(e)$ to be the first element of $\mathcal{S}(e)$ in this ordering. If the set $\mathcal{S}(e)$ is empty, then we set $\pi(e)=\emptyset$.
\end{df}

\subsection{Properties of the detour paths}\label{sec: properties}
We give the properties of the collection of detours $(\pi(e) : e \in \hat \gamma_n)$ which we use in the next section to prove Theorem \ref{thm: mainthm} and Corollary \ref{cor: probratio}. The definition of $\pi(e)$ (Definition \ref{def: shieldeddetours}) appeared in Section \ref{sec: detours}. 

Let $0<\epsilon<1$. Then:
\begin{enumerate}
\item Each $\pi(e)$ is open and for distinct $e,e' \in \hat \gamma_n$, $\pi(e)$ and $\pi(e')$ are either equal or have no vertices in common.
\item If $e \in \hat \gamma_n$ and $\pi(e) \neq \emptyset$, write $\pi(e) = \{w_0,e_0, \ldots, e_{M-1},w_M\}$. Then $w_0,w_M \in \gamma_n$ but $w_i \in (A(n) \cap \text{ext } \gamma_n)$ for $i=1, \ldots, M-1$.
\item If $e \in \hat \gamma_n$ then the segment $\hat \pi(e)$ of $\gamma_n$ from $w_0$ to $w_M$ containing $e$ (that is, the ``detoured'' portion of $\gamma_n$) is such that $\hat \pi(e) \cup \pi(e)$ is a circuit that does not surround the origin. Furthermore,
\begin{equation}\label{eq: length_less}
\# \pi(e) \leq \epsilon \# \hat \pi(e)\ .
\end{equation}
\item There exists $C_5(\epsilon)>0$ such that for all $n \geq C_5$ and $e \in \hat A(n)$,
\begin{equation}
\mathbf{P}(\pi(e) = \emptyset \mid e \in \hat \gamma_n) \leq \epsilon^2 \ . \label{eqn: Svest}
\end{equation}
\end{enumerate}

We must show the above properties follow from Definition \ref{def: shieldeddetours}. Most of the work will be in showing item 4, the probability estimate \eqref{eqn: Svest}.  Items 2 and 3 hold by definition. For item 1, we have
\begin{prop}\label{prop: four}
If $\omega \in \Omega_n$, then for distinct $e,e' \in \hat \gamma_n$, $\pi(e)$ and $\pi(e')$ are either equal or have no vertices in common.
\end{prop}
The proof of Proposition \ref{prop: four} will be found in Section \ref{sec: lemmata}.

For the rest of this section, we will identify paths with their edge sets. Given such detour paths $\pi(e)$ (which necessarily do not share edges with $\gamma_n$, due to \eqref{eq: length_less}) and detoured paths $\hat \pi(e)$ (subpaths of $\gamma_n$), we construct $\sigma_n$ as follows. First choose a subcollection $\Pi$ of $\{\pi(e) : e \in \hat \gamma_n\}$ that is maximal in the following sense: for all $\pi(e),\pi(e') \in \Pi$ with $e \neq e'$, the paths $\hat \pi(e)$ and $\hat \pi(e')$ share no vertices and the total length of detoured paths $\sum_{\pi \in \Pi} \#\hat \pi$ is maximal. The choice of $\Pi$ can be arbitrary among all maximal ones. 

We put
\[\hat{\Pi} = \{\hat \pi : \pi \in \Pi\}.\] 
Now define $\sigma_n$ to be the path with edge set equal to the union of $\Pi$ and those edges in $\gamma_n$ that are not in $\hat{\Pi}$.

We must now show that if such a construction can be made, then $\limsup_n \frac{\mathbf{E}\# \sigma_n }{\mathbf{E}L_n} \leq \epsilon$. To do this we have to show two things:
\begin{lma} \label{circuit}
For $\omega \in \Omega_n,~ \sigma_n$ is an open circuit in $A(n)$ surrounding the origin.
\end{lma}
and
\begin{lma}\label{lem: piempty}
For $\omega \in \Omega_n$, if $e \in \hat \gamma_n \setminus \hat{\Pi}$ then $\pi(e) = \emptyset$.
\end{lma}
The proofs of these two lemmas is detailed in Section \ref{sec: lemmata}.

\section{Estimate for $\mathbf{E}S_n$}\label{sec: ESn}
Now we show that if paths $\pi(e)$ can be defined so as to satisfy the properties in Definition \ref{def: shieldeddetours},  and if we prove Lemmas~\ref{circuit} and \ref{lem: piempty}, we can then conclude Theorem~\ref{thm: mainthm} and Corollary~\ref{cor: lowestcros}.

Let $\ell(\Pi) = \sum_{\pi\in \Pi} \#\pi$ be the total length of the detours in the collection $\Pi$. Assuming Lemmas \ref{circuit} and \ref{lem: piempty}, we estimate the length of $\sigma_n$:

\begin{align*}
\# \sigma_n &= \ell(\Pi) + \#( \gamma_n \setminus \hat{\Pi}) \\
&\leq \ell(\Pi) + \#\{ e\in A(n): e\cap (A(n)\setminus \hat A(n)) \neq \emptyset \} + \#(\hat \gamma_n \setminus \hat{\Pi}) \\
&\leq \ell(\Pi) + 30n^{1+C_4/2} + \# \{e \in \hat \gamma_n : \pi(e)=\emptyset\}\ .
\end{align*}
We have:
\begin{equation}\label{eqn: notachance}
\ell(\Pi) = \sum_{\pi \in \Pi} \# \pi \leq  \epsilon \sum_{\pi \in \Pi} \# \hat \pi =  \epsilon \# \left( \cup_{\pi \in \Pi} \hat \pi\right) \leq \epsilon \# \gamma_n\ .
\end{equation}
Furthermore, due to \eqref{eqn: Svest},
\begin{equation}\label{eqn: bobwasright}
\mathbf{E} \# \{e \in \hat \gamma_n : \pi(e) = \emptyset\} = \sum_{e \subset \hat A(n)} \mathbf{P}(\pi(e) = \emptyset \mid e \in \hat{\gamma}_n) \mathbf{P}(e \in \hat{\gamma}_n) \leq \epsilon^2\mathbf{E}L_n\ .
\end{equation}

Therefore
\[
\mathbf{E} S_n \leq \mathbf{E} \# \sigma_n \leq (\epsilon+\epsilon^2)\cdot\mathbf{E}L_n + 30n^{1+C_4/2}\ .
\]
Since $\epsilon$ is arbitrary and $\mathbf{E}L_n \geq C_3n^{1+C_4}$, Theorem~\ref{thm: MZ} gives $\mathbf{E}S_n = o(n^2 \pi_3(n))$, finishing the proof of Theorem \ref{thm: mainthm}.

To obtain Corollary \ref{cor: probratio} for crossings of a box, we repeat the construction above to obtain a crossing $\tilde{\sigma}_n$ and a union $\tilde{\Pi}$ of detours from the lowest crossing $l_n$. Denoting by $\tilde{\mathbf{P}}$ the probability measure conditioned on the existence of an open crossing 
\[\tilde{\mathbf{P}}=\mathbf{P}(\ \cdot\mid H_n),\]
 write:
\begin{align}
\tilde{\mathbf{P}}(\# \tilde{\sigma}_n > 3\epsilon^{1/2} \tilde L_n) &\le \tilde{\mathbf{P}}(\ell(\tilde{\Pi}) >\epsilon^{1/2} \tilde L_n) + \tilde{\mathbf{P}}(30n^{1+C_4/2} > \epsilon^{1/2}\tilde L_n) \label{eqn: clausius}\\
&\quad + \tilde{\mathbf{P}}(\# \{e\in \hat{l}_n:\pi(e)=\emptyset\} > \epsilon^{1/2} \# l_n ) \nonumber
\end{align}
By an estimate analogous to \eqref{eqn: notachance}, the first probability on the right is zero. We decompose the last term further:
\begin{align}
 \tilde{\mathbf{P}}(\# \{e\in \hat{l}_n:\pi(e)=\emptyset\} > \epsilon^{1/2} \tilde L_n ) &\le \tilde{\mathbf{P}}(\# \{e\in \hat{l}_n:\pi(e)=\emptyset\} > \epsilon \mathbf{E}\tilde{L}_n) \nonumber \\
&\quad +\tilde{\mathbf{P}}(\tilde L_n \le (\epsilon/\epsilon^{1/2})\mathbf{E}\tilde{L}_n).\label{eqn: lordkelvin}
\end{align}
By Markov's inequality and \eqref{eqn: bobwasright}, the term $\tilde{\mathbf{P}}(\# \{e\in \hat{l}_n:\pi(e)=\emptyset\} > \epsilon\mathbf{E}\tilde{L}_n)$ is bounded by $\epsilon$. (Recall that $0<\mathbf{P}(H_n)<1$ uniformly in $n$, so $\tilde{\mathbf{P}}$ is uniformly absolutely continuous with respect to $\mathbf{P}$.)

It remains to estimate the second term on the right in \eqref{eqn: clausius} and the last term on the right in \eqref{eqn: lordkelvin}. Using again that 
$\mathbf{E}\tilde{L}_n\ge C_3 n^{1+C_4}$, we see that it will suffice to show that
\[
\limsup_{n\rightarrow \infty} \tilde{\mathbf{P}}(0<\tilde{L}_n \le \epsilon^{1/2}\mathbf{E}\tilde{L}_n)\le \lambda(\epsilon)
\]
for $\lambda(\epsilon)\rightarrow 0$ as $\epsilon\rightarrow 0$.
Recalling that $\tilde{\mathbf{P}}$ is supported on the event $\{\tilde{L}_n >0\}$ and $\mathbf{P}(H_n)>0$ uniformly in $n$, this reduces to showing
\begin{equation}\label{eqn: timetogo}
\lim_{\epsilon \downarrow 0} \limsup_{n\rightarrow \infty} \mathbf{P}(0<\tilde{L}_n \le \epsilon^{1/2}\mathbf{E}\tilde{L}_n)=0.
\end{equation}
This will be done in the Section \ref{sec: lowertail}.

\section{The events $E_k$}\label{sec: eksec}
In this section, we define events $E_k$ (depending on $\epsilon$), which will be used in the proof of the probability estimate \eqref{eqn: Svest} (see \eqref{eqn: dodo} in Section \ref{sec: conditioning3}). We show that, for some $K(\eta)$,
\begin{equation}
\mathbf{P}(E_k\mid A_3(d))>C_{7} \text{ for all } k \in \{K(\eta), \ldots, \lfloor (C_4/8)\log n\rfloor\} \text{ and } d \geq \lfloor (C_4/8)\log n\rfloor \label{eqn: Eklower}
\end{equation}
for some uniform constant $C$ independent of $n$. $E_k$ will be defined as the conjunction of a large number of crossing events (see equation \eqref{eqn: ekdef}). These events will be gradually introduced over the next few subsections. Figure \ref{consopic} illustrates most of the crossing events in $E_k$. 

The essential property of $E_k$ is 
\begin{prop}\label{prop: whyek}
Let $e$ be an edge of $\mathbb{Z}^2$ with $\epsilon>0$, and let $E_k(e)=\tau_{-e}E_k$ be the translation of $E_k$ by the edge $e$. That is, for any $\omega\in \Omega$,
\[(\omega_{e'})_{{e'}\in\mathcal{E}^2} \in \tau_{-e}E_k \iff (\omega_{e'-e})_{e'\in\mathcal{E}^2} \in E_k.\]

There is a constant $K(\epsilon)$ such that the following holds. Let 
\[k  \in \{K(\epsilon), \ldots, \lfloor (C_4/8) \log n \rfloor\}.\] 
On $E_k(e)\cap \{e\in \hat{\gamma}_n \}$, there is a short detour $\pi(e) \in \mathcal{S}(e)$ around $\hat{\gamma}_n$ contained in $B(3^{k+1}) \setminus B(3^{k-1})$. That is, 
\[
\mathcal{S}(e)\neq \emptyset.
\]
\end{prop}
This proposition will be proved in Section \ref{sec: essential}. After we prove this, Lemmas~\ref{circuit} and \ref{lem: piempty}, and Proposition~\ref{prop: four}, we can conclude Theorem~\ref{thm: mainthm}.

\subsection{Sketch of proof of Proposition~\ref{prop: whyek}}
Because the proof of the above proposition requires many constructions, we now give a sketch of the main ideas. Let $\eta \in (1,11/10)$ be given and $k \geq 1$. The quantity $\delta=\eta-1>0$ should be thought of as small. 
Define the annuli 
\[
Ann_i = Ann_i(\eta) = B(\eta 3^{k+i-2}) \setminus B(3^{k+i-2}) \text{ for } i=1,2.
\]
The event $E_k$ has three main features:
\begin{enumerate}
\item Two closed circuits (in green in Figure \ref{consopic}), in $Ann_1$ and $Ann_2$, each with two defects, in thin concentric annuli. These serve to isolate the inside of the annulus from the rest of $\hat{\gamma}_n$. Their thickness is controlled by the small parameter $\delta \ll 1$.  If the origin has 3 arms to a macroscopic distance, the open arms are forced to pass through the defects. 

The existence of these circuits is shown to have probability bounded below independently of $n$ in Section \ref{sec: shielded}.
 
\item An open half-circuit connected to the crossings of the annulus emanating from the origin (in red in Figure \ref{consopic}). This will act as a detour for the portion of $\hat{\gamma}_n$ inside the larger box. 

In Section \ref{sec: thin}, we show that, given the existence of the circuits in the previous item, the open half-circuit contains at most  $\epsilon 3^{2k}\pi(3^k)$ edges with positive probability, where $\epsilon>0$ is small.

\item Boxes containing a sizable number of three-arm points connected to the arms emanating from $e$. (The relevant connections appear in blue in Figure \ref{consopic}.) On $E_k(e)\cap \{e \in \hat{\gamma}_n\}$, these lie on the $\hat{\gamma}_n$.

We give a lower bound of order $C3^{2k}\pi_3(3^k)$ for the number of three arm points on the open arms emanating from the origin. This holds with probability bounded below independently of $n$, conditionally on the events in the previous items. This is done in Section \ref{sec: lowerbd}.
\end{enumerate}

Given these points, the proof proceeds as follows. We first show existence of circuits with defects: we show that for any given $\eta$ close enough to 1, there is a constant $D_1 = D_1(\eta)$ such that for suitable values of $k$, one has
\begin{equation}\label{eq: X_1}
\mathbf{P}(X_1(k,\eta)) \geq D_1(\eta),
\end{equation}
where $X_1(k,\eta)$ is the event that the closed circuits with defects from item one above exist in the annuli $Ann_i$. Next, we give an upper bound for the length of a thin detour. Namely, let $X_2(k,\eta, \epsilon)$ be the event that there is an open half-circuit connected to the defects from item one, staying in $Ann_2$, and having length at most $\epsilon (3^k \pi_3(3^k))^2$. We show in Section~\ref{sec: thin} that for any $\epsilon>0$, there exists $\eta(\epsilon)$ close enough to 1 such that 
\begin{equation}\label{eq: X_2}
\mathbf{P}(X_2(k,\eta,\epsilon) \mid X_1(k,\eta)) \geq D_2(\epsilon)
\end{equation}
for suitable values of $k$. Last, we show existence of many edges on paths that will function as the innermost circuit. Let $X_3(k,\eta,c)$ be the event that in boxes in the interior of the annulus, there are at least $c(3^k \pi_3(3^k))^2$ edges with three disjoint arms: two open to the defects and one closed dual path to the bottom of the annulus. In Section~\ref{sec: lowerbd}, we prove that there exists $c>0$ and a constant $D_3>0$ such that for all $\eta$ close to 1 and suitable values of $k$,
\begin{equation}\label{eq: X_3}
\mathbf{P}(X_3(k,\eta,c) \mid X_2(k,\eta(\epsilon),\epsilon), X_1(k,\eta)) \geq D_3.
\end{equation}
The most important thing here is that $c$ has no dependence on $\eta$ or $\epsilon$, essentially because as $\eta \downarrow 1$, the size of the boxes in which the three arm points lie does not decrease to 0.

To put these pieces together, we first choose $c$ such that \eqref{eq: X_3} holds. Next, given $\epsilon>0$, choose $\eta = \eta(c\epsilon)$ to guarantee \eqref{eq: X_2} with $c\epsilon$ in place of $\epsilon$. For this value of $\eta$, one also has \eqref{eq: X_1}. Combining the above three inequalities, and putting
\[
E_k = X_1(k,\eta) \cap X_2(k,\eta,\epsilon) \cap X_3(k,\eta,c),
\]
one has
\begin{equation}\label{eq: E_k_fake_bound}
\mathbf{P}(E_k) \geq D_1D_2D_3
\end{equation}
for suitable $k$. On this intersection, one is guaranteed that the volume of the detour is at most $\epsilon$ times the volume of the three-arm points in the boxes connected to the defects. To finish the proof, one notes that $E_k$ implies that there is a three-arm connection across the annulus in which $E_k$ occurs. Thus one uses arm separation to see that if $E_k$ and $A_3(d)$ both occur, then one can, with uniformly positive conditional probability, route the arms from $A_3(d)$ to the inner and outer boundaries of the annulus to connect to the the arms from $E_k$. Thus there is a constant $D_4$ independent of $n, \eta, c$ and $k$ such that
\[
(1/D_4) \mathbf{P}(E_k) \mathbf{P}(A_3(d)) \leq \mathbf{P}(E_k,A_3(d)) \leq D_4 \mathbf{P}(E_k) \mathbf{P}(A_3(d)).
\]
This step is standard, so we leave the details to the reader. Combining it with \eqref{eq: E_k_fake_bound} completes the proof.

\begin{figure}
\centering
\includegraphics[scale = 0.6]{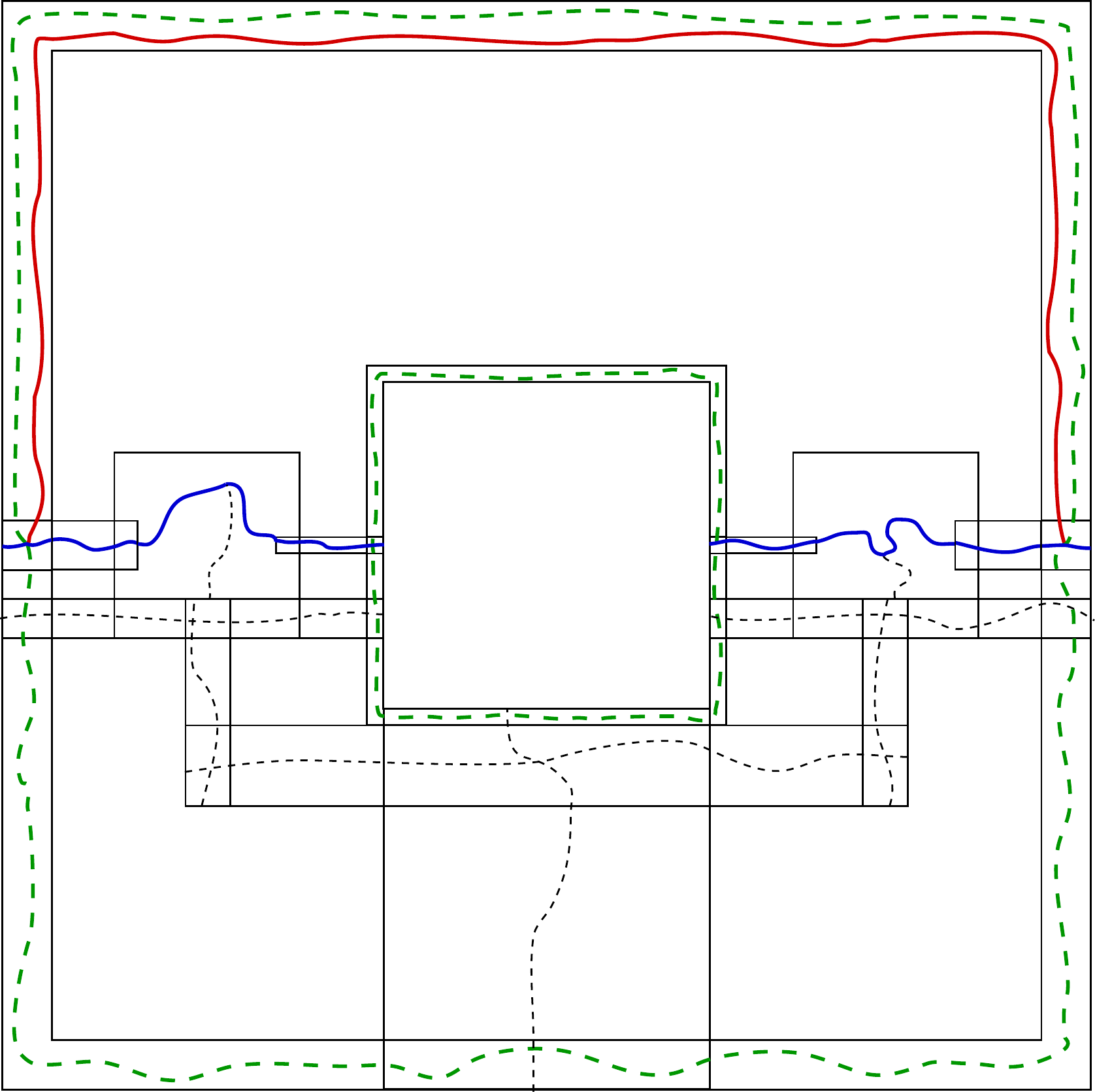}
\caption{An approximate depiction of the event $E_k$. Not all connections are shown.}
\label{consopic}
\end{figure}
\subsection{Five-arm events and shielded circuit}\label{sec: shielded}
In this section we put in place most of the components of the construction of the event $E_k$. Define the boxes 
\begin{align*}
B_1 &= [-\eta 3^k,-3^k] \times [-\frac{3^{k}}{2}(\eta-1),\frac{3^{k}}{2}(\eta-1)],~\\
B_2 &= [-\eta 3^{k-1},-3^{k-1}] \times [-\frac{3^{k-1}}{2}(\eta-1),\frac{3^{k-1}}{2}(\eta-1)],\\
B_3 &= B_1+(3^k(\eta-1),0), \quad B_4 = B_2 +(-3^{k-1}(\eta-1),0),\\
%B_5 &= B_1 + (0,-3^k(\eta-1)),
\end{align*}
and the ``long'' rectangles:
\begin{align*}
B_5 &= [-\eta 3^k, -\frac{3^k(1+\eta)}{2}] \times [\frac{3^k}{2}(\eta-1), 3^k\eta],~\\
B_6 &= [-\frac{3^k(1+\eta)}{2}, -3^{k}] \times [\frac{3^{k}}{2}(\eta-1), \frac{3^k}{2}(1+\eta)],\\
B_7 &= [-\eta 3^k, -3^k]\times [-\eta 3^k,-\frac{3^k}{2}(\eta-1)],\\
B_8&= [-\eta 3^{k-1}, -3^{k-1}]\times [-\eta 3^{k-1},-\frac{3^{k-1}}{2}(\eta-1)],
\end{align*}
\begin{align*}
B_9 &= [-\eta 3^{k-1},-3^{k-1}]\times [\frac{3^{k-1}}{2}(\eta-1),\eta 3^{k-1}], & B_{10}&= [-\eta 3^{k-1},\eta 3^{k-1}]\times [3^{k-1},\eta 3^{k-1}],\\
B_{11} &= [-\eta 3^{k-1},\eta 3^{k-1}]\times [-\eta 3^{k-1},-3^{k-1}], & B_{12} &= [-\eta 3^k, \eta 3^k]\times [\frac{3^k}{2}(1+\eta), \eta 3^k ],\\
B_{13} &= [-\frac{3^k}{2}(1+\eta),\frac{3^k}{2}(1+\eta)]\times [3^k, \frac{3^k}{2}(1+\eta)], & B_{14}&= [-\eta3^k, \eta 3^k]\times [-\eta 3^k, -3^k].
\end{align*}
The relative placement of these boxes inside $S(e,\eta 3^k)$ is shown in Figure \ref{consolidation-pic2}. From this point on, we will restrict to $n$ and $k$ such that 
\begin{equation}\label{eq: n_k_choice}
n \geq 1 \text{ and } k \in \{K(\eta), \ldots, \lfloor (C_4)/8 \log n \rfloor\}\ ,
\end{equation}
where $K(\eta)$ is chosen so that all boxes involved have lengths at least some constant, say 10. (This includes the above boxes, but also those used in Section~\ref{sec: lowerbd}.) If we decrease $\eta$, then the range of valid $k$ decreases.

\begin{figure}
\centering
\includegraphics[scale = 0.55]{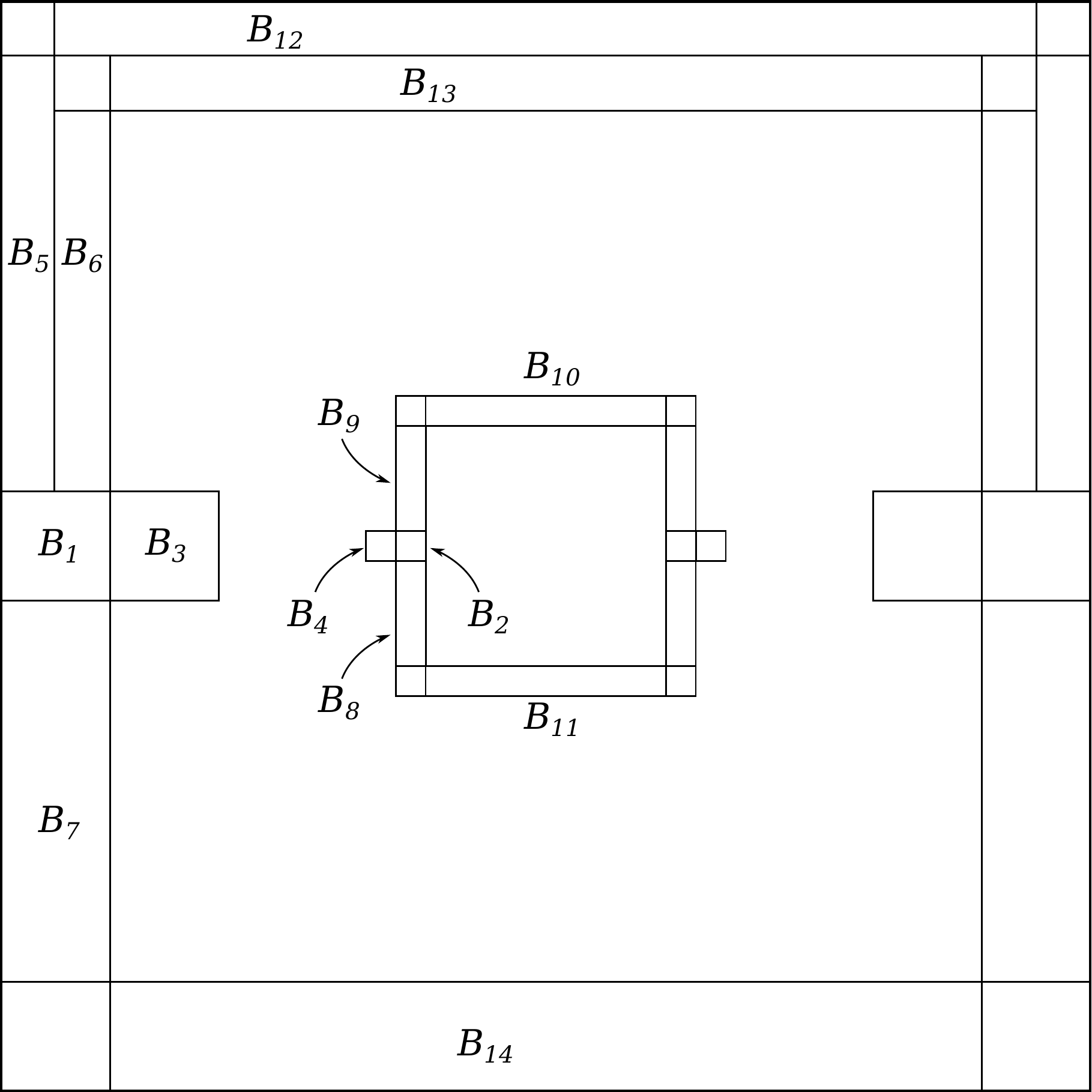}
\caption{The boxes $B_1,\ldots, B_{14}$. This figure is not to scale; only the relative placement of the boxes is illustrated. In particular, in our application, $\eta$ is smaller relative to the ratio of the sizes of the inner to outer annuli.}
\label{consolidation-pic2}
\end{figure}

The most important definition of this section is the following
\begin{df}[Five-arm event] 
\label{def: fivearm}
$M_1 =M_1(k)$ is the event: there is a \emph{five-arm point} $w\in \mathbb{Z}^2$ in the box $B_1$. That is,
\begin{enumerate}
\item The edge $\{w+(1/2)(\mathbf{e}_2-\mathbf{e}_1), w+(1/2)(3\mathbf{e}_2-\mathbf{e}_1)\}$ is closed and has a closed arm $\gamma_1$ to 
\[I_1 = [-\eta 3^k, -\frac{3^k}{2}(1+\eta)]\times \{\frac{3^k}{2}(\eta-1)\},\]
\item $\{w,w+\mathbf{e}_2\}$ has an open arm $\gamma_2$ to 
\[I_2 = [-\frac{3^k}{2}(1+\eta),-3^k]\times \{\frac{3^k}{2}(\eta-1)\},\]
\item $\{w,w+\mathbf{e}_1\}$ has an open arm $\gamma_3$ to the right side of $B_1$,
\item $\{w-(1/2)(\mathbf{e}_1+\mathbf{e}_2),w-(1/2)(3\mathbf{e}_2+\mathbf{e}_1)\}$ has a closed arm $\gamma_4$ to the bottom of $B_1$, and 
\item $\{w-\mathbf{e}_1,w\}$ has an open arm $\gamma_5$ to the left side of $B_1$.
\end{enumerate}
The event $M_1$ is illustrated in Figure \ref{consolidation-pic3}.
\end{df}

\begin{figure}
\centering
\includegraphics[scale = 0.35]{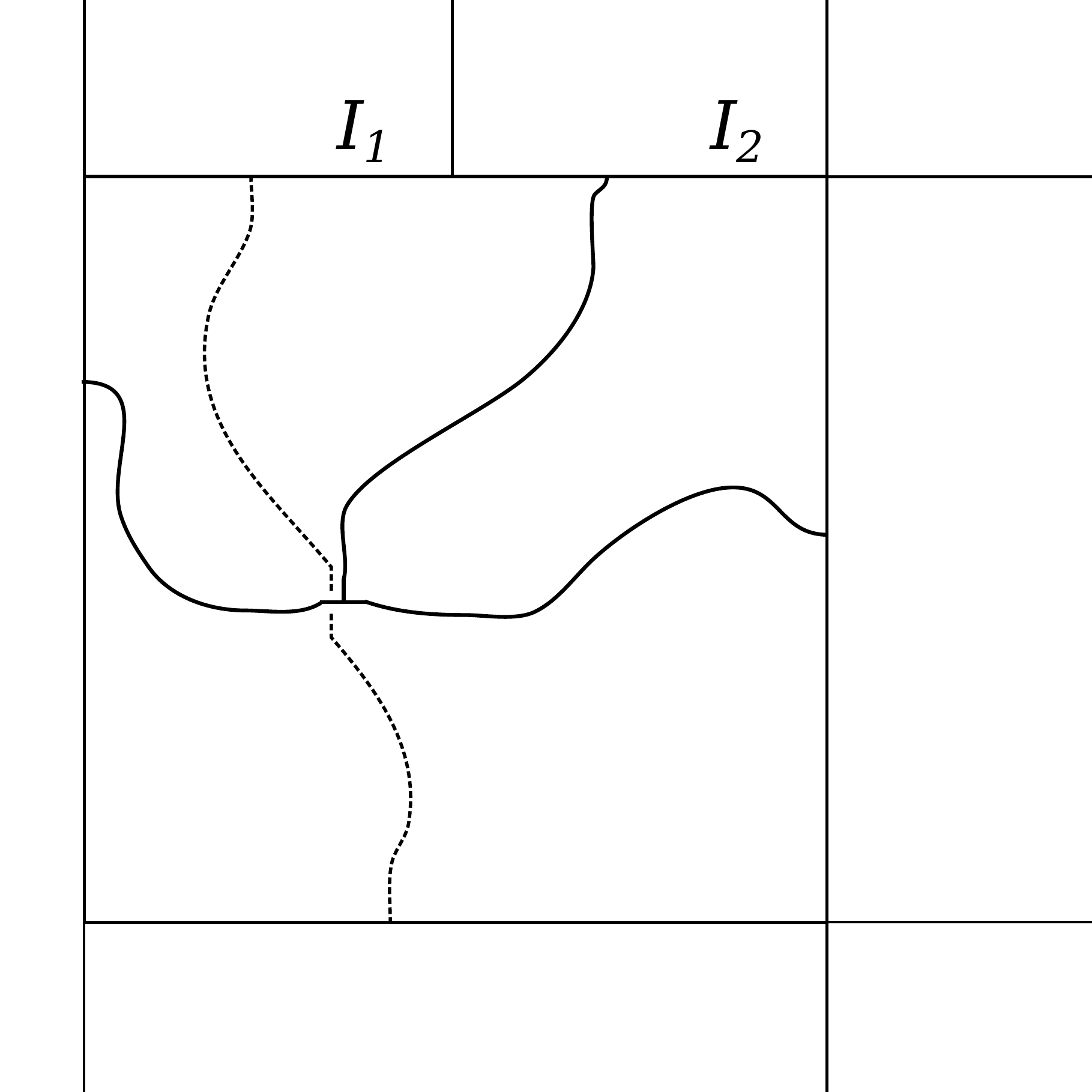}
\caption{The ``five-arm'' event $M_1(k)$}
\label{consolidation-pic3}
\end{figure}

$M_2=M_2(k)$ is the event that there is a four-arm point $z$ in the box $B_2$ with a two open arms, one to each horizontal side of $B_2$, and two closed arms to the top and bottom of $B_2$.

Our first claim is 
\begin{prop}
There is a constant $C$ independent of $k$ such that 
\[\min\left\{ \mathbf{P}(M_1), \mathbf{P}(M_2) \right\} \ge C\]
\begin{proof}
Let $Z_1$ be the number of vertices $w$ in $B_1$ satisfying the conditions in the definition of $M_1$. Then
\[\mathbf{P}(Z_1 >0) \ge \mathbf{P}(\hat Z_1 >0),\]
where $\hat Z_1=\# \hat W_1$ is the number of $w$  with five arms to $\partial B_1$ as in the definition of $M_1$, inside the box $\hat B_1$ with half the side length of $B_1$,  and centered at the same point.
By arms separation arguments, we have
\[\mathbf{E}\hat Z_1 = \sum_{w\in \hat B_1}\mathbf{P}(w \in \hat W_1) \asymp \sum_{w\in \hat B_1}\mathbf{P}\left(A_5\left(\frac{3^k}{2}(\eta-1)\right)\right),\]
where $A_5(3^k(\eta-1)/2)$ is the event that $0$ has 5 arms to distance $\frac{3^k}{2}(\eta-1)$ (with no further conditions on the arms except the ``color sequence'' -- their open and closed statuses -- which is open, open, closed, open, closed). Here $\asymp$ means the ratio of the left and right sides is bounded away from 0 and $\infty$. The 5-arm exponent is universal and equal to $2$ \cite[Lemma 5]{KSZ}, \cite[Theorem 24, 3.]{nolin}, so
\[\sum_{w\in \hat B_1}\mathbf{P}\left(A_5\left(\frac{3^k}{2}(\eta-1)\right)\right) \asymp C>0.\]
On the other hand, by planarity, there can be at most one point in $\hat W_1$, so
\[\mathbf{P}(M_1)\ge \mathbf{P}(\hat Z_1 >0) = \mathbf{E}\hat Z_1 > C.\]
The argument for $M_2$ is similar, noting that the existence of a 5-arm point implies in particular that of a 4-arm point.
\end{proof}
\end{prop}
Let $M_3=M_3(k)$ be the event that there is a closed top-down crossing of $B_5$, and an open top-down crossing of $B_6$. $M_4(k)$ is defined to be the event that there are closed top-bottom crossings of $B_7$, $B_8$ and $B_9$ and open left-right and top-down crossings of $B_3$ and $B_4$. By Russo-Seymour-Welsh, we have 
\[
\mathbf{P}(M_3), \mathbf{P}(M_4)\ge C(\eta)>0.
\]

We let $G_1(k)$ be the event that
\begin{enumerate}
\item $M_1$, $M_2$, $M_3$ and $M_4$ occur.
\item The closed arm $\gamma_1$ is connected to the crossing of $B_5$.
\item The open arm $\gamma_2$ is connected to the crossing of $B_6$.
\item The open arm $\gamma_3$ is connected to the crossings of $B_3$.
\item The closed arm $\gamma_4$ is connected to the crossing of $B_7$.
\item The four arm-point in $B_2$ is connected to an open crossing of $B_4$.
\end{enumerate}

Here and elsewhere in the paper, we will make extensive use of the following generalized FKG inequality \cite[Lemma 3]{kestenscaling} (see also \cite[Lemma 13]{nolin}):
\begin{lma}[Generalized FKG inequality]\label{lma: fkg}
Let $A$ and $D$ be increasing events, and $B$ and $E$ decreasing events. Assume that $A$, $B$ $D$, $E$ depend only on edges in the finite sets $\mathcal{A}$, $\mathcal{B}$, $\mathcal{D}$, and $\mathcal{E}$, respectively. If
\[\mathcal{A}\cap \mathcal{B} =\mathcal{A}\cap \mathcal{E}=\mathcal{B}\cap \mathcal{D}=\emptyset,\]
then 
\[\mathbf{P}(A\cap B\mid D\cap E) \ge \mathbf{P}(A\cap B).\]
\end{lma}

By generalized FKG and standard gluing constructions, we have
\[\mathbf{P}(G_1)\ge C_9\mathbf{P}(M_3)\mathbf{P}(M_4),\]
for some constant $C_9$ independent of $\eta$. 

%Moreover, there are constants $C_{31}(\eta)$, $C_{32}(\eta)>0$ such that
%\begin{align*}
%\mathbf{P}(M_3)&\ge C_{31}(\eta),\\
%\mathbf{P}(M_4)&\ge C_{32}(\eta).
%\end{align*}

$G_2(k)$ is defined to be the reflection of $G_1$ about the vertical axis through $0$. For a box $B_i$, $i=1,\ldots,9$, we let $B'_i$ be its reflection about the $\mathbf{e}_2$-axis. That is, if $x=(x_1,x_2) \in \mathbb{Z}^2$, then $x\in B_i'$ if and only if $(-x_1,x_2)\in B_i$. The same applies to the ``landing zones'' $I_i$. We say that $G_2(k)$ occurs if all the conditions in the definition of $G_1(k)$ occur, replacing each $B_i$, $i=1,\ldots 9$, and $I_1$, $I_2$ by $B'_i$ and $I_1'$, $I_2'$, respectively. By symmetry and independence:
\[\mathbf{P}(G_1(k)\cap G_2(k))\ge C_{9}^2 \mathbf{P}(M_3(k))^2(\mathbf{P}(M_4(k))^2.\]

We let $R_1(k)$ be the event that there is a closed left-right crossing of $B_{12}$; $R_2(k)$ is the event that there is an open left-right crossing of $B_{13}$; $R_3(k)$ is the event that there are closed left-right crossings of $B_{10}$, $B_{11}$ and $B_{14}$.

By Russo-Seymour-Welsh and generalized FKG, we have
\begin{align*}
&\mathbf{P}(G_1(k)\cap G_2(k)\cap R_1(k)\cap R_2(k)\cap R_3(k)) \\
\ge~& C_{9}^2 \mathbf{P}(M_3(k))^2\mathbf{P}(M_4(k))^2\mathbf{P}(R_1(k))\mathbf{P}(R_2(k))\mathbf{P}(R_3(k)).
\end{align*}

The occurrence of the intersection 
\begin{equation}
\label{eqn: q1def}
Q_1(k,\eta) =  G_1(k)\cap G_2(k)\cap R_1(k)\cap R_2(k)\cap R_3(k)
\end{equation}
implies the events:
\begin{enumerate}
\item There are five arm points $w$ and $w'$ in $B_1$ and $B_1'$.
\item There is a closed dual circuit $\alpha_1$ with 2 defects near $w$ and $w'$ in $Ann_2$. The arc $\tilde{\alpha}_1$ of $\alpha_1$ between $w$ and $w'$ is contained in $\Gamma(k,\eta)=Ann_2 \cap\left( \mathbb{R}\times [-\frac{3^k}{2}(\eta-1),\infty)\right)$.
\item There is an open arc $\alpha_2$ contained in $\Gamma(k,\eta)$ with endpoints at $w$ and $w'$. Moreover, $\alpha_2$ is contained in the interior of the dual circuit $\alpha_1$.
\end{enumerate}
We note that there exists $C_{10}(\eta)>0$ depending on $\eta$ such that
\begin{equation}\label{eq: Q_1_bound}
\mathbf{P}(Q_1(k,\eta)) \geq C_{10} \text{ for } k \geq K(\eta)\ .
\end{equation}

\begin{df}[Outermost open path] \label{def: outermost} Given the occurrence of $Q_1(k,\eta)$, we can define the \emph{outermost open arc} $\tilde{\alpha}_2$ contained inside the dual circuit $\alpha_1$. For this, we let $\tilde{\alpha}_1$ denote the portion of $\alpha_1$ between $w$ and $w'$ in $\Gamma(k,\eta)$. The outermost arc $\tilde{\alpha}_2$ is the open arc in $\Gamma(k,\eta)$ with endpoints at $w$ and $w'$ such that the region enclosed by the Jordan curve $\tilde{\alpha}_1 \cup \tilde{\alpha}_2$ (extended near the five-arm points to be a closed curve) is minimal. 
\end{df}

In the above definition, we may choose $\alpha_1$ arbitrarily, but this choice uniquely defines $\alpha_2$. We have the following
\begin{lma}
On the event $Q_1(k,\eta)$, any edge $e$ on $\tilde{\alpha}_2$ has 3 disjoint arms: two open and one closed. These arms reach to distance at least $3^k(\eta-1)$ from $e$. 
\begin{proof}
If $e$ is on $\tilde{\alpha}_2$, then necessarily, there is a closed path from $e^*$ to $\tilde{\alpha}_1$ contained in the region $\mathrm{int}(\tilde{\alpha}_1\cup\tilde{\alpha}_2)$. Following this closed path until we reach the closed path $\tilde{\alpha}_1$, we can extend it into a closed path of the required length, because the path $\alpha_1$ has diameter greater than $3^k$. On the other hand, the path $\tilde{\alpha}_2$ itself has two ends, one of which necessarily has length $3^k/2$. As for the other, it emanates from a five-arm point, itself connected to a crossing of a box of width $3^k(\eta-1)$.
\end{proof} 
\end{lma}

\subsection{Upper bound for the volume of the thin detour}
\label{sec: thin}
The main estimate of this section is the following:
\begin{lma}\label{lma: thin}
Let $\epsilon>0$, and define the events
\begin{equation}
Q_1(k) = Q_1(k,\eta) =  G_1(k)\cap G_2(k)\cap R_1(k)\cap R_2(k)\cap R_3(k)
\end{equation}
and 
\begin{equation}
\label{eqn: q2def}
Q_2=Q_2(k,\eta,\epsilon) =\{ \# \tilde{\alpha}_2\le \epsilon 3^{2k}\pi_3(3^k) \}.
\end{equation}
There exists $\delta=\eta-1$ small enough and $C_{11}>0$ such that
\begin{equation}\label{eqn: q2ineq}
\mathbf{P}(Q_2(k,\eta,\epsilon) \mid Q_1(k))\ge C_{11}>0,
\end{equation}
for all $n$ and all $k \in \{K(\eta), \ldots, \lfloor (C_4/8)\log n \rfloor\}$.
\begin{proof}
The key estimate we need is: for each $e \subset \Gamma(k,\eta)$
\[\mathbf{P}(e\in \tilde{\alpha}_2 \mid Q_1(k)) \le C\pi_3(3^k(\eta-1)).\]
We rewrite the probability as
\begin{align*}
&\frac{\mathbf{P}(e\in \tilde{\alpha}_2, Q_1(k))}{\mathbf{P}(Q_1(k))}\\
\le~& \frac{1}{C_{9}^2}\mathbf{P}(M_3(k))^{-2}\mathbf{P}(M_4(k))^{-2}\mathbf{P}(R_1(k))^{-1}\mathbf{P}(R_2(k))^{-1} \mathbf{P}(R_3(k))^{-1} \mathbf{P}(e\in \tilde{\alpha}_2,Q_1(k)).
\end{align*}

Let \[B(e,3^k(\eta-1)) =\{v: \|(v_1+v_2)/2-v\|_{\infty}\le 3^k(\eta-1)\},\]
where $e = \{v_1,v_2\}$. Also let $A_\eta(e)$ be the event that $e$ has 3 arms to the boundary of $B(e,3^k(\eta-1))$, two open and one closed. 
Recalling the definition of $Q_1$ \eqref{eqn: q1def}, we have
\[\mathbf{P}(e\in \tilde{\alpha}_2 , Q_1(k))\le \mathbf{P}(A_\eta(e), M_3\cap M_3' \cap M_4 \cap M_4' \cap R_1\cap R_2 \cap R_3).\]

We now define crossing events in ``truncated'' regions, which depend on the position of $e$: $\tilde{M}_3$ is the event
\begin{enumerate}
\item There are vertical closed crossings of each component of
\[[-\eta 3^k,-\frac{3^k}{2}(\eta+1)]\times [\frac{3^k}{2}(\eta-1),3^k(2-\eta)] \setminus B(e,3^k(\eta-1)).\] 
\item There are vertical open crossings of each component of
\[[-\frac{3^k}{2}(\eta+1),-3^k]\times [\frac{3^k}{2}(\eta-1), 3^k(2-\eta)] \setminus B(e,3^k(\eta-1)).\]
\end{enumerate}
$\tilde{M}_4$ is the event that there are closed top to bottom crossings of $B_8$ and $B_9$, open left-right crossings of $B_3$ and $B_4$, and a closed vertical crossing of each component of 
\[[-\eta 3^k, -3^k]\times[-3^k,-\frac{3^k}{2}(\eta-1)]\setminus B(e,3^k(\eta-1)).\]

$\tilde{M}_3'$, $\tilde{M}_4'$ are the reflections of the events $\tilde{M}_3$ and $\tilde{M}_4$ about the $y$-axis. 
$\tilde{R}_1$ is the event that there are closed left-right crossings of each component of
\[[-3^k(2-\eta), 3^k(2-\eta)]\times [\frac{3^k}{2}(\eta+1),\eta 3^k]\setminus B(e,3^k(\eta-1)),\] 
and similarly $\tilde{R}_2$ is the event that there are open left-right crossings of each component of
\[ [-3^k(2-\eta), 3^k(2-\eta)]\times [3^k,\frac{3^k}{2}(\eta+1)] \setminus B(e,3^k(\eta-1)).\]
The definition of the events implies that the truncated regions considered are either rectangles, or consist of the union of two disjoint rectangles which also do not abut (see Figure \ref{fig3} -- this is the reason for the choice of slightly smaller rectangles with bound $3^k(2-\eta)$), which implies:
\begin{align*}
M_3 &\subset \tilde{M}_3,\ & M_3'\subset \tilde{M}_3',\\
M_4 &\subset \tilde{M}_4,\ & M_4'\subset \tilde{M}_4',\\
R_1 & \subset \tilde{R}_1,\ &  R_2\subset \tilde{R}_2.
\end{align*}

\begin{figure}
\centering
\includegraphics[scale = 0.7]{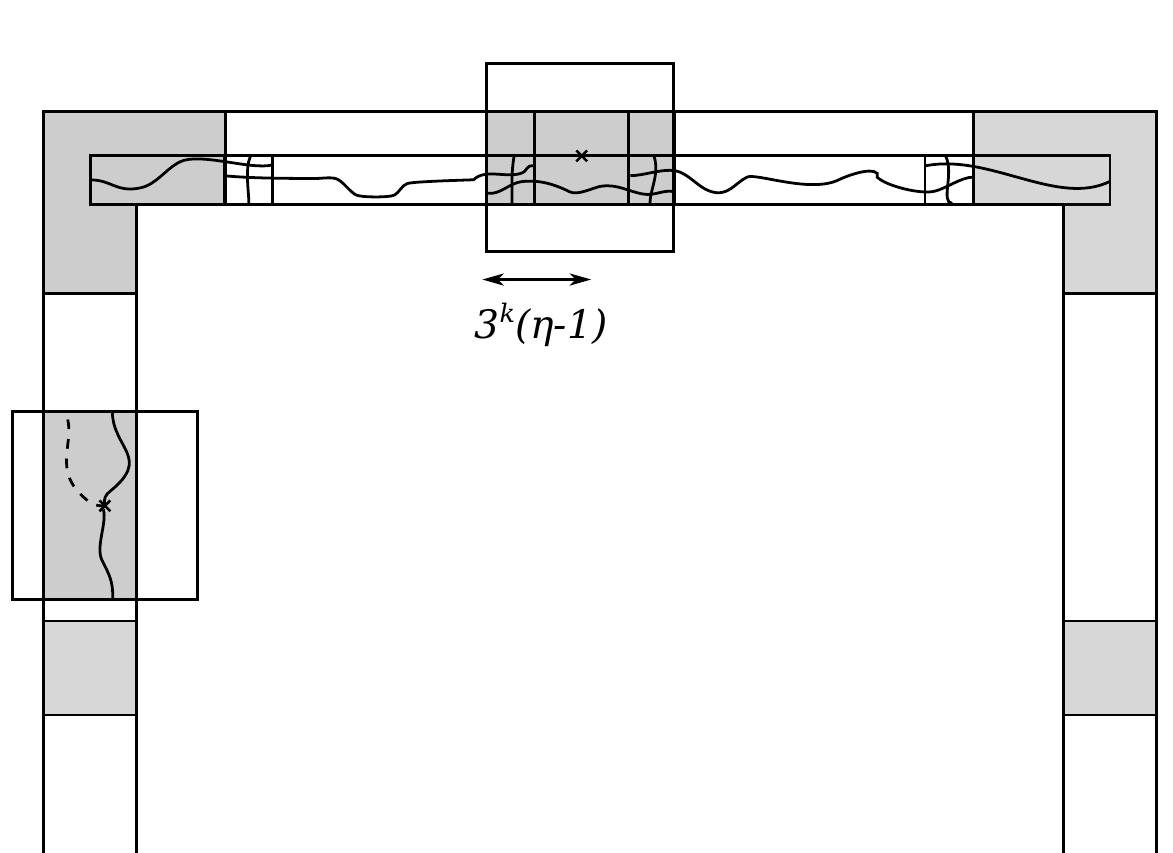}
\caption{An illustration of the truncated regions for two possible placements of the edge $e$.}
\label{fig3}
\end{figure}

Then:
\begin{align*}
&\mathbf{P}(A_\eta(e), M_3\cap M_3' \cap M_4 \cap M_4' \cap R_1\cap R_2 \cap R_3) \\
\le~& \mathbf{P}(A_\eta(e), \tilde{M}_3\cap \tilde{M}_3'\cap \tilde{M}_4\cap \tilde{M}_4' \cap \tilde{R}_1\cap \tilde{R}_2\cap R_3)\\
=~& \mathbf{P}(A_\eta(e))\mathbf{P}(\tilde{M}_2)^2\mathbf{P}(\tilde{M}_3)^2\mathbf{P}(\tilde{R}_1)\mathbf{P}(\tilde{R}_2)\mathbf{P}(R_3).
\end{align*}
In the second step we have used independence.
Using a gluing construction and FKG, it is easy to ``fill in'' the truncated regions and show
\begin{align*}
\mathbf{P}(\tilde{M}_4)&\le C\mathbf{P}(M_4),& \mathbf{P}(\tilde{M}_3)&\le C\mathbf{P}(M_3),\\
\mathbf{P}(\tilde{R}_1)&\le C\mathbf{P}(R_1), & \mathbf{P}(\tilde{R_2})&\le C\mathbf{P}(R_2).
\end{align*}
The point here is that the constants represented by $C$ do not depend on $\eta$ and $k$. This is due to the fact that the regions we must fill in have size of order $3^k(\eta-1)$.

Summarizing all the above, we now have
\[\mathbf{P}(e\in \tilde{\alpha}_2\mid Q_1(k))\le C\pi_3(3^k(\eta-1) ).\]
Summing over $e$ in $\Gamma(k,\eta)$, this gives
\[\mathbf{E}[\# \tilde{\alpha}_2\mid Q_1(k)] \le C\delta 3^{2k}\pi_3(3^k(\eta-1)), \]
with $\delta=\eta-1$. The lemma now follows by Chebyshev's inequality, and the observation that for all $1 \leq m \leq n$, one has 
%\begin{equation}\label{eqn: theboss'slegacy}
$\pi_3(m,n) \ge C_{3}(n/m)^{-\alpha}$ 
%\end{equation}
for some $C_{3}>0$ and $\alpha\in (0,1)$ (see Lemma~\ref{lem: AB}). Combined with quasi-multiplicativity, it gives:
\[\pi_3(3^k(\eta-1))\le (C_{3}(\eta-1))^{-\alpha}\pi_3(3^k),\]
and 
\[\mathbf{E}[\# \tilde{\alpha}_2\mid Q_1(k)] \le C \delta^{1-\alpha} 3^{2k}\pi_3(3^k).\]
Choosing $\delta$ sufficiently small gives the result.
\end{proof}
\end{lma}

\subsection{Lower bound for the volume of detoured crossings}
\label{sec: lowerbd}
We define boxes inside the annulus $B(\eta 3^k)\setminus B(3^{k-1})$, and events on which these boxes will be traversed by the open arms emanating from the origin, and contribute on the order of $(3^k)^2\pi_3(3^k)$ edges. The boxes are centered at the midpoints of $[-3^k,-\eta 3^{k-1}]$ and $[\eta 3^{k-1},3^k]$, respectively. Let
\[x_0(k) = \left(-\frac{3^k+3^{k-1}\eta}{2},0\right)= \left(-\frac{3^{k-1}(3+\eta)}{2},0\right).\]
The left ``interior box'' is
\[ B_{15}= B\left(x_0(k),\frac{3^{k-1}}{4}(3-\eta)\right).\]
Inside $B_{15}$, we place a smaller box (a quarter of the size), also centered at $x_0(k)$: it is defined by 
\[ B_{16} = B\left(x_0(k), \frac{3^{k-1}}{16}(3-\eta)\right).\]

\begin{figure}
\centering
\includegraphics[scale = 0.70]{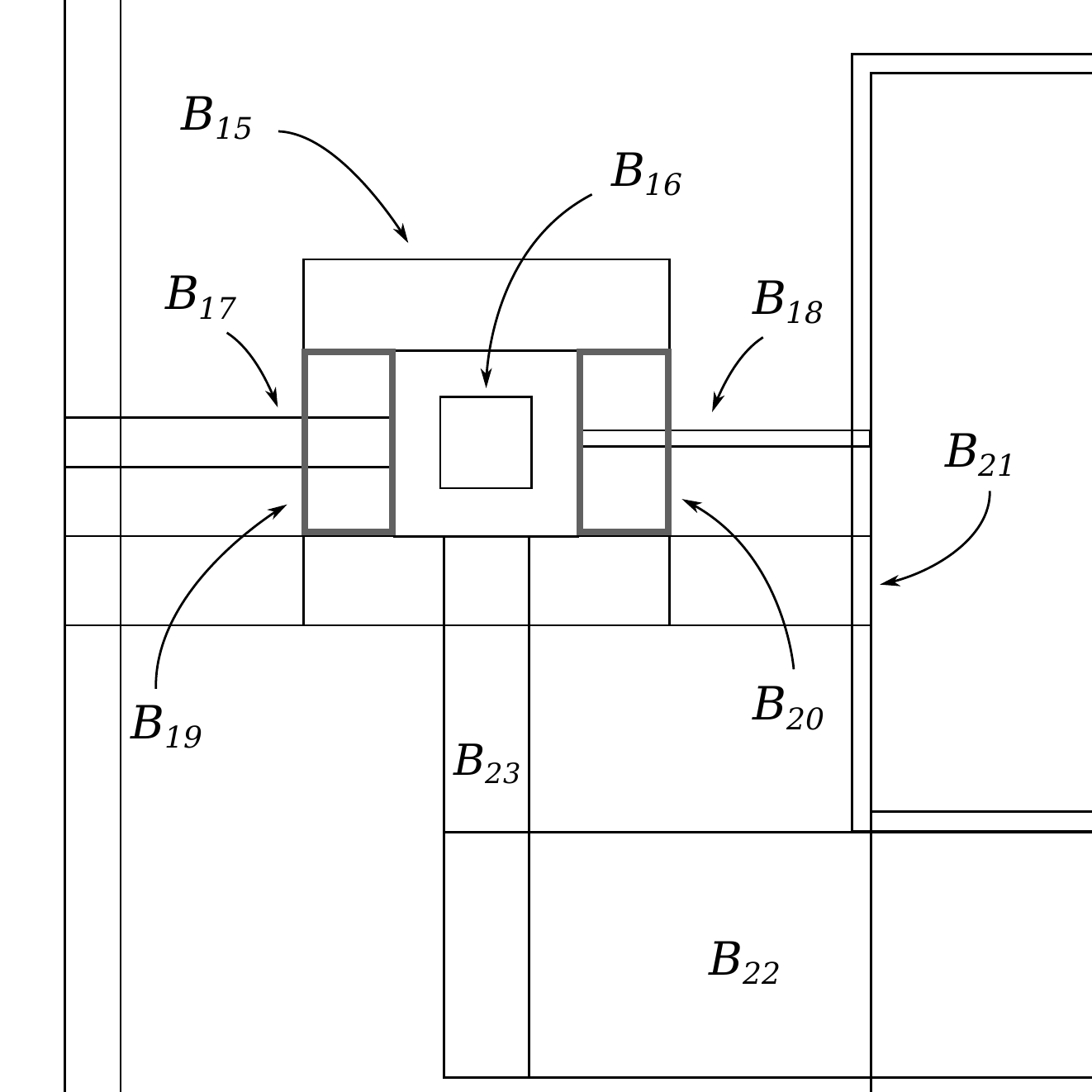}
\caption{The boxes appearing in the derivation of the lower bound for detoured crossings}
\label{consolidation-pic5}
\end{figure}

The boxes $B_{17}$ and $B_{18}$ have aspect ratios depending on $\delta=\eta-1$. Together with $B_{19}$ and $B_{20}$, they will be used to connect the 3-arm points inside $B_{15}$ to an open crossing of $Ann_1(k)$:
\begin{align*}
B_{17}&= [-3^{k},-\frac{3^{k-1}}{2}(3+\eta)-\frac{3^{k-1}}{8}(3-\eta)]\\
&\qquad \times [-\frac{3^{k}}{2}(\eta-1),\frac{3^{k}}{2}(\eta-1)],\\
B_{18}&= [-\frac{3^{k-1}}{2}(3+\eta)+\frac{3^{k-1}}{8}(3-\eta),-\eta 3^{k-1}]\\
&\qquad \times [-\frac{3^{k-1}}{2}(\eta-1),\frac{3^{k-1}}{2}(\eta-1)],\\
B_{19} &= [-\frac{3^{k-1}}{2}(3+\eta)-\frac{3^{k-1}}{4}(3-\eta),-\frac{3^{k-1}}{2}(3+\eta)-\frac{3^{k-1}}{8}(3-\eta)]\\
&\qquad \times [-\frac{3^{k-1}}{8}(3-\eta),\frac{3^{k-1}}{8}(3-\eta)],\\
B_{20} &= [-\frac{3^{k-1}}{2}(3+\eta)+\frac{3^{k-1}}{8}(3-\eta),-\frac{3^{k-1}}{2}(3+\eta)+\frac{3^{k-1}}{4}(3-\eta)]\\
&\qquad \times [-\frac{3^{k-1}}{8}(3-\eta),\frac{3^{k-1}}{8}(3-\eta)].
\end{align*}

The remaining boxes will serve to define crossing events to connect 3-arm points in $B_{16}$ to a closed crossing of $Ann_1(k)$.
\begin{align*}
B_{21} &=[-\eta 3^k, -3^{k-1}]\times [-\frac{3^{k-1}}{4}(3-\eta),-\frac{3^{k-1}}{8}(3-\eta)],\\
B_{22}&=[-\frac{3^{k-1}}{2}(3+\eta)-\frac{3^{k-1}}{16}(3-\eta),\frac{3^{k-1}}{2}(3+\eta)+\frac{3^{k-1}}{16}(3-\eta)]\\
&\qquad \times [-3^{k-2}(3+\eta),-\eta 3^{k-1}],\\
B_{23}&= [-\frac{3^{k-1}}{2}(3+\eta)-\frac{3^{k-1}}{16}(3-\eta),-\frac{3^{k-1}}{2}(3+\eta)+\frac{3^{k-1}}{16}(3-\eta)]\\
&\qquad\times [-3^{k-2}(3+\eta),-\frac{3^{k-1}}{8}(3-\eta)],\\
B_{24}&= [-3^{k-1},3^{k-1}]\times[-\eta 3^k,-3^{k-1}], \\
B_{25} &= B\left(x_0(k), \frac{3^{k-1}}{8}(3-\eta)\right).
\end{align*}

Let $M_5=M_5(k)$ be the event that there are open left-right crossings of $B_{17}$ and $B_{18}$. 
%By RSW, we have
%\begin{equation}
%\mathbf{P}(M_5)\ge C_{40}(\eta)>0.
%\end{equation}
$M_6=M_6(k)$ is defined as the event that there is a top-down closed dual crossing of $B_{23}$, and a left-right closed dual crossing of $B_{21}$. $M_7=M_7(k)$ is the event that there are open top-down crossings of both $B_{19}$ and $B_{20}$.

By Russo-Seymour-Welsh and Harris' inequality, there are positive constants $C(\eta)$ and $C$ such that
\begin{align*}
\mathbf{P}(M_5)&\ge C(\eta),\\
\mathbf{P}(M_6)&, \mathbf{P}(M_7)\ge C.\\
\end{align*}

We let $R_4(k)$ be the event that there is a left-right dual closed crossing of $B_{22}$ and a top-down dual closed crossings of $B_{24}$. Again by RSW and Harris, we obtain
\[\mathbf{P}(R_4(k))\ge C>0\]
for some constant $C$ independent of $k$.

We let
\[G_3(k) = G_3(k,\eta) = M_5(k)\cap M_6(k)\cap M_7(k),\]
and $G_4(k)$ is the reflection of $G_3$ about the $\mathbf{e}_2$-axis, as previously. The event $Q_2=Q_2(k,\eta)$ is defined as 
\begin{equation}
\label{eqn: q3def}
Q_3(k,\eta) = G_3(k)\cap G_4(k) \cap  R_4(k).
\end{equation}
By generalized FKG and independence, we have:
\[\mathbf{P}(Q_3) \ge C\mathbf{P}(M_5)^2.\]

By generalized FKG,
%\cite[Lemma 3]{kestenscaling} (see also \cite[Lemma 12]{nolin})
we have
\begin{equation}\label{eqn: capineq}
\mathbf{P}(Q_1\cap Q_2 \cap Q_3)\ge C(\eta)\mathbf{P}(Q_1\cap Q_2).
\end{equation}

%By generalized FKG \cite[Lemma 12]{nolin}, there is a constant $C_{46}$ independent of $k$ and $\eta$ such that
%\[\mathbf{P}(Q_1\cap Q_2)\ge C_{46}\mathbf{P}(Q_1)\mathbf{P}(Q_2).\]

On the event 
\begin{equation}\label{eqn: q4def}
Q_4=Q_4(k,\eta)=Q_1\cap Q_2\cap Q_3,
\end{equation}
let $W_2$ be the set of $e\in B_{15}$ such that $e$ has three disjoint arms, two of which are open and connected to the open crossings of $B_{17}$ and $B_{18}$, respectively, and one closed, and connected to the closed vertical crossing of $B_{23}$. We apply the second moment method (inequality \eqref{eqn: PZinequality}), to the number $Z_2=\#W_2$, conditionally on the event $Q_4$.

If $Q_4$ holds and $e \in B_{16}$ has three arms inside the rectangle $B_{17}\cup B_{19} \cup B_{25}$: one open and connected to the horizontal open crossing of $B_{17}$, another open arm, connected to the horizontal open crossing of $B_{18}$, and a closed arm connected to the vertical closed dual crossing of $B_{23}$, then $e\in W_2$. We denote the set of such edges $e\in B_{16}$ by $W_2^*$, and let $Z_2^*=\# W_2^*$. We have the following:
\begin{prop}\label{prop: q5}
Let 
\begin{equation}
\label{eqn: q5def}
Q_5 =Q_5(k,c) = \{Z_2 \ge c3^{2k}\pi(3^k)\}.
\end{equation}
There are constants $c,C_{12}>0$ such that 
\begin{equation}\label{eqn: q4ineq}
\mathbf{P}(Q_5 \mid Q_4) \ge C_{12}.
\end{equation}
for all $\delta=\eta-1$, $n \geq 1$ and $k \in \{K(\eta), \ldots, \lfloor (C_4/4) \log n \rfloor \}$.
\end{prop}

The important point here is that $c$ does not depend on $\delta$. As $\eta \downarrow 1$, the size of the detour shrinks, whereas the lower bound in this proposition does not change.
\begin{proof}
Recall the Paley-Zygmund inequality: if $Z\ge 0$ a.s., then 
\begin{equation}
\label{eqn: PZinequality}
\mathbf{P}(Z\ge \lambda\mathbf{E}Z)\ge (1-\lambda)^2\frac{(\mathbf{E}Z)^2}{\mathbf{E}Z^2}.
\end{equation}
To apply this with $Z=Z^*_2$ and $\mathbb{P}=\mathbf{P}(\cdot \mid Q_4)$, we give an estimate for the expectation
\begin{equation}\label{eqn: firstmmt}
(1/C)3^{2k}\pi(3^k) \le \mathbf{E}[Z^*_2\mid Q_4] \le C3^{2k}\pi(3^k),
\end{equation}
and an upper bound for the second moment:
\begin{equation}\label{eqn: secondmmt}
\mathbf{E}[(Z^*_2)^2\mid Q_4] \le C3^{2\cdot 2k}(\pi(3^k))^2.
\end{equation}

For $e\in B_{16}$, let $A^*(k,e)$ be the event that $e$ has three arms to the boundary of $B_{25}$: two open arms, one to each vertical side of $B_{25}$, and a closed arm to the middle third of the bottom side of $B_{25}$. By a simple gluing construction with generalized FKG and arms-separation \cite[Theorem 11]{nolin}, we have
\begin{align*}
\mathbf{P}(e\in W_2^*\mid Q_4)&\ge C\mathbf{P}(A^*(k,e))\\
&\ge C\mathbf{P}(A_3(e,3^k)),
\end{align*}
Summing over $e\in B_{16}$, we obtain
\[\mathbf{E}[Z_2^* \mid Q_4] \ge C3^{2k}\pi(3^k).\]
This gives the lower bound in \eqref{eqn: firstmmt}.

We now estimate the second moment. A similar argument gives the upper bound for the first moment. For simplicity of notation, let
\[m=\frac{3^{k-1}}{16}(3-\eta).\]
\begin{align*}
\mathbf{E}[(Z_2^*)^2\mid Q_4] &= \sum_{e_1,e_2\in B_{16}}\mathbf{P}(e_1\in Z_2^*, e_2\in Z_2^*\mid Q_4)\\
&\le \sum_{e_1,e_2\in B_{16}}\mathbf{P}\left(A_3(e_1,m),A_3(e_2,m)\right),
\end{align*}
where in the second inequality we have used that $Q_4$ is independent of the status of edges inside $B_{25}$.

The last double sum is decomposed following an idea of Nguyen \cite{nguyen}:
\begin{equation}\label{eq: meaty_sum}
\sum_{e_1}\sum_{d=1}^{2m} \sum_{|e_1-e_2|_\infty =d}\mathbf{P}(A_3(e_1,m), A_3(e_2,m)).
\end{equation}
For $k\le l$, let $A_3(e,k,l)$ be the probability that there are 3 arms from $\partial B(e,k)$ to $\partial B(e,l)$. Note that
\[\mathbf{P}(A_3(e,k,l))=\pi_3(k,l),\]
the three-arm probability (open, open, closed) for connection across the annulus $B(l) \setminus B(k)$. For convenience, if $l<k$, we define $A_3(e,k,l)$ to be the entire sample space; that is, $A_3(k,l)$ always holds. Correspondingly,
we let $\pi_3(k,l)=1$ in this case.
Then
\begin{align*}
\mathbf{P}(A_3(e_1,m), A_3(e_2,m)) &\le \mathbf{P}(A_3(e_1,d/2),A_3(e_1,3d/2,m),A_3(e_2,d/2))\\
&\le \mathbf{P}(A_3(e_1,d/2))\mathbf{P}(A_3(e_1,3d/2,m))\mathbf{P}(A_3(e_2,d/2)) \\
&= \pi_3(d/2) \pi_3(3d/2,m)\pi_3(d/2).
\end{align*}

Returning to the sum, we find, for each $e_1$, the bound
\begin{align*}
\sum_{d=1}^{\lfloor 2m/3\rfloor} 8d \pi_3(d/2) \pi_3(3d/2,m)\pi_3(d/2) +\sum_{d=\lfloor 2m/3\rfloor +1}^{2m} 8d \pi_3(d/2)\pi_3(d/2).
\end{align*}
Now we use RSW theory to rescale some of these quantities by constant factors. First, we have
\[\pi_3(d/2)\ge C\pi_3(d).\]
If $d\ge 2m/3$, we also obtain $\pi_3(d/2)\ge C\pi_3(m)$. By quasimultiplicativity \cite[Proposition 12.2]{nolin}:
\[\pi_3(d/2)\pi_3(3d/2,m)\asymp \pi_3(m).\]
Putting all this back into the sum \eqref{eq: meaty_sum}, we find the bound
\[m\pi_3(m)\sum_{d=1}^m \pi_3(d).\]
This is bounded by $m^2 \pi_3(m)^2$. To see why, choose (by RSW) $\beta>0$ such that $\pi_3(d,m) \geq C (m/d)^\beta$, and use quasimultiplicativity:
\[
\sum_{d=1}^m \pi_3(d) \asymp \pi_3(m) \sum_{d=1}^m \pi_3(d,m)^{-1} \leq C \frac{\pi_3(m)}{m^\beta} \sum_{d=1}^m d^\beta \leq Cm \pi_3(m)\ .
\]

Summing over $e_1$, we obtain an overall bound 
\[
\mathbf{E}[(Z_2^*)^2\mid Q_4]  \le C m^4 (\pi_3(m))^2 \le C3^{2\cdot 2k} (\pi_3(3^k))^2.
\]
A similar, but simpler argument gives the estimate
\[\mathbf{E}[Z_2^* \mid Q_4] \le C m^2(\pi_3(m))^2 \le C3^{2k}\pi(3^k),\]
concluding the proof.
\end{proof}

\subsection{Definition of $E_k$}

We can now give the definition of the events $E_k$:
\begin{df}\label{def: ekdef}
Let $\epsilon>0$ and $\eta=\eta(\epsilon)$ be given by \eqref{eqn: q2ineq} in Lemma \ref{lma: thin}. Then
\begin{equation}
E_k=E_k(\eta,\epsilon) = Q_1\cap Q_2 \cap Q_3 \cap Q_5 \text{ for } k \geq K(\eta), \label{eqn: ekdef}
\end{equation}
where $Q_1$ is defined in \eqref{eqn: q1def}, $Q_2$ in \eqref{eqn: q2def}, $Q_3$ in \eqref{eqn: q3def}, 
%$Q_4$ in \eqref{eqn: q4def}, 
and $Q_5$ appears in \eqref{eqn: q5def}.
\end{df}

Combining \eqref{eqn: capineq}, Lemma \ref{lma: thin}, Proposition \ref{prop: q5} and inequality \eqref{eq: Q_1_bound}, there is a constant $C_{13}(\epsilon)$ such that
\begin{equation}\label{eqn: ekbound}
\mathbf{P}(E_k)\ge C_{13}(\epsilon)>0 \text{ for } k \geq K(\eta)\ .
\end{equation}

To derive the lower bound
\begin{equation}
\mathbf{P}(E_k \mid A_3(d))\ge C_{7}(\epsilon)>0, \label{eqn: condek}
\end{equation}
for $k \geq K(\eta)$ such that $\eta 3^k \le d$, we use a gluing construction and arms separation \cite[Theorem 11]{nolin}, together with the equivalence \cite[Proposition 12, 2.]{nolin}
\[\mathbf{P}(A_3(3^{k-1}))\mathbf{P}(A_3(\eta 3^k,d))\asymp \mathbf{P}(A_3(d)).\]
Note that the definition of $E_k$ implies $A_3(3^{k-1},\eta 3^k)$ and that the connections across the annulus are easily extended.

\subsection{Proof of Proposition \ref{prop: whyek}}
\label{sec: essential}
\begin{proof}
We will show that if the event
\[E_k(e)\cap \{e\in \hat{\gamma}_n \} \]
 occurs for $k \in \{ K(\eta), \ldots, \lfloor (C_4/8)\log n \rfloor\}$, then there is a ``short'' detour around the origin in the sense that
\[\mathcal{S}(e)\neq \emptyset.\]

On $E_k(e)$, there is a closed circuit $\mathcal{C}_2$ with two defects near the five-arm points in $B_1$ and $B_1'$ inside $Ann_2$. We denote these (unique) points by $x$ and $y$. Since $e$ lies on the open, self-avoiding, circuit $\gamma_n$, the latter must pass through each of the two five-arm points in the definition of $M_1$, resp. $M_1'$, exactly once. We denote the portion of $\gamma_n$ between $x$ and $y$, and inside $\mathcal{C}_2$, by $q$. We also let 
\[p=\tilde{\alpha}_2,\]
where $\tilde{\alpha}_2$ is from Definition~\ref{def: outermost}.

\begin{claim}\label{claim: clam}
On $E_k(e) \cap \{e \in \hat \gamma_n\}$, $e^*$ has a closed connection to the bottom of $B(e,\eta 3^k)$ and to the bottom arc of the closed circuit $\mathcal{C}_2$ with defects in $Ann_2$.
\begin{proof}
Recall that $\Gamma(k,\eta)=Ann_2 \cap\left( \mathbb{R}\times [-\frac{3^k}{2}(\eta-1),\infty)\right)$. The open (detour) arc $\sigma_o$ in $\Gamma(k,\eta)$ between the two five-arm points and the closed arc $\sigma_c$ through $Ann_2\setminus \Gamma(k,\eta)$ form a circuit around $e$ (we can connect them by two line segments of length $1/\sqrt 2$ to make their union a closed curve). The closed arm emanating from $e^*$ reaches $\partial B(e, \eta 3^k)$, so it must intersect $\sigma_c$, which intersects the closed vertical crossing of $B_{24}$. This crossing is connected to the bottom side of $B(e,\eta 3^k)$.
\end{proof}
\end{claim}
It is important to note that since the closed arm from $e^*$ intersects the bottom of $B(e,\eta 3^k)$ and $e \in \hat \gamma_n$, the bottom part of the closed circuit in $Ann_1 \setminus \Gamma(k,\eta)$ must be connected to $B(n)$. This is what forces the ``orientation'' of the box $B(e,d)$ to be such that $p$ is indeed a detour off the innermost circuit (see Lemma~\ref{eqn: pingn} below).

\begin{claim}The open arc $p$ is disjoint from $q$ except for the five-arm points $x$ and $y$.
\begin{proof}
This follows from the definition of the five arm events. From this, we obtain the existence of an open crossing $\alpha$ of $B(e,\eta 3^k)$ inside 
\[B_{17}\cup B_{15} \cup B_{18} \cup B(e,\eta 3^{k-1}) \cup B_{17}'\cup B_{15}' \cup B_{18}'.\]
whose only intersection with the outermost arc $p$ is $x$ and $y$. By the previous claim, every dual edge touching $q$ is connected to the bottom of $B(e,\eta 3^k)$ by a closed dual path. This implies that $q$ lies in the region below the Jordan curve $\alpha$, which separates the box $B(\eta 3^k)$ into two connected components. In particular, $q$ is disjoint from $p$, except at its endpoints.
\end{proof}
\end{claim}

It follows that $p \cup q$ is a Jordan curve lying entirely inside the box $B(e,\eta 3^k)$. This in turn implies
\begin{lma} \label{eqn: pingn}
$p$ lies outside  $\mathrm{int} \gamma_n$.
\begin{proof}
The dual edge $\{x-(1/2)(\mathbf{e}_1 + \mathbf{e}_2), x-(1/2)(\mathbf{e}_1-\mathbf{e}_2)\}$ crosses $\gamma_n$ and so one endpoint is in each component of the complement of $\gamma_n$. The top endpoint can be connected to $p$ and the bottom one can be connected to the bottom arc of $\mathcal{C}_2$, both without crossing $\gamma_n$. By Claim~\ref{claim: clam}, the bottom is in the interior of $\gamma_n$, so $p$ is in the exterior of $\gamma_n$. 
\end{proof}
\end{lma}

We can now prove Proposition \ref{prop: whyek}, by setting $K(\epsilon) = K(\eta)$. Letting $P=p$ and $Q=q$, $w_0=x$, and  $w_M=y$, Lemma~\ref{eqn: pingn} implies that Condition 1. in Definition~\ref{def: shieldeddetours} is satisfied. Condition 2. holds by the definition of the five-arm points $x$ and $y$ (Definition \ref{def: fivearm}). Condition 3. follows because $P\cup Q =p\cup q$ is contained in the box $B(e,\eta 3^k)$, which does not contain the origin. Condition 4. follows from the existence of the closed dual arc $\tilde{\alpha}_1$, which is implied by the event $Q_1$ \eqref{eqn: q1def}. Condition 5. holds because of the conjunction of $Q_2$ \eqref{eqn: q2def} and $Q_5$ \eqref{eqn: q5def} (choose $\epsilon \cdot c$ for $\epsilon$, where $c$ is from the definition of $Q_5$).\end{proof}

\section{Probability Estimate}
Our goal in this section is to derive the estimate \eqref{eqn: Svest}. We recall it here:
\begin{equation}
\label{eqn: decay}
\mathbf{P}(\mathcal{S}(e)=\emptyset \mid e\in \hat \gamma_n)< \epsilon^2
\end{equation}
for some $n\ge C_5$. Once we show this estimate, Lemmas~\ref{circuit} and \ref{lem: piempty}, and Proposition~\ref{prop: four}, we can conclude Theorem~\ref{thm: mainthm}.

For $k=K(\eta),...,\lfloor (C_4/8)\log n \rfloor$, we let $E_k(e)=\tau_{-e}E_k$ be the ``detour event'' inside the annulus $Ann_1(e,k)=B(e,\eta 3^k)\setminus B(e,3^{k-1})$. (Here $\eta$ is slightly bigger than 1 and $K=K(\eta)$ is a constant depending on $\eta$ and which is defined under \eqref{eq: n_k_choice}.) It is defined precisely in Section \ref{sec: eksec} (see Definition \ref{def: ekdef}). The property we need here is proved in Proposition \ref{prop: whyek}: if $E_k(e)$ occurs and $e \in \hat \gamma_n$, then $\mathcal{S}(e) \neq \emptyset$. Thus,
\begin{equation}
\label{eqn: dodo}
\mathbf{P}(\mathcal{S}(e)=\emptyset \mid e\in \hat \gamma_n) \le \mathbf{P}((\cup_{k=K}^{\lfloor C_4/8 \log n\rfloor} E_{2k}(e))^c \mid e\in \hat \gamma_n).
\end{equation}

\subsection{Conditioning on 3-arm event in a box}
\label{sec: conditioning3}
The next step is to replace the conditioning in \eqref{eqn: dodo} by conditioning on a ``three arm'' event:
\begin{prop}\label{prop: switch}
There is a constant $C$ such that 
\begin{equation}\label{eqn: conditionswitch}
\mathbf{P}(\cap_{k=K}^{\lfloor C_4/8 \log n\rfloor} E_{2k}(e)^c \mid e\in \hat{\gamma}_n) \le C \mathbf{P}(\cap_{k=K}^{\lfloor C_4/8 \log n\rfloor} E_{2k}(e)^c \mid A_3(e,\mathrm{dist}(e,\partial A(n))),
\end{equation}
where $A_3(e,m)$ is the probability that $e$ has three arms, two open and one closed, to distance $m$ from $e$.
\end{prop}
We will omit some details, since most of the arguments are lengthy but standard. To prove \eqref{eqn: conditionswitch}, we use a gluing construction that depends on the position of $e$ inside the annulus $A(n)$, which we split into a number of different regions:
\[A(n) = A\cup B \cup C\cup D\cup E.\]
Region $B$ is is the disjoint union of four rectangles:
\begin{align*}
B&=[-2n,2n]\times [\frac{5}{2}n,3n] \cup [-2n,2n]\times[-3n,-\frac{5}{2}n]\\
&\quad \cup [-3n,-\frac{5}{2}n]\times [-2n,2n]\cup [\frac{5}{2}n,3n]\times[-2n,2n].
\end{align*}
Region $A$ is
\[A = \left( B(3n)\setminus B(5/2n) \right) \setminus B.\]
Region $C$ is given by
\[C=B(5n/2)\setminus B(3n/2).\]
Region $D$ is
\begin{align*}
D &= [-n/2,n/2]\times [n,3n/2] \cup [-n/2,n/2]\times [-3n/2,-n]\\
&\quad \cup [-3n/2,-n]\times [-n/2,n/2] \cup [n,3n/2]\times[-n/2,n/2].
\end{align*}
Finally, region $E$ is given by
\[E = (B(3n/2) \setminus B(n))\setminus D. \]

\begin{figure}
\centering
\includegraphics[scale = 0.5]{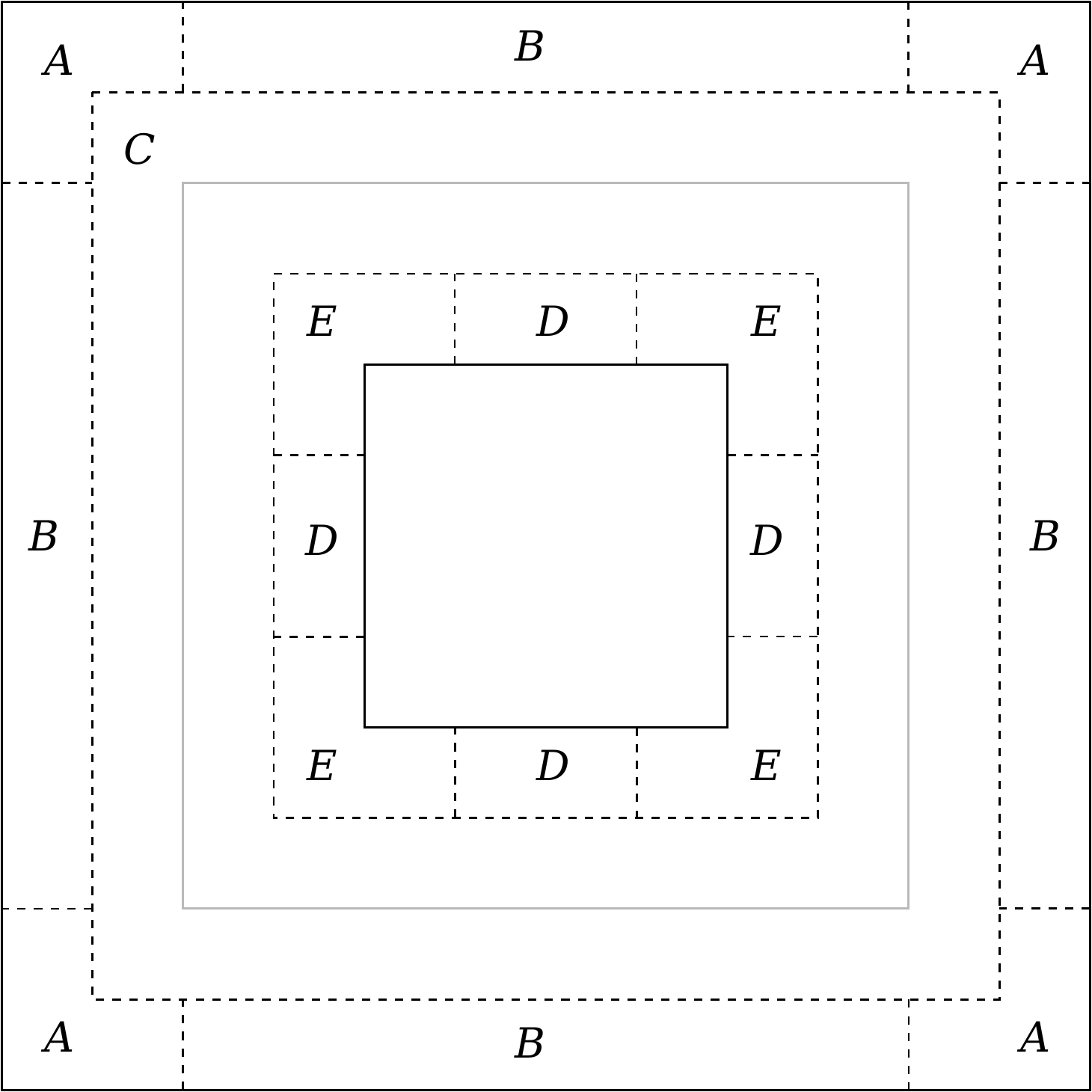}
\caption{The annulus $A(n)$ is split in a number of regions.}
\label{fig: regions}
\end{figure}

In each case, we use an adapted gluing construction to connect $e$ to $B(n)$ by a closed dual path inside an open circuit around $B(n)$. The proof will be different, depending on which region the edge $e$ lies in. Figure \ref{fig: regions} depicts the partitioning we will use.

We concentrate on the proof of \eqref{eqn: conditionswitch} in case $e\in A$. Furthermore, we assume by symmetry that $e$ is in the top-right component of $A$. We only consider $e\in \hat \gamma_n $, so 
\begin{equation}\label{eq: d_e}
d(e)=\mathrm{dist}(e,\partial A(n))\ge n^{C_4/2}.
\end{equation}
Since $e\in A$, we have $d(e)=\mathrm{dist}(e,\partial B(3n))$.

Denote by 
\[L(e)\in \{\{3n\}\times[-3n,3n], [-3n,3n]\times \{3n\}\},\]
the side of $\partial B(3n)$ such that $d(e) = \mathrm{dist}(e,L(e))$. If there is more than one possible choice, choose the earliest in the list above.

Let $B(e)$ be the box of side length $2d(e)$ centered at $e$. We define
\[d'(e) = \mathrm{dist}(e,\partial B(3n) \setminus L(e)),\]
and 
\[H(e) = B\left(e^L+d'(e) \cdot \frac{e-e_L}{|e-e_L|},d'(e)\right),\]
where $e^L$ is the projection of $e$ onto $L(e)$, and let $K(e)$ be the box
\[K(e) = [n,3n]\times [n,3n].\]

We now have
\begin{lma}\label{pathorder}
If $e\in \hat \gamma_n$, then the event $F_1(e)$ occurs: There are two open paths and one closed dual path joining $\partial B(e) \setminus \partial B(3n)$ to $\partial H(e) \setminus \partial B(3n)$ inside $H(e) \setminus B(e)$, appearing in the order 
\begin{equation}\label{eqn: order}
\text{open},\text{closed}, \text{open}
\end{equation}
(on the boundary of $H(e)$). In particular, the closed path is separated from $\partial B(3n)$ by the two open paths.

Similarly, the event $F_2(e)$ occurs: there two open paths and a closed dual path from $\partial H(e)\setminus \partial B(3n)$ to $\partial K(e)\setminus \partial B(3n)$ inside $K(e)\setminus H(e)$. These paths also appear in the order \eqref{eqn: order}, with the closed path separated from $\partial B(3n)$ by the open paths.
\end{lma}

In this lemma, we define $F_1(e)$ to be the sure event (that is, the entire sample space) if $8 d(e) > d'(e)$ and we define $F_2(e)$ to be the sure event if $4d'(e) > n$. This is to guarantee that later in the proof, there is enough room between the boxes $B(e)$ and $H(e)$ (or between $H(e)$ and $K(e)$) to do arm separation arguments. 

\begin{proof}
If $e\in _n$ then $e$ belongs to an open circuit surrounding $B(n)$ in $A(n)$. Moreover, $e^*$ is connected to $\partial B(n)$ by a closed path contained in the interior $\mathrm{int}\gamma_n$ of the circuit. 

Let $r_1$ be the portion of the open circuit $\hat{\gamma}_n$ obtained by traversing the circuit in one direction from $e$, until first time it exits $H(e)$. Call $a(e)\in \partial H(e)$ the point of exit. $r_2$ is the portion of $\gamma_n$ obtained by traversing the circuit in the other direction, until it first exits $H(e)$, at a point $b(e)$. The curve $\gamma_e$ contained in $H(e)$ joining $a(e)$ to $b(e)$ separates $H(e)$ into two regions, each bounded by the curve of $\gamma_e$ and a portion of $\partial H(e)\setminus \partial B(3n)$. Exactly one of these regions, $R(e)$, say, lies inside the circuit $\gamma_n$, and hence contains the portion of the closed dual path from $e$ to $\partial B(n)$ until it first exists $H(e)$. Following this path from $e$ until this exit point, we obtain a closed dual path whose endpoint $c(e)$ must lie on $\partial H(e)$, between $a(e)$ and $b(e)$. Traversing $r_1$ backwards from $a(e)$ and $r_2$ and $b(e)$ toward $e$ until the first time they enter $B(e)$, we obtain two points $a'(e)$ and $b'(e)$ on $\partial B(e)\setminus \partial B(3n)$. Following the closed path backwards similarly, we find a point $c'(e)$ lying between $a'(e)$ and $b'(e)$ on $\partial B(e)\setminus\partial B(3n)$.

The proof for the paths in $K(e)\setminus H(e)$ is similar.
\end{proof}

Returning to the probability in \eqref{eqn: conditionswitch}, write:
\begin{align*}
&\mathbf{P}(\cap_{k=K}^{\lfloor C_4/8\log n \rfloor} E_{2k}(e)^c \mid e\in \hat \gamma_n)\\
=&~\frac{1}{\mathbf{P}(e\in \hat \gamma_n)}\mathbf{P}(\cap_{k=K}^{\lfloor C_4/8\log n \rfloor} E_{2k}(e)^c, e\in \hat \gamma_n)\\
\le&~\frac{1}{\mathbf{P}(e\in \hat \gamma_n)}\mathbf{P}\left(\cap_{k=K}^{\lfloor C_4/8\log n \rfloor} E_{2k}(e)^c,F_1(e),F_2(e),A_3(e,d(e))\right)
\end{align*}

Note that the event
\begin{equation}
\cap_{k=K}^{\lfloor C_4/8 \log n \rfloor} E_{2k}(e)\label{eqn: capprob}
\end{equation}
depends only on edges inside $B(e, n^{C_4/4}) \subset B(e,n^{C_4/2})$, so  $F_1(e)$ and $F_2(e)$ are independent of \eqref{eqn: capprob}, so we have
\begin{align*}
&\mathbf{P}(\cap_{k=K}^{\lfloor C_4/8\log n \rfloor} E_{2k}(e)^c,F_1(e),F_2(e),A_3(e,d(e)))\\ 
\le& \ \mathbf{P}(F_1(e))\mathbf{P}(F_2(e))\mathbf{P}(\cap_{k=K}^{\lfloor C_4/8\log n \rfloor} E_{2k}(e)^c,A_3(e,d(e))).
\end{align*}

Proposition \ref{prop: switch} now follows from the next estimate:
\begin{lma}
There is a constant $C$ independent of $n$ and $e$ such that
\begin{equation}\label{eqn: gluing}
\mathbf{P}(e\in \hat \gamma_n) \ge \frac{1}{C}\mathbf{P}(F_1(e))\mathbf{P}(F_2(e))\mathbf{P}(A_3(e,d(e))),
\end{equation}
\end{lma}
\begin{proof}
We first introduce two events that will serve to complete a circuit around $A(n)$, once connected to the open arms coming out of $e$.
Let $\mathcal{C}_1$ be the event that there are open crossings along the long sides of the rectangles
\begin{align*} 
[-3n,n]\times [5n/2,3n],~& [-3n,-5n/2]\times [-3n,3n],\\
 [-3n,3n]\times [-3n,-5n/2],~&  [5n/2,3n]\times [-3n, n].
\end{align*}
Let $\mathcal{C}_2$ be the event that there are top-down and left-right closed crossings of the rectangle:
\[ [3n/4,n]\times [n,3n/2].\]

By the Russo-Seymour-Welsh theorem, the Harris inequality and independence, there is a positive $C_{6}$ independent of $n$ such that
\begin{equation}\label{eqn: Clwrbd}
\mathbf{P}(\mathcal{C}_1), \mathbf{P}(\mathcal{C}_2)\ge C_{6}.
\end{equation}

To connect the partial circuits $\mathcal{C}_1$ and $\mathcal{C}_2$ into a circuit containing the edge $e$, we use a standard arms separation argument (see for example \cite[Lemma 4]{kestenscaling}, \cite[Theorem 11]{nolin}), which allows us to specify landing areas on $\partial B_e$, $\partial H_e$ and $\partial K_e$ for the arms in events $F_1(e)$ and $F_2(e)$, while not modifying the probability of these events by more than a constant factor. The conclusion \eqref{eqn: gluing} is then obtained using the generalized FKG inequality. 

To define the modified arm events, we need to specify regions (``landing zones'') that will contain the endpoints. For this, we divide the left side of $\partial B_e$ into four vertical segments, which we label from top to bottom: $I_1, I_2, I_3, I_4$ of equal length $d(e)/2$. The bottom side of $\partial B_e$ is also divided into three horizontal segments of equal size, which we label according to their position from left to right: $I_1'$,$I_2'$, $I_3'$ and $I_4'$. We proceed similarly with the left side of $\partial H_e$, which we also divide into four parts $J_1$, $J_2$, $J_3$, $J_4$ of equal size, labeled from top to bottom. 
%These have size $d'(e)/2$ or $d'(e)/4$, depending on whether the left side of $\partial H_e$ is has length $2d'(e)$ or $d'(e)$. 
The bottom side of $\partial H_e$ is also divided into four parts of equal size: $J_1'$, $J_2'$, $J_3'$ and $J_4'$, labeled from left to right. $I_4$ and $I_1'$ intersect at the lower left corner of $B_e$; $J_4$ and $J_1'$ intersect at the lower left corner of $H_e$. Note also that 
\begin{align*}
J_1&\subset [-3n,3n]\times [5n/2,3n]\\
J_4'&\subset [5n/2,3n]\times[-3n,3n].
\end{align*}

$\tilde{A}_3(e,d(e))$ is the event that $A_3(e,d(e))$ occurs, one of the open arms from $e$ has its other endpoint in $I_2$, and the other arm has its endpoint in $I_2'$. The closed arm has its endpoint in $I_4^*$. $\tilde{F}_1(e)$ is the event that $F_1(e)$ occurs, one of the open arms having its endpoints in $I_2$ and $J_1$, respectively, and the other open arm having endpoints in $I_2'$ and $J_4'$. Moreover, we require the closed arm to have its endpoints in $I_4^*$ and $J_4^*$. $\tilde{F}_2(e)$ is the event that $F_2(e)$ occurs, one open arm has endpoints in $J_1$ and $\{n\}\times [5n/2,3n]$, and the other in $J_4'$ and $[5n/2,3n]\times \{n\}$. Finally the closed dual arm is required to have one endpoint in $J_4^*$, and the other in $\{n\}\times [n,3n/2]$.

\begin{figure}
\centering
\includegraphics[scale = 0.5]{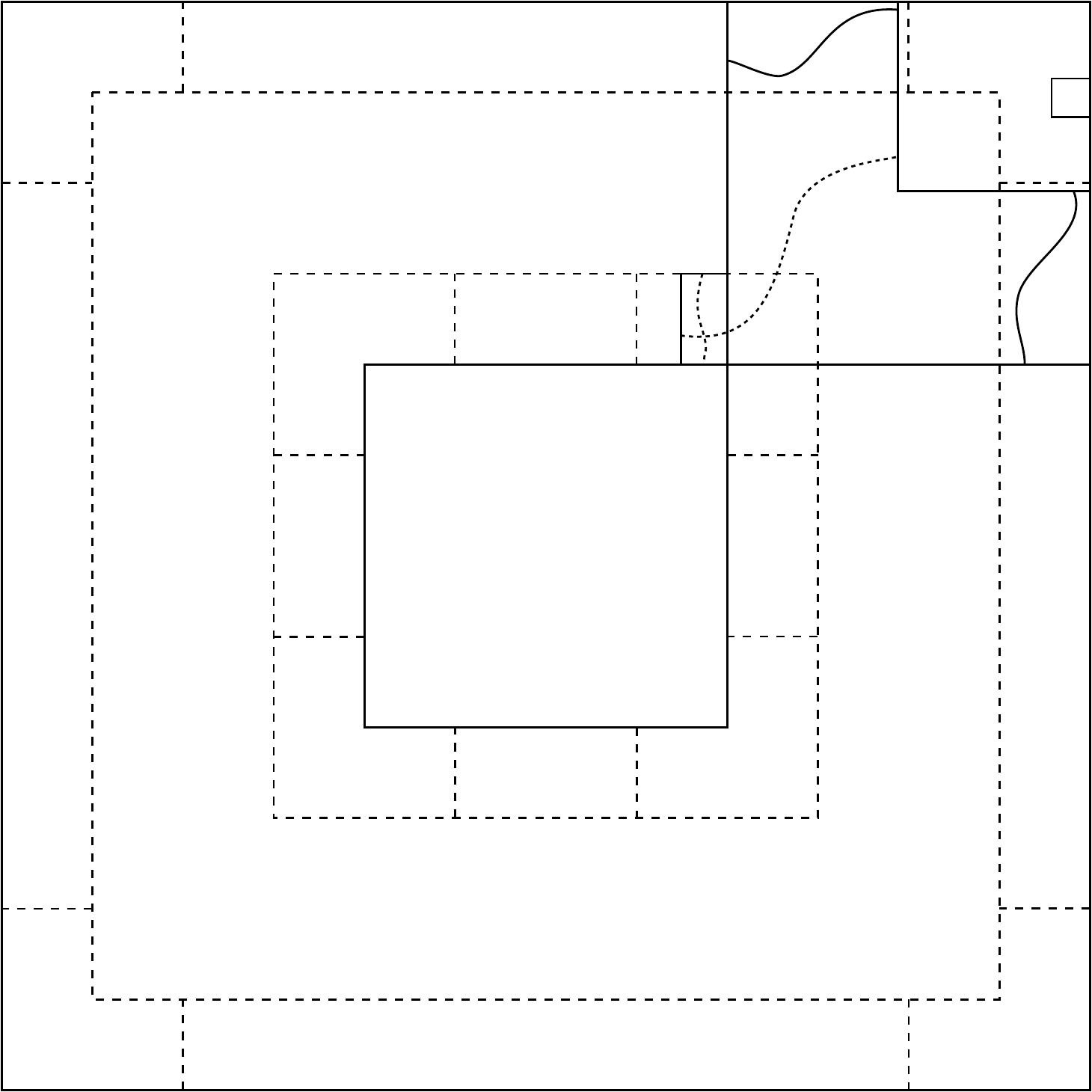}
\caption{A sketch of the construction inside $K_e$ in case $e$ lies in the top right component of $A$}
\label{fig: pic3}
\end{figure}

On $\tilde{A}_3(e,d(e))\cap \tilde{F}_1(e)\cap\tilde{F}_2(e)$, standard gluing techniques allow us connect each of the open arms the definition of $\tilde{F}_1(e)$ to one end of the open arc in the event $\mathcal{C}_1$, and the closed arm to the vertical crossing of $[3n/4,n]\times [n,3n/2]$ appearing in the definition of the event $\mathcal{C}_2$, and to connect the arms in each of the three events to the arms with endpoints in the same region. Combined with the generalized FKG inequality, this gives:
\begin{equation}\mathbf{P}(e\in _n) \ge C_{6}^2C\mathbf{P}(\tilde{F}_1(e))\mathbf{P}(\tilde{F}_2(e))\mathbf{P}(\tilde{A}_3(e,d(e))),\label{eqn: FKGapp}
\end{equation}
where $C$ is independent of $n$ and $C_{6}$ appears in \eqref{eqn: Clwrbd}.

By arms separation, \cite[Theorem 11]{nolin}, we have
\begin{equation}\label{eqn: separatedA}
\mathbf{P}(\tilde{A}_3(e,d(e)))\ge C\mathbf{P}(A_3(e,d(e))).
\end{equation}

An argument similar to the proof of \cite[Theorem 11]{nolin} (see also the proof of \cite[Lemma 4]{kestenscaling})
gives the existence of a constant $C$ independent of $n$, such that:
\begin{align}
\mathbf{P}(\tilde{F}_1(e))&\ge C\mathbf{P}(F_1(e)) \label{eqn: separatedF1}\\
\mathbf{P}(\tilde{F}_2(e))&\ge C\mathbf{P}(F_2(e)). \label{eqn: separatedF2}
\end{align}
It is important here that the arms in the definition of $F_1(e)$ and $F_2(e)$ appear in a definite order, as guaranteed by Lemma \ref{pathorder}.

Combining \eqref{eqn: separatedA}, \eqref{eqn: separatedF1}, \eqref{eqn: separatedF2} and \eqref{eqn: FKGapp}, we obtain \eqref{eqn: gluing} in the case where $e$ lies in the upper right part of the region $A$. Similar gluing constructions also apply in the other cases.
\end{proof}

\subsection{Arms separation conditional on $A_3$}
By Proposition \ref{prop: switch}, we have, for any $e\in \hat{A}(n)$:
\begin{equation}\label{eqn: applyhere}
\mathbf{P}(S(e) =\emptyset \mid e\in \hat \gamma_n) \le C\mathbf{P}(\cap_{k=K}^{\lfloor C_4/8 \log n \rfloor} E_{2k}(e)^c \mid A_3(e,d(e))),
\end{equation}
where $d(e)$ is defined in \eqref{eq: d_e} and $K=K(\eta)$ is defined below \eqref{eq: n_k_choice}. The events $E_{2k}$ depend on disjoint sets of edges, and each occurs with probability bounded below independently of $k$ (see \eqref{eqn: ekbound}), so we expect an estimate of the form \eqref{eqn: decay}. 

However, we must ensure that the conditioning on the three arm event $A_3(e,d(e))$ does not affect the probability of occurrence of the $E_{2k}$'s too drastically. To state our result, let $d=3^m$ be a (large) integer. We will later take $m=\lfloor \log_3 d(e)\rfloor$.
Our assumption will be
\begin{equation}\label{eq: assumption}
\mathbf{P}(E_{2k} \mid A_3(d)) \geq C_{7}>0 \text{ for } K(\eta) \leq 2k \leq \lfloor (C_4/8)\log n \rfloor.
\end{equation}
This is the lower bound \eqref{eqn: condek}. From the definition of $E_{2k}$ (Definition \ref{def: ekdef}), we also have
\begin{equation}\label{eq: E_k_dependence}
E_{2k} \text{ depends only on edges in } B(3^{2k+1})\setminus B(3^{2k-1}).
\end{equation}

%The main result of the section is
%\begin{prop}\label{eqn: a3separation}
%For each $\epsilon>0$, $n$ sufficiently large, $m := \lfloor \log_3 d\rfloor \geq \lfloor (C_4/8)\log n\rfloor$, and any events $E_k$ satisfying \eqref{eq: assumption} and \eqref{eq: E_k_dependence},
%\[
%\mathbf{P}\left(\text{at most }c\log n \text{ of the } E_k\text{'s occur} \mid A_3(d) \right) <\epsilon\ .
%\]
%\end{prop}
%Assuming \eqref{eq: assumption}, the estimate \eqref{eqn: decay} follows immediately by applying the proposition to the right side of \eqref{eqn: applyhere}.

%\bigskip
%{\tt NEW STUFF}

Set $m_n = \lfloor (C_4/8)\log n \rfloor$. Our goal will be to prove that
\begin{equation}\label{eq: new_main_result_cheese_head}
\mathbf{P}(\cap_{k=K}^{m_n} E_{2k}^c \mid A_3(N)) \to 0 \text{ as } n \to \infty
\end{equation}
uniformly in $N \geq n^{C_4/2}$. Given this result, we find by translation invariance
\begin{prop}\label{eqn: a3separation}
As  $n \to \infty$,
\[\mathbf{P}(S(e) = \emptyset \mid e \in \hat \gamma_n) \to 0\]
uniformly for $e \in \hat A(n)$.
\end{prop}

%The main result of the section is
%\begin{prop}\label{eqn: a3separation}
%For each $\epsilon>0$, $n$ sufficiently large, $m := \lfloor \log_3 d\rfloor \geq \lfloor (C_4/8)\log n\rfloor$, and any events $E_k$ satisfying \eqref{eq: assumption} and \eqref{eq: E_k_dependence},
%\[
%\mathbf{P}\left(\text{at most }c\log n \text{ of the } E_k\text{'s occur} \mid A_3(d) \right) <\epsilon\ .
%\]
%\end{prop}
%Assuming \eqref{eq: assumption}, the estimate \eqref{eqn: decay} follows immediately by applying the proposition to the right side of \eqref{eqn: applyhere}.

We begin with with the following intermediate statement.
\begin{claim}\label{claim: clam_head}
For a sequence of integers
\[3^{2K+1} < i_1<i_2<\ldots<i_k<\ldots\]
let $B_k$ be the event that there exists a closed dual circuit with two defects (that is, two edges that are open) around 0 in $Ann(i_k,i_{k+1})$, and
\[k_N = \max\{k: i_{k+1}<N\}.\]
Furthermore, let $\hat B_k$ be the event that there exists an open circuit with one defect around 0 in $Ann(i_k,i_{k+1})$. Given $\epsilon>0$, there is a choice of $i_1, i_2, \dots$ such that $i_{k+1} > 3^6 i_k$ for all $k$ and
\begin{equation}
\mathbf{P}(\hat B_1^c \cup (\cup_{k=2}^{k_N}B_k^c) \mid A_3(N))<\epsilon
\end{equation}
for all $N$.
\begin{proof}
By quasimultiplicativity, we can choose $C_{14}$ such that for all $m_1 < m_2 < N$,
\[\mathbf{P}(A_3(m_1))\mathbf{P}(A_3(m_1,m_2))\mathbf{P}(A_3(m_2,N))\le C_{14}\mathbf{P}(A_3(N)).\]
For any sequence $i_1<i_2<\ldots$, let 
\[\alpha_k=\mathbf{P}(\text{there is an open crossing or a closed dual crossing of } Ann(i_k,i_{k+1})).\]
Choose the sequence $(i_k)_{k\ge 1}$ such that
\[\sum_{k=1}^\infty \alpha_k \le \epsilon/C_{14}.\]
Then, estimate
\begin{align}
\mathbf{P}(A_3(N), \hat B_1^c \cup (\cup_{k=2}^{k_N}B_k^c)) &\le \mathbf{P}(A_3(N),\hat B_1^c) + \sum_{k=2}^{k_N} \mathbf{P}(A_3(N), B_k^c) \nonumber \\ 
&\leq \mathbf{P}(A_3(i_1)) \mathbf{P}(A_3(i_1,i_2),\hat B_1^c) \mathbf{P}(A_3(i_2,N)) \nonumber \\
&+ \sum_{k=2}^{k_N} \mathbf{P}(A_3(i_k))\mathbf{P}(A_3(i_k,i_{k+1}), B_k^c)\mathbf{P}(A_3(i_{k+1},N)). \label{eqn: popeye}
\end{align}

If $A_3(i_k,i_{k+1})$ occurs but $B_k$ does not occur, then there must be an open crossing of $Ann(i_k,i_{k+1})$ that is disjoint from the three crossings from $A_3(i_k,i_{k+1})$. A similar statement holds for $\hat B_1$. Therefore, by Reimer's inequality \cite{reimer}, \eqref{eqn: popeye} is bounded by

\begin{align*}
&\left(\sum_{k=1}^{k_N} \alpha_k\right) \mathbf{P}(A_3(i_k))\mathbf{P}(A_3(i_k,i_{k+1}))\mathbf{P}(A_3(i_{k+1},N))\\
\le& ~C_{14}\mathbf{P}(A_3(N))\sum_{k=1}^{k_N} \alpha_k \le \epsilon \cdot \mathbf{P}(A_3(N)).
\end{align*}

\end{proof}
\end{claim}

\begin{proof}[Proof of Proposition \ref{eqn: a3separation}]
Choose $i_1 < i_2 < \cdots$ from the previous claim corresponding to $\epsilon/2$. For $k_N \geq 3$, consider the event
\[
C_N = \hat B_1 \cap (\cap_{k=2}^{k_N} B_k).
\]
The first step is to show that there is a constant $c_1>0$ such that for all $N$ large, all $k$ satisfying $3 \leq k \leq k_N-3$, $E$ any event depending on the state of edges in $B(i_k)$, and $F$ any event depending on the state of edges in $B(i_{k+4})^c$, then
\begin{equation}\label{eq: step_one}
\mathbf{P}(F \mid A_3(N),C_N,E) \geq c_1 \mathbf{P}(F\mid A_3(N), C_N)\ .
\end{equation}

On the event $C_N \cap A_3(N)$, there is an innermost dual circuit with two defects in each annulus $Ann(i_k,i_{k+1})$. (This circuit is also vertex self-avoiding.) For a given dual circuit $\mathcal{C}$ in $Ann(i_k,i_{k+1})$ with edges $e_1,e_2$ on $\mathcal{C}$, we let $Circ_k(\mathcal{C})$ be the event that $\mathcal{C}$ is the innermost closed dual circuit with defects $e_i$. Generally, if $A_3(N)$ does not occur, then the event $Circ_k(\mathcal{C})$ means that $\mathcal{C}$ is closed, $e_1$ and $e_2$ are open, and there is no closed circuit with two defects around 0 in $Ann(i_k,i_{k+1})$ entirely contained in the union of $\mathcal{C}$ with its interior.

Then we can decompose
\begin{equation}\label{eq: pre_decompose}
\mathbf{P}(F,E,A_3(N),C_N) = \sum_{\mathcal{C},\mathcal{D}} \mathbf{P}(F,E,A_3(N),Circ_k(\mathcal{C}),Circ_{k+3}(\mathcal{D}),C_N)\ .
\end{equation}
To decouple, we must introduce three events that build $A_3(N)$. Every dual circuit $\mathcal{C}$ or $\mathcal{D}$ above contains two arcs between its defects. Given a deterministic ordering of all arcs and a dual circuit $\mathcal{C}$, let $\mathcal{A}_i(\mathcal{C})$ be the $i$-th arc of $\mathcal{C}$ in this ordering.  For $i=1,2$, let $X_-(\mathcal{C},i)$ be the event that 0 is connected to $e_1$ and $e_2$ by disjoint open paths in the interior of $\mathcal{C}$, and to $\mathcal{A}_i(\mathcal{C})$ by one closed dual path in the interior of $\mathcal{C}$. Let $C_N^{k-}$ be the event that $\hat B_1 \cap (\cap_{i =2}^{k-1} B_i)$ occurs. For $i,j=1,2$, let $X_0(\mathcal{C},\mathcal{D},i,j)$ be the event that $e_l$ is connected to $f_l$ (for $l=1,2$) by an open path in the region between $\mathcal{C}$ and $\mathcal{D}$ (not including $\mathcal{C}$) such that these paths are disjoint and $\mathcal{A}_i(\mathcal{C})$ is connected to $\mathcal{A}_j(\mathcal{D})$ by a closed dual path in this same region. Let $X_+(\mathcal{D},j)$ be the event that $f_1$ and $f_2$ are connected to $\partial B(N)$ by disjoint open paths in the exterior of $\mathcal{D}$ and $\mathcal{A}_j(\mathcal{D})$ is connected to $\partial B(N)$ by a closed dual path in the exterior of $\mathcal{D}$. Also let $C_N^{k+}$ be the event that $\cap_{i=k+4}^{k_N}B_i$ occurs.

Then \eqref{eq: pre_decompose} becomes by independence,
\begin{align*}
&\sum_{\mathcal{C},\mathcal{D}}\sum_{i,j} \mathbf{P}(F,E,X_-(\mathcal{C},i), X_0(\mathcal{C},\mathcal{D},i,j),X_+(\mathcal{D},j),Circ_k(\mathcal{C}),Circ_{k+3}(\mathcal{D}), C_N) \\
=~& \sum_{\mathcal{C},\mathcal{D}}\sum_{i,j} \mathbf{P}(E,X_-(\mathcal{C},i), Circ_k(\mathcal{C}),C_N^{k-}) ~\mathbf{P}(X_0(\mathcal{C},\mathcal{D},i,j), Circ_{k+3}(\mathcal{D}), B_{k+1},B_{k+2})\\
&\qquad \times  \mathbf{P}(X_+(\mathcal{D},j), F, C_N^{k+})\ .
\end{align*}

The effect of this decoupling will be to ``reset'' the system outside of the outer circuit $\mathcal{D}$, so that the event $E$ no longer significantly affects the occurrence of $F$. Intuitively speaking, $E$ could affect the system by biasing certain circuits $\mathcal{C}$ to appear in $Ann(i_k,i_{k+1})$, and these could change the conditional probability of $F$. However, a lemma from \cite[Lemma~6.1]{DS} below will show that the second circuit $\mathcal{D}$ will mostly remove this possible bias and allow the system to start fresh. We give here a modification of that lemma, which follows from essentially the same proof.
\begin{lma}\label{lem: DS}
Consider dual circuits $\mathcal{C}$ in $Ann(i_k,i_{k+1})$, $\mathcal{D}$ in $Ann(i_{k+3},i_{k+4})$, edges $e_1, e_2$ on $\mathcal{C}$ and $f_1, f_2$ on $\mathcal{D}$ respectively. For $i,j = 1,2$, let $P(\mathcal{C},\mathcal{D},i,j)$ be the probability, conditional on the event that all edges in $\mathcal{C}\setminus \{e_1, e_2\}$ are closed and $e_1, e_2$ are open, that (1) there are disjoint open paths from $e_i$ to $f_i$ in the region between $\mathcal{C}$ and $\mathcal{D}$ (not including $\mathcal{C}$), (2) there is a closed dual path from $\mathcal{A}_i(\mathcal{C})$ to $\mathcal{A}_j(\mathcal{D})$ in the region between $\mathcal{C}$ and $\mathcal{D}$ (not including $\mathcal{C}$), (3) $\mathcal{D}$ is the innermost closed dual circuit with defects $f_1, f_2$ around 0 in $Ann(i_{k+3},i_{k+4})$, and (4) $B_{i_{k+1}} \cap B_{i_{k+2}}$ occurs. We similarly define $\mathcal{C}',\mathcal{D}', i', j'$, etc. There exists a finite constant $C_8$ (it does not depend on the particular choice of circuits, defects, or $i,i',j,j'$) such that
\[
\frac{P(\mathcal{C},\mathcal{D},i,j) P(\mathcal{C}',\mathcal{D}',i',j')}{P(\mathcal{C},\mathcal{D}',i,j')P(\mathcal{C}',\mathcal{D},i',j)} < C_8\ .
\]
\end{lma}
The proof of this statement uses extensions of arm separation techniques developed by Kesten. One obtains
\begin{align*}
 &\mathbf{P}(E,F,A_3(N), C_N) \\
= &~\sum_{\mathcal{C},\mathcal{D}} \sum_{i,j} \mathbf{P}(E,X_-(\mathcal{C},i),Circ_k(\mathcal{C}), C_N^{k-}) P(\mathcal{C},\mathcal{D},i,j) \mathbf{P}(X_+(\mathcal{D},j),F, C_N^{k+})\ .
\end{align*}
Similarly,
\[
\mathbf{P}(A_3(N), C_N) = \sum_{\mathcal{C}',\mathcal{D}'} \sum_{i',j'} \mathbf{P}(X_-(\mathcal{C}',i'),Circ_k(\mathcal{C}'), C_N^{k-})P(\mathcal{C}',\mathcal{D}',i',j') \mathbf{P}(X_+(\mathcal{D}',j'), C_N^{k+})\ .
\]
Multiplying these and using Lemma~\ref{lem: DS}, one obtains
\begin{align*}
 &\mathbf{P}(E,F,A_3(N), C_N)~\mathbf{P}(A_3(N), C_N) \\
\geq & \left(\frac{1}{C_8}\right)^2 \sum_{\mathcal{C},\mathcal{C}',\mathcal{D},\mathcal{D}'} \sum_{i,j,i',j'} \bigg[ \mathbf{P}(E,X_-(\mathcal{C},i),Circ_k(\mathcal{C}), C_N^{k-}) P(\mathcal{C},\mathcal{D}',i,j') \mathbf{P}(X_+(\mathcal{D}',j'), C_N^{k+})\\
& \times \mathbf{P}(X_-(\mathcal{C}',i'),Circ_k(\mathcal{C}'), C_N^{k-})P(\mathcal{C}',\mathcal{D},i',j) \mathbf{P}(X_+(\mathcal{D},j),F, C_N^{k+}) \bigg] \\
= & \left( \frac{1}{C_8} \right)^2 \mathbf{P}(E,A_3(N), C_N) \mathbf{P}(A_3(N), F, C_N).
\end{align*}
Dividing gives
\[
\mathbf{P}(F \mid E,A_3(N), C_N) \geq (1/C_8)^2 \mathbf{P}(F \mid A_3(N), C_N)\ .
\]
This shows \eqref{eq: step_one} with $c_1 = (1/C_8)^2$.

To finish the proof of \eqref{eq: new_main_result_cheese_head}, we use estimate \eqref{eq: step_one} to show that at least one $E_{2k}$ occurs. The idea is to consider a maximal sub collection  $F_1, F_2, \ldots$ of the $E_{2k}$'s such that $F_1$ depends on the state of edges in $B(i_3)$, $F_2$ depends on the state of edges in $B(i_8) \setminus B(i_7)$, $F_3$ depends on the state of edges in $B(i_{13}) \setminus B(i_{12})$, and so on. Write $r_n$ for the largest $k$ such that $F_k$ depends on edges in $B(3^{2m_n + 1})$. Then
\begin{align*}
\mathbf{P}\left( \cap_{k=1}^{r_n} F_k^c \mid A_3(N), C_N\right) &= \prod_{k=1}^{r_n} \mathbf{P}\left( F_k^c \mid A_3(N), C_N, \cap_{l=1}^{k-1}F_l^c \right) \\
&\leq \prod_{k=1}^{r_n} \left( 1- c_1 \mathbf{P}(F_k \mid A_3(N), C_N) \right) \\
&\leq (1-a_{n,N})^{r_n},
\end{align*}
where
\[
a_{n,N} = \min_{1 \leq k \leq r_n} \mathbf{P}(F_k \mid A_3(N), C_N).
\]
However by the bound \eqref{eq: assumption}, one has
\[
\mathbf{P}(E_{2k} \mid A_3(N),C_N) \geq \mathbf{P}(E_{2k} \mid A_3(N)) - \mathbf{P}(C_N^c \mid A_3(N)) \geq C_7-\epsilon/2.
\]
So for $\epsilon < C_7$,
\[
\mathbf{P}\left( \cap_{k=1}^{r_n} F_k^c \mid A_3(N), C_N \right) \to 0 \text{ in } n \text{ uniformly in }N \geq n^{C_4/2}.
\]
Combining this with Claim~\ref{claim: clam_head}, one has
\[
\mathbf{P}(\cap_{k=K}^{m_n} E_{2k}^c \mid A_3(N)) \leq \mathbf{P}(\cap_{k=K}^{m_n} E_{2k}^c \mid A_3(N), C_N) + \mathbf{P}(C_N^c \mid A_3(N)) < \epsilon
\]
for $n$ large and uniformly in $N \geq n^{C_4/2}$.

\end{proof}

\section{The lower tail of $\tilde{L}_n$}
\label{sec: lowertail}
\begin{lma}\label{lma: lowertail}
Let $\tilde{L}_n$ be the number of edges in the lowest crossing of $[-n,n]^2$. Then
\begin{equation}\label{eq: last_sec_to_show}
\lim_{\epsilon \downarrow 0} \limsup_n \mathbf{P}(0<\tilde{L}_n < \epsilon n^2\pi_3(n)) = 0.
\end{equation}
\end{lma}

A bound analogous to \eqref{eq: last_sec_to_show} for another set, the pivotal edges in $[-n,n]^2$, was stated in \cite{GPS} (see Remark 1.7 there). There it appears as an application of a more general method developed to study the lower tail of the Fourier spectrum of the indicator of the existence of an open crossing. Here, we will adopt a different strategy. 

The idea of the proof is similar to that of Kesten's proof that at criticality, the expected number of edges which are pivotal for a crossing event is at least order $\log n$ \cite{kestencrit}. We first restrict to the event that the maximum number of disjoint open left-right crossings of the box is exactly $k$. Next, we condition successively on the $k$ upper-most disjoint open crossings. Calling $T_k$ the $k$-th such crossing, we then condition on the leftmost top-down dual closed crossing $p$ connecting $T_k$ to the bottom of the box (there must be such a crossing, since there are no more disjoint left-right open crossings below $T_k$). Calling $e_k$ the edge at the intersection of $p$ and $T_k$, we then use independence of the edge variables for edges in the region below and to the right of $e_k$ to build many ``three-arm'' edges in this region in annuli centered at $e_k$. The crucial point is that each such edge will have two disjoint open arms to $T_k$ and one closed arm to $p$, and will therefore be an edge on the lowest crossing of the box. Since there is a lower bound for the probability of many such edges existing in each annulus, we obtain that with high probability, many such edges exist in at least one annulus, and this implies $\tilde{L}_n \geq  \epsilon n^2\pi_3(n)$ with high probability.

The main difficulty in our construction (and it is this point that makes ours more complicated than the one in Kesten's proof) is that we do not want only $\log n$ number of edges on the lowest crossing, but at least $\epsilon n^2\pi_3(n)$. This corresponds to the fact that in each annulus, Kesten needs only to produce one pivotal point, whereas we need to build many. For us to do so, the region below and to the right of $e_k$ must have many open spaces. Specifically, we must first show that with probability of order $1-o_\epsilon(1)$, in each large enough annulus centered at $e_k$, we can find a box of size at least $\epsilon^\delta n$, for some $c$ and $\delta>0$, which lies entirely in the region below and to the right of $e_k$. This will be done using a six-arm argument: if $T_k$ and $p$ come too close to each other, certain annuli will have six-arm events, and this is unlikely. Next, we must construct such a box and show that three-arm edges in this box have enough room to connect to $T_k$ and to $p$.

\subsection{Proof of Lemma \ref{lma: lowertail}}Let $A_n=A_n(\epsilon)$ be the event in the probability \eqref{eq: last_sec_to_show}. Let $M_K = M_K(n)$ be the event that there are at most $K$ disjoint open crossings of the box $[-n,n]^2$.
Note that by the BK-Reimer inequality and the RSW theorem, 
\[\mathbf{P}(M_K^c)\le (1-C)^K,\] 
uniformly in $n$.

We further let $D_k = M_k \setminus M_{k+1}$ be the event that the maximal number of disjoint crossings equals $k$. Then
\[M_K = \cup_{k=1}^K D_k,\]
and the union is disjoint. Hence, we are left with showing that for $\epsilon>0$ small,
\begin{equation}\label{eqn: akdk}
\mathbf{P}(A_n, D_k) \le \left( C' \log \frac{1}{\epsilon} \right)^k \epsilon^{c'} \text{ for all } k \geq 1,
\end{equation}
if $n$ is large. Here, $C',c'>0$ are independent of $\epsilon, k$ and $n$. It then follows that
\begin{align*}
\mathbf{P}(A_n) &\le \mathbf{P}(M_K^c)+\left( C'' \log \frac{1}{\epsilon} \right)^K \epsilon^{c'}\\
&\le (1-C)^K+\left( C'' \log \frac{1}{\epsilon} \right)^K \epsilon^{c'}.
\end{align*}
Letting $K=\lceil \log \log\frac{1}{\epsilon}\rceil$, we obtain \eqref{eq: last_sec_to_show}.

We condition successively on top-most paths $T_1,\ldots, T_k$. $T_1$ is defined as the horizontal open crossing of $B(n)$ such that the region above $T_1$ is minimal. $T_2$ is then defined as the highest crossing of the region below $T_1$, and $T_i$, $i=3,\ldots, k$ is defined analogously.

For any  $k$-tuple of paths $t_1, \ldots, t_k$ that is admissible in the sense that
\[\mathbf{P}(T_1=t_1,\ldots, T_k=t_k)>0,\]
the event
\[\{T_1=t_1,\ldots, T_k=t_k\}\]
is independent of the status of edges below $t_k$ (see \cite[Prop. 2.3]{kestenlowest}).
Moreover, the event 
\[D_k \cap \{T_1=t_1, \ldots, T_k = t_k\} \]
is equal to 
\[E(t_1,\ldots,t_k)=\{T_1=t_1, \ldots, T_k = t_k\}\cap R(t_k),\]
where $R(t_k)$ is the event that there is a dual closed path from $e^*$, where $e$ is some edge on the path $t_k$, to the bottom of the box $[-n,n]^2$. Note in particular that $R(t_k)$ is independent of the status of edges on and above $t_k$.

On $E(t_1,\ldots,t_k)$, there is a unique left-most closed dual path from $t_k$ to the bottom of $[-n,n]^2$, and we denote it by $P(t_k)$. It is characterized by the following three-arm condition: each dual edge on $p$ has one closed dual arm to the path $t_k$, a disjoint closed dual arm to the bottom of the box, and an open arm to the left side of the box in the region below $t_k$. Given a closed dual path $p$ in the region below $t_k$, the event $E(t_1,\ldots,t_k)\cap \{P(t_k)=p\}$ is independent of the status of edges in the region below $t_k$ and to the right of $p$. Our goal will be to use this independence to connect at least $\sqrt{\epsilon} n^2 \pi_3(n)$ points to $T_k$ by two disjoint open paths and to $P(T_k)$ by a closed dual path. On $D_k$, we can uniquely define the edge $e_k=\{ x_k,y_k \}$ where $P(T_k)$ meets $T_k$.

We can assume $\epsilon \le 1$. Let $\nu$ be chosen such that
\begin{equation}\label{eq: delta_nu}
\nu < \frac{\alpha}{2+\alpha}\delta \text{ and } 0 < \nu < \delta<1/2,
\end{equation}
where $\alpha>0$ is any number such that 
\begin{equation}\label{eq: eta_definition}
n^\alpha \pi_1(n) \to 0.
\end{equation} 
Here, $\pi_1(n)$ is the one-arm probability $\pi_1(n) =\mathbf{P}(0 \to \partial B(n))$. Let $\alpha(T_k)$ and $\beta(T_k)$ be the left and right endpoints, respectively, of $T_k$. By RSW, we have, for some $c=c(\nu)>0$, $n$ large, $\epsilon$ small, and all $k$,
\begin{align*}
\mathbf{P}(\mathrm{dist}(e_k,\alpha(T_k)) < 2\epsilon^\nu n ,~D_k)&\le \epsilon^c  \\
\mathbf{P}(\mathrm{dist}(e_k,\beta(T_k)) < 2\epsilon^\nu n,~D_k)&\le \epsilon^c.
\end{align*}
Also by RSW, we can arrange that $T_k$ remains at distance $10\epsilon^\nu n$ from the each of the 4 corners: for $n$ large, $\epsilon$ small, and all $k$,
\[\mathbf{P}(\mathrm{dist}(T_k \cup P(T_k),Corner_i)<10\epsilon^\nu n,~D_k) \le \epsilon^c,\]
for $i=1,\ldots,4$ and
\begin{align*}
Corner_1 &= (-n,-n) & Corner_2 &= (n,- n)\\
Corner_3 &= (n,n) & Corner_4 &= (-n,n).
\end{align*}

On $D_k$,  let $I_k=I_k(n)$ be the event that
\begin{align}
\mathrm{dist}(e_k,\alpha(T_k)) &\ge 2\epsilon^\nu n, \label{rightside}\\
\mathrm{dist}(e_k,\beta(T_k)) &\ge 2\epsilon^\nu n,  \label{leftside}\\
\mathrm{dist}((T_k \cup P(T_k)),Corner_i)& \ge 10\epsilon^\nu n,\quad  i=1,\ldots,4,
\end{align}
and there is pair $u,v\in \mathbf{Z}^2$ such that $u\in T_k$, $v\in P(T_k)$,
\begin{equation}\label{eqn: far}
\mathrm{dist}(u,e_k)\ge 10\epsilon^\nu n, \quad  \mathrm{dist}(v,e_k)\ge 10\epsilon^\nu n
\end{equation}
and
\begin{equation}
\label{eqn: close}
\mathrm{dist}(u,v)< \epsilon^\delta n.
\end{equation}

\begin{figure}
\centering
\includegraphics[scale = 0.75]{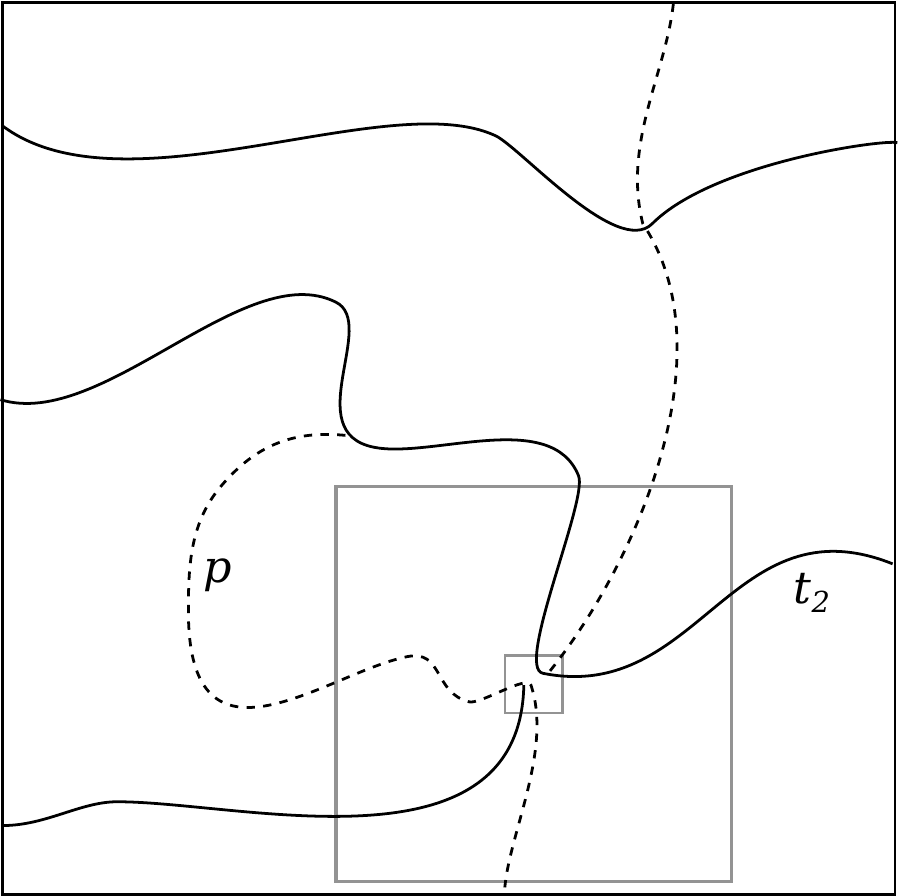}
\caption{An illustration of the estimate for $I_k$ with $k=2$. If $p$ and $t_k$ come too close together at some point, a six-arm event occurs in an annulus around that point.}
\label{bprime}
\end{figure}

\begin{lma}
Let $\delta>\nu>0$ be in \eqref{eq: delta_nu}. There exist $C,\eta'>0$ such that for small enough $\epsilon>0$,
\[
\mathbf{P}(I_k) \leq \left( C \log \frac{1}{\epsilon} \right)^k \epsilon^{\eta'} \text{ for all } k \geq 1
\]
if $n$ is large.
\end{lma}

\begin{proof}

Consider an overlapping tiling of $\mathbb{Z}^2$ by $2r\times 2r$ squares, where $r=2\epsilon^\delta n$, defined as $xr + [-r,r]^2$ for $x \in \mathbb{Z}^2$. Note that for some $x$, any choice of $u$ and $v$ in the definition of $I_k$ both lie in $B_x(r)$. Indeed, choose $x$ such that $\|u-xr\|_\infty \leq r/2$. Then $\|v-xr\|_\infty < r/2 + \epsilon^\delta n = r$. Write $B_r(u,v)$ for this box.

We consider two cases.

\paragraph{Case 1:} There is a choice of $u,v$ such that the mid-point of $u$ and $v$ lies at distance greater than $2\epsilon^\nu n$ from the boundary $\partial [-n,n]^2$. In this case we have the

\begin{claim} For $\epsilon$ small, the conditions \eqref{eqn: far} and \eqref{eqn: close} induce a six-arm event (one of the arms having at most $k-1$ defects) in an annulus centered at a point $xr$ with inner dimension $2\epsilon^\delta n$ and outer dimension $2\epsilon^\nu n$. 
\end{claim}

% If $u$ and $v$ do not lie in the same square, by \eqref{eqn: close}, we can shift the lattice by one of the vectors $(r/2,0)$, $(0,r/2)$ and $(r/2,r/2)$ and they will now lie in the same square $B_r(u,v)$.

%To see this, consider the $\ell_{\infty}$ neighborhood of $u$ of size $n^{1-\delta}$. It touches at most three other squares, and its boundary is at most $r/2$ from the boundary of each of these squares.

%Thus, in one of four shifted versions of our original lattice, $u$ and $v$ lie in the same square. 

Since $u\in T_k$, it has two disjoint open arms to the vertical sides of $[-n,n]^2$, and a closed arm with at most $k-1$ defects to the top. These induce corresponding crossings of the annulus $B_{\epsilon^\nu n}(u,v) \setminus B_r(u,v)$ of outer radius $\epsilon^\nu n$ with the same center as $B_r(u,v)$. (Note that this annulus is contained in $[-n,n]^2$ for $\epsilon$ small due to the assumption of case 1.) Similarly, $v\in P(T_k)$, but by condition \eqref{eqn: far}, it is distance at least $10\epsilon^\nu n$ from its endpoint, so it has two closed arms and an open arm which also traverse the annulus. 
%The assumption on the distance of $u$ and $v$ to the boundary of $[-n,n]^2$ ensures that the arms extend to distance at least $2n^{1-\nu}$.

\paragraph{Case 2:} We deal with the case where the mid-point of any such $u$ and $v$ lies within distance $2\epsilon^\nu n$ of the boundary of $[-n,n]^2$, but at distance at least $5 \epsilon^\nu n$ away from its corners. As previously, there is a square $B_r(u,v)$ containing $u\in T_k$ and $v \in P(T_k)$. Let $d$ denote the distance of this square to $\partial [-n,n]^2$. Then there are 6 arms (one with at most $k-1$ defects) from $B_r(u,v)$ to $B_d(u,v)$, with the same center as $B_r(u,v)$, and at least $3$ arms (one with at most $k-1$ defects) from $B_d(u,v)$ to $B_{\epsilon^\nu n}(u,v)$. Furthermore, these $3$ arms occur in a half-space. The reader may verify this is true no matter which side the mid-point is near; for instance, in the case that it is near the left side of the square, we may choose for the $3$ arms the following paths: the portion of $T_k$ from $u$ leading to the right side of the square, a closed dual path with at most $k-1$ defects leading from $u$ to the top of the square, and the portion of $p$ leading from $v$ to the bottom of the square.
%In the other cases (near the other sides), there are still at least 5 arms, but in a different configuration.

The contribution of the pairs of points corresponding to Case 1 to the probability of $I_k$ is bounded by the probability that one of the annuli (with center at least $2\epsilon^\nu n$ from the boundary) has a 6-arm point: there exists $C=C(\nu,\delta)$ such that if $\epsilon$ is small, then our upper bound is 
\[\left(\frac{n}{\epsilon^\delta n}\right)^2 \left(C \log \frac{1}{\epsilon}\right)^k\pi_6(\epsilon^\delta n,\epsilon^\nu n) \text{ for all } k \geq 1\]
if $n$ is large. Here we have used asymptotics \cite[Proposition 18]{nolin} for probabilities of arm events with defects. By the universal behavior of the 5-arm exponent \cite[Lemma 5]{KSZ} \cite[Theorem 24, 3.]{nolin} and Reimer's inequality, we have
\[\pi_6(\epsilon^\delta n,\epsilon^\nu n)\le C \left(\frac{\epsilon^\delta n}{\epsilon^\nu n}\right)^{2+\alpha}\le C \epsilon^{(2+\alpha)(\delta-\nu)} \text{ for }n \text{ large},\]
where $\alpha$ is from \eqref{eq: eta_definition}.
It follows from \eqref{eq: delta_nu} that for $\epsilon$ small, the sum is bounded for some $\eta'>0$ by
\[\left(C \log \frac{1}{\epsilon} \right)^k \epsilon^{\eta'} \text{ for all } k \geq 1\]
if $n$ is large.

To bound the contribution from points near the boundary, we sum over positions of boxes close to the boundary, using the universal behavior of the exponent for $\pi_3^H$, the half-plane 3 arm probability \cite[Theorem~24]{nolin}. We obtain the bound
\begin{align*}
\frac{n}{\epsilon^\delta n}\left(C \log \frac{1}{\epsilon}\right)^k \sum_{l=1}^{\lceil 3 \epsilon^{\nu-\delta}\rceil} \pi_6(\epsilon^\delta n,l\epsilon^\delta n)\pi_3^H (l\epsilon^\delta n,\epsilon^\nu n)& \le \epsilon^{\delta-2\nu}\left(C \log \frac{1}{\epsilon}\right)^k \sum_{l=1}^{\lceil 3\epsilon^{\nu-\delta}\rceil}l^{-\alpha}\\
&\le \left(C \log \frac{1}{\epsilon}\right)^k \epsilon^{\delta-2\nu+(1-\alpha)(\nu-\delta)}\\
&\le \left(C \log \frac{1}{\epsilon}\right)^k \epsilon^{\alpha\delta-\nu(1+\alpha)}.
 \end{align*}
Using \eqref{eq: delta_nu}, this will be bounded by $\left( C \log \frac{1}{\epsilon} \right)^k \epsilon^{\eta'}$ for $\eta'>0$ sufficiently small.
\end{proof}

It follows that to estimate $\mathbf{P}(A_n, D_k)$, we can write
\begin{align*}
\mathbf{P}(A_n,D_k) &\le 10\epsilon^c +\mathbf{P}(D_k, I_k) + \left( C \log \frac{1}{\epsilon} \right)^k \epsilon^c \\ 
&\quad +\sum_{t_1,\ldots, t_k, p} \mathbf{P}(A_n,D_k, T_i=t_i~ \forall i, P(T_k)=p),
\end{align*}
where the sum is only over $t_i$ and $p$ such that $I_k$ does not occur and
\begin{enumerate}
\item $\mathrm{dist}((T_k \cup P(T_k)),Corner_i)> 10\epsilon^\nu n$, for all $i=1,\ldots,4$ and
\item $e_k$ is at least distance $2\epsilon^\nu n$ from the bottom and right sides of $[-n,n]^2$,
\end{enumerate}
where $e_k$ is the edge where $p$ and $t_k$ meet. Condition 2 follows from a similar half-plane $3$-arm argument: this is where the term $\left( C \log \frac{1}{\epsilon} \right)^k \epsilon^c$ comes from.

Consider concentric annuli 
\[Ann(e_k,2^l)=B(e_k,2^l)\setminus B(e_k,2^{l-1}),\]
for $l= 4+\lceil\log_2 \epsilon^\delta n\rceil, \ldots, \lceil\log_2 \epsilon^\nu n \rceil$. On $\{D_k, T_i=t_i~\forall i, P(T_k)=p\}$, if $Ann(e_k, 2^l)$ contains more than $\epsilon n^{2}\pi_3(n)$ points connected to $t_k$ by two disjoint open paths, and to $p$ by a closed dual path, then 
\[\tilde{L}_n \ge \epsilon n^2 \pi_3(n).\]
We claim that with probability bounded away from 0 independently of $l$, $n$, $t_k$ and $p$, the number of such edges in $Ann(e_k,2^l)$ is bounded below by
\begin{equation}\label{eq: secondmoment-lwr}
C\epsilon^{2\delta} n^2\pi_3(\epsilon^\delta n).
\end{equation}
From \eqref{eq: delta_nu}, this is bounded below by $\epsilon n^{2} \pi_3(n)$ for $\epsilon$ small.

To obtain the lower bound \eqref{eq: secondmoment-lwr} with uniformly positive probability, we use the second moment method to find a large number of three arm points in a box inside the region $\mathcal{R}$ below $t_k$ and to the right of $p$. To this effect, we need to show that it is always possible to find such a box of side-length $r$ at least $\epsilon^\delta n$. In addition, to use RSW and connect the three arm points to $p$ and $t_k$, we need the box to be at a distance from these crossings that is roughly comparable to $r$, and the crossings themselves to be separated on this scale.

Define the annulus 
\[Ann_l'\equiv B(e_k,7/4\cdot 2^{l-1})\setminus B(e_k, 5/4 \cdot 2^{l-1}).\]

\begin{claim}\label{claim: box}
Suppose that $\mathrm{dist}(Ann(e_k,2^l)\cap t_k, Ann(e_k,2^l)\cap p)\ge \epsilon^\delta n$. For each annulus $Ann(e_k,2^l)$, there is a box $B$ of dimensions $r \times r$ with $r \ge (1/10)\epsilon^\delta n$ centered at a point in $\mathcal{R}\cap Ann_l'$, and such that also $B \cap Ann_l' \subset \mathcal{R}$. Moreover, 
\begin{align*}
\partial B\cap p\cap Ann_l'  &\neq \emptyset,\\
\partial B\cap t_k\cap  Ann_l'  &\neq  \emptyset.
\end{align*}
\begin{proof}
Starting from the first crossing by the closed arm with $k-1$ defects from $e_k$, enumerate the crossings of $Ann(e_k,2^l)$ by the arms emanating from $e_k$. We let the $t_k(l)$ be the last crossing of the annulus by $t_k$ in this clockwise order, and $p(l)$ be the first crossing of $Ann(e_k, 2^l)$ after $t_k$.

Let $\mathcal{U}$ be the region bounded by $t_k(l)$ and $p(l)$, and the segments of $\partial B(e_k,2^{l-1})$ and $\partial B(e_k,2^l)$, respectively, between the endpoints of these two crossings, always in the clockwise order. Let 
\begin{equation}
\label{eqn: sdef}
\mathcal{S}=\mathcal{U}\cap \mathcal{R}.
\end{equation}
 By the definition of $t_k(l)$ and $p(l)$, we have $\mathcal{S}\neq\emptyset.$
Since $\mathcal{U}$ contains no crossing of $Ann(e_k,2^l)$ by either $p$ or $t_k$, the boundary $\partial S\subset \partial \mathcal{U}\cup\partial \mathcal{R}$ consists of $t_k(l)$, $p(l)$, portions of $\partial B(e_k,2^{l-1})$ and $\partial B(e_k,2^l)$, and finitely many arcs of $p$ or $l$ with both endpoints on either $\partial B(e_k,2^{l-1})$ or $\partial B(e_k,2^l)$. 

Similarly, we let $t_k'(l)$ be the last crossing, in the clockwise order, of $Ann_l'$ by $t_k(l)$ and $p'(l)$ the first crossing after $t_k'(l)$. These crossings, together with segments of $\partial B(e_k,5/4\cdot 2^{l-1})$ and $\partial B(e_k,7/4\cdot 2^{l-1})$ between their endpoints, delimit a region $\mathcal{U'}\subset \mathcal{U}$. Finally, we let $\mathcal{S}'=\mathcal{U}'\cap \mathcal{R}$.

In particular, $\mathcal{S}$ and $\mathcal{S}'$ are Jordan domains, and so there exists a path $\gamma:[0,1]\rightarrow \mathcal{S}'$ with $\gamma(0)\in t'_k(l)$ and $\gamma(1)\in p'(l)$.

We now define the compact sets 
\begin{align*}
S_1'&=\overline{\mathcal{S}'\cap p}\\
S_2'&=\overline{\mathcal{S}'\cap t_k}.
\end{align*}
Letting 
\begin{align*}
d_1(t)&=\mathrm{dist}_\infty(\gamma(t),S_1')\\
d_2(t)&=\mathrm{dist}_\infty(\gamma(t),S_2'),
\end{align*}
if $\mathrm{dist}(p\cap Ann(e_k,2^l),t_k\cap Ann(e_k,2^l))\ge \epsilon^\delta n$, then
\begin{align*}
d_1(0)&= d_2(1) = 0,\\
d_1(1)&>0,\\
d_2(0)&>0.
\end{align*}
By continuity, $(d_1-d_2)(t_0)=0$ for some point $t_0\in (0,1)$. Consider the box $B=B(\gamma(t_0),d_1(t_0))$. Since $\mathrm{dist}(S_1',S_2')\ge \epsilon^\delta n$, $B$ has side length $r \ge \epsilon^\delta n/10$, and $\partial B$ contains two points $q_1\in S_1'$ and $q_2\in S_2'$.
\end{proof}
\end{claim}

By Claim \ref{claim: box}, we can choose points $q_1\in \partial B\cap p\cap Ann_l'$ and $q_2\in \partial B\cap t_k\cap Ann_l'$. Our goal is to use RSW to connect three arm points from the inside of $B$ to $p$ and $t_k$ close to $q_1$ and $q_2$, respectively. It remains to ensure that the configuration of paths in a neighborhood of $q_1$ and $q_2$ allows us to do this with positive probability. This is the purpose of the final step in our construction.

\begin{df}We say that $q_1$ and $q_2$ have \emph{linear separation less than} $\kappa$ along $\partial (B\cap Ann_l')$ if there is a connected segment $\alpha:[0,1]\rightarrow \partial (B\cap Ann_l')$ of length $\le \kappa$ such that $q_1,q_2 \in \alpha([0,1])$.
\end{df}
In the previous definition, it is important that the segment lie in $\partial (B\cap Ann_l')$ and not merely $\partial B$. This is needed to deal with the extremal case where $B$ contains $B(e_k,5/4\cdot 2^{l-1})$ and part of the boundary $\partial B$ coincides with $\partial Ann_l'$. Note that since $B\subset \mathcal{R}$, the interior of $B$ cannot $B(e_k,(5/4)\cdot 2^{l-1})$.

Next we define two annuli 
\begin{align*}
a(q_1)&=B(q_1,r/40)\setminus B(q_1,r/80)\\ 
a(q_2)&=B(q_2,r/40)\setminus B(q_1,r/80).
\end{align*}

\begin{df}\label{def: good}We say the configuration outside $\mathcal{R}$ is \emph{good} for the box $B$ if
\begin{enumerate}
\item $q_1$ and $q_2$ have linear separation at least $r/5$ along $\partial(B\cap Ann_l')$,
\item Given any circuit $c$ in $a(q_1)$ around $B(q_1,r/80)$, when $c$ is traversed starting from any point inside $B$, the circuit intersects $p$ before $t_k$ (if it intersects the latter),
\item Given any circuit in $c'$ is $a(q_2)$ around $B(q_2,r/80)$, the circuit intersects $t_k$ before intersecting $p$ (if it intersects the latter).
\end{enumerate}
\end{df}
Let $\tilde{B}$ be the box with the same center as $B$ and a quarter of the side length. If $B\subset Ann_l'$ and the configuration is good then, by placing a closed dual arc in $a(q_1)$ and two disjoint open arcs in $a(q_2)$, any set of well-separated arms (in the sense of \cite[Definition 7]{nolin}) can be extended from the boundary of the box $\tilde{B}$ and connected to $p$ and $t_k$, respectively. Moreover, since $\mathrm{dist}(p\cap Ann(e_k,2^l),t_k\cap Ann(e_k,2^l))\ge \epsilon^\delta n$, if $r \le 5\epsilon^\delta n$, then the configuration is automatically good for $B$. 

We now use an iterative procedure, formalized in Proposition \ref{prop: iteration}. Either we can always extend arms from the smaller box $\tilde{B}$ to $\partial B$ and connect them to $p$ and $t_k$ with positive probability, or we can find a smaller box $B_1$ centered in $Ann(e_k,2^l)$ such that $p$ and $t_k$ intersect the boundary of $B_1$. $B_1$ is then a new candidate to contain at least $\epsilon^\delta n$ three arm points. Since the sizes of the boxes decrease exponentially, eventually we reach scale $\epsilon^\delta n$, in which case the points of $p$ and $t_k$ on $\partial B$ are necessarily separated on the scale of the box. In the process of the iteration, it will be necessary to replace the annulus $Ann_l'$ by progressively larger regions $\mathcal{D}(D)$ which will contain the center of $B_1$ and points of $\partial B_1\cap p$ and $\partial B_1 \cap t_k$.

\begin{prop}\label{prop: iteration}
Suppose $B$ is centered at 
\[x(B)\in\mathcal{D}(D)\equiv\{y\in\mathbb{R}^2: \mathrm{dist}_\infty(y,Ann_l')\le D\},\]
and has side-length $r$. In addition, suppose that
\begin{align*}
\partial B\cap \mathcal{D}(D)\cap p&\neq \emptyset,\\
\partial B \cap \mathcal{D}(D)\cap t_k&\neq \emptyset.
\end{align*}
There is a constant $C>0$ independent of $D$ and a choice of of landing sequence $\{I_i\}$, $i=1,2,3$ on $\partial \tilde{B}$ such that one of the three following options hold: 
\begin{enumerate}
\item every collection of three arms can be extended from $\{I_i\}$ can be extended to $p$ and $t_k$ with probability at least $C$,
\item there exists another box $B_1\in \mathcal{S}$ of side-length at most $(1/5)r$ centered at
\[x(B_1)\in\mathcal{D}(D+ r/40),\]
with $B_1\cap Ann(e_k,2^l)\subset\mathcal{R}$. Moreover, 
\begin{align*}
\partial B_1\cap t_k \cap \mathcal{D}(D+r/40) &\neq \emptyset,\\
\partial B_1 \cap p \cap \mathcal{D}(D+r/40)  &\neq \emptyset.
\end{align*}
\item there exists another box $B_1\in \mathcal{S}$ of side-length at most $(3/5)r$ centered at
\[x(B_1)\in\mathcal{D}(D),\]
with $B_1\cap Ann(e_k,2^l)\subset\mathcal{R}$. Moreover, 
\begin{align*}
\partial B_1\cap t_k \cap \mathcal{D}(D) &\neq \emptyset,\\
\partial B_1 \cap p \cap \mathcal{D}(D)  &\neq \emptyset.
\end{align*}
\end{enumerate}
\end{prop}
We apply this proposition repeatedly to the box $B$ found in Claim \ref{claim: box} (for which $D=0$), until either the first condition holds, or the side-length $r$ is smaller than $5 \epsilon^\delta n$ for the first time. By the exponential decrease for both the side length and the expansion of the region $\mathcal{D}(D)$, the box $B$ remains centered in $\mathcal{D}(D+r/40)$ and both $\partial B\cap p$ and $\partial B \cap t_k$ contain a point of $\mathcal{D}(D+r/40)$ at each iteration. Indeed, the box $B$ obtained in Claim \ref{claim: box} has side length at most $7\cdot 2^{l-2}$, whereas 
\[\mathrm{dist}_\infty(Ann_l', Ann(e_k,2^l)) = 2^{l-3}.\]
The corresponding box $B_1$ obtained from applying Proposition \ref{prop: iteration} with $D=0$ lies within $(7/40)\cdot 2^{l-2}< 2^{l-3}$ of $Ann_l'$, and there are points of $\partial B_1\cap t_k$ and $\partial B_1\cap p$ within this distance of $Ann_l'$. In subsequent iterations, the centers of the boxes remain contained in the region within distance
\[(7/40)\cdot 2^{l-2} + 7\cdot 2^{l-2}\sum_{k=1}^\infty 1/(5\cdot 40)^k< 2^{l-3} \]
of $Ann_l'$, and moreover we can find points of $\partial B_1\cap p$ and $\partial B_1\cap t_k$ in this region. Once $r \le 5 \epsilon^\delta n$, we can automatically extend arms from the inside of $\tilde{B}_1$ to $p$ and $t_k$.

For clarity, we derive the first step of the iteration (the case $D=0$) in Propositions \ref{claim: bprime} and \ref{claim: notin} below. The general case follows with nearly identical proofs, replacing $Ann_l'$ by $\mathcal{D}(D)$ at the appropriate places. A key point is that the only step in the construction where the center of the box $B_1$ moves closer to $\partial Ann(e_l,2^l)$ is when we construct circuits around $q_1$ or $q_2$ in the proofs of Proposition \ref{claim: bprime} and \ref{claim: notin}.

The next proposition shows that if the box $B$ obtained from Claim \ref{claim: box} is entirely contained in $Ann_l'$, but the configuration is not good for $B$, we can find a smaller candidate box $B_1$.

\begin{prop}\label{claim: bprime}
Suppose $B\subset Ann_l'$. If the configuration is not good for $B$, there exists another box $B_1\subset \mathcal{S}$, with side length at most $r/5$, centered at a point within distance $r/40$ of $\partial Ann_l'$, and such that
\begin{align*}
\partial B_1 \cap t_k \cap \{x:\mathrm{dist}_\infty(x,Ann_l')\le r/40 \} &\neq \emptyset,\\
\partial B_1 \cap p \cap \{x:\mathrm{dist}_\infty(x,Ann_l')\le r/40 \} &\neq \emptyset.
\end{align*}
The set $\mathcal{S}\subset \mathcal{R}\cap Ann(e_k,2^{l})$ was defined in \eqref{eqn: sdef}.
\begin{proof}
Suppose the first condition in Definition \ref{def: good} fails. Then it is easy to see that there is a box $B_1\subset B$ of side-length no greater than $r/5$ such that $\partial B_1$ contains the segment $[q_1,q_2]\subset \partial B$.

Now suppose, for example, that the condition on $a(q_1)$ fails. Then there exists a path inside $B(q_1,r/40)$ starting inside $B$, which intersects $t_k$ before intersecting $p$. The portion $\gamma'$ of this path between the last intersection with $t_k$ before $p$, and $p$ is contained in $\mathcal{S}$. Indeed, before intersecting $p$, $\gamma'$ never intersects any part of $\partial \mathcal{S}$ except for $t_k$. It must also traverse $t_k$ an even number of times, since it lies inside $\mathcal{S}$ immediately before the first intersection with $p$. (See Figure \ref{bprime}).

\begin{figure}
\centering
\includegraphics[scale = 0.85]{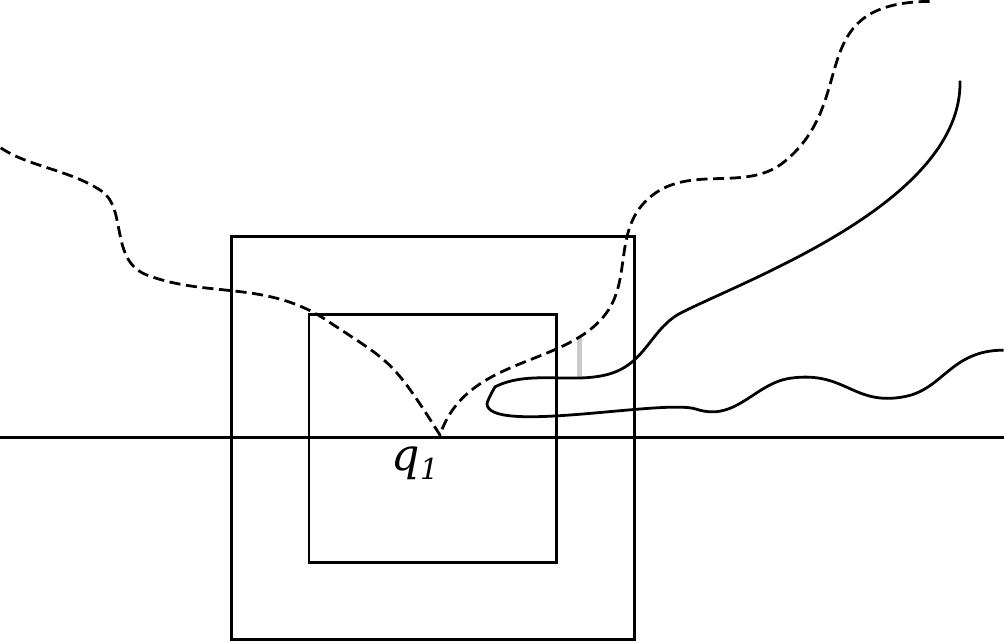}
\caption{Proposition \ref{claim: bprime}: the open path could block closed dual circuits inside the annulus used to connect to the piece of the closed path meeting the box $B$ at $q_1$, but then we can find two points of $p$ and $t_k$ even closer together.}
\label{bprime}
\end{figure}
Repeating the construction in the proof of Claim \ref{claim: box} with $\gamma'$ instead of $\gamma$, we obtain a box $B'\subset \mathcal{S}$ centered at equal distance from $p$ and $q$. Since the entire path $\gamma'$ is contained inside $B(q_1,r/40)\subset Ann(e_k,2^l)$, the box $B'$ has side length at most $r/5$.
\end{proof}
\end{prop}

For the case where $B$ intersects $\partial Ann_l'$, we have the following proposition. The proof is somewhat involved because it is necessary to keep the center of $B_1$ from being too close $\partial Ann(e_k,2^l)$. Recall that $\tilde{B}$ denotes the box with the same center as $B$ and a quarter of the side-length.
\begin{prop}\label{claim: notin}
There is a constant $C>0$ and a choice of landing sequence $\{I_i\}$, $i=1,2,3$ on $\partial \tilde{B}$ having the properties:
\begin{itemize}
\item every collection of three arms in $\tilde{B}$ from the inside of $\tilde{B}$ to $\{I_i\}$ can be extended to connect to $p$ and $t_k$ with probability at least $C$,
\item the expected number of sites of $\tilde{B}$ having three arms with landing sequence $\{I_i\}$ is at least $C r^2 \pi_3(r)$.
\end{itemize}
 or there exists a box $B'\subset\mathcal{S}$ such that
\begin{enumerate}
\item $B'$ has side length at most $(3/5)r$, is centered at a point of $Ann_l'$, and $\partial B' \cap t_k$ and $\partial B'\cap p$ intersect $Ann_l'$, or
\item $B'$ has side length at most $(1/40)r$, is centered at a point of $\{y: \mathrm{dist}_\infty(y,Ann_l')\le r/40 \}$, and such that $\partial B' \cap t_k$ and $\partial B'\cap p$ intersect $\{y:\mathrm{dist}_\infty(y,Ann_l')\le r/40 \}$.
\end{enumerate}
\begin{proof}
By Proposition \ref{claim: bprime}, we need only consider the case where $B\cap Ann_l'\neq \emptyset$. 

Let $\sigma_0$ be a side of $\partial B(e_k,5/4\cdot 2^{l-1})$ at the least distance from $\gamma(t_0)$, the center of $B$. Let $L$ be the line containing $\sigma_0$. If $B\cap \partial B(e_k,5/4\cdot 2^{l-1})\neq \emptyset$, $L$ separates $B\cap Ann_l'$ into two pieces. We let $R_1$ be the component containing the center of $B$. If $L$ does not intersect $B$ or if $B \cap \partial B(e_k, 5/4 \cdot 2^{l-1}) = \emptyset$, we let $R_1=B \cap Ann_l'$.

Note that $R_1$ is a rectangle, with aspect ratio bounded above by $2$ and below by $1/2$. Moreover, $R_1\subset \mathcal{S}$. If $q_1,q_2\in \partial R_1$ and the configuration is good, then we can extend arms from a landing sequence with $I_1$, $I_2$, $I_3$ lying in $R_1$ to connect to $p$ and $t_k$. If either $q_1$ or $q_2$ is on $\partial R_1$ and one of the conditions in Definition \ref{def: good} fails, we can proceed as in the proof of Proposition \ref{claim: bprime} to find a box $B'\subset R_1$ of side-length $\le r/5$ satisfying the conditions in Proposition \ref{claim: bprime}.

Thus, we can assume that $B\cap B(e_k,5/4\cdot 2^{l-1})$ (so $R_1\neq B$), and at least one of $q_1$, $q_2$ lies on $\partial (B\cap Ann_l')\setminus \partial R_1$. We let $s_1, s_2$ denote the parts of the sides of $B$ that are not in $B(e_k,(5/4)\cdot 2^{l-1})$, but are perpendicular to  $L$ and in the half-plane of $\mathbb{R}^2\setminus L$ which does not contain $R_1$. The definition of $R_1$ implies that each of $s_1$ and $s_2$ has length no greater than $r/2$. $s_1$ or $s_2$ may be empty.

The side of $\partial B$ parallel to $L$ which is not in $\partial R_1$ necessarily intersects $B(e_k,(5/4)\cdot 2^{l-1})$ if $R_1\neq B\cap Ann_l'$. Let $s_3$ be the part of this side which is not in $B(e_k,5/4\cdot 2^{l-1})$. $s_3$ consists of two connected segments $s_3^1$ and $s_3^2$, each possibly empty, with $s_3^1$ connected to $s_1$ and $s_3^2$ connected to $s_2$. See Figure \ref{map}.

\begin{figure}
\centering
\includegraphics[scale = 0.55]{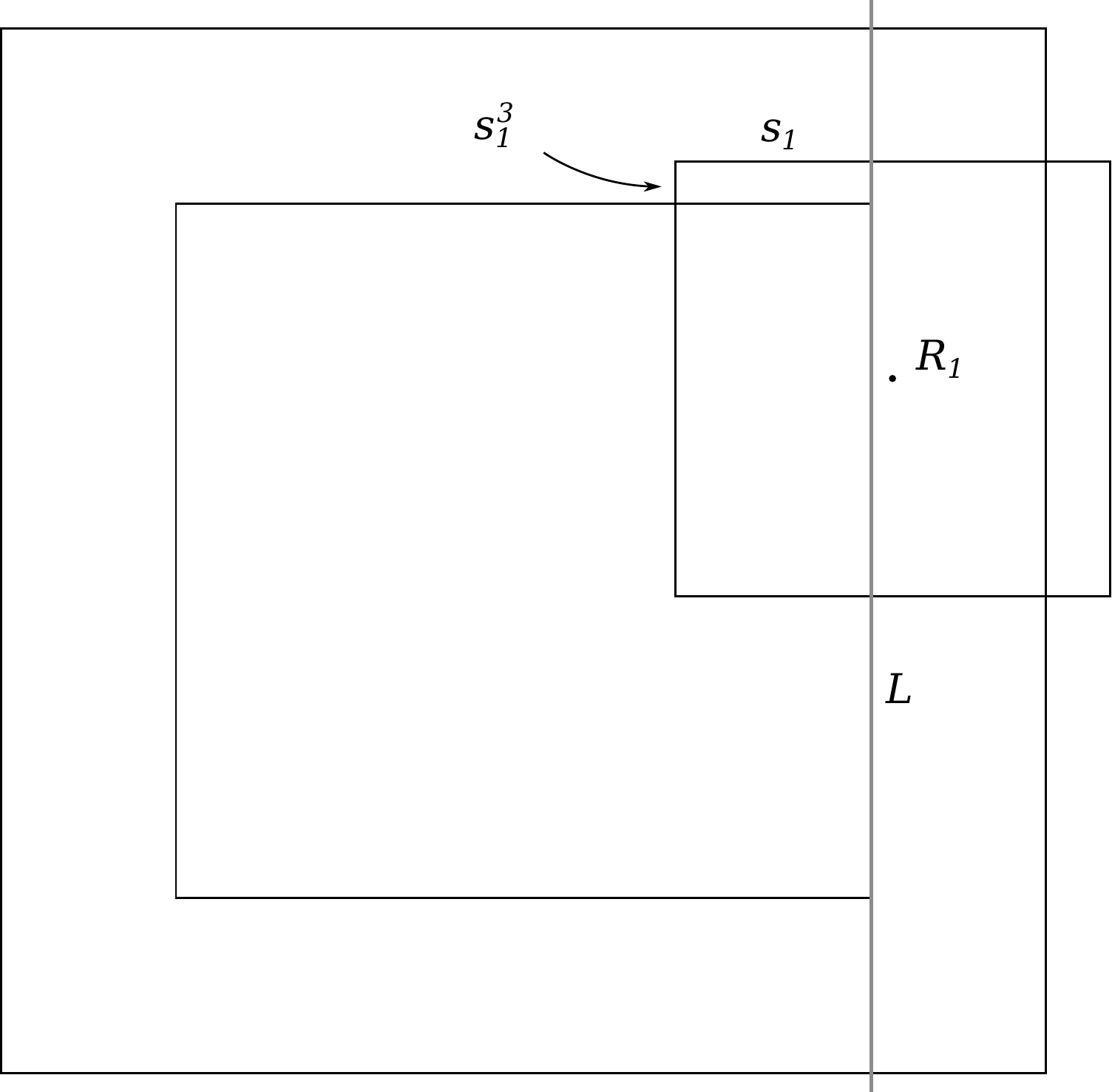}
\caption{An illustration of the case when $B\cap Ann_l' \neq \emptyset$. In this picture, $s_2=s_3^2=\emptyset$.}
\label{map}
\end{figure}

We now have a new dichotomy, which we can apply to each side $s_1$, $s_2$. We state it for $s_1$: 
\begin{itemize}
\item \emph{either} $\mathrm{dist}(s_1,B(e_k,5/4\cdot 2^{l-1}))\ge r/10$, in which case for any choice of locations for $q_1$ or $q_2$ in $s_1\cup s_3^1$, we can use RSW to route arms inside $B\cap Ann_l'$ from a box in $R_1$, provided the configuration is good in the sense of Definition \ref{def: good}, or find a smaller box satisfying condition 2 in the statement of the proposition. 

\item \emph{or} $\mathrm{dist}(s_1,B(e_k,5/4\cdot 2^{l-1}))\le r/10$; in this case $s_3^1$ has length no greater than $r/10$. 

If both $q_1$ and $q_2$ are on $s_1\cup s_3^1$, then they are at linear distance along the boundary less than $r/10+r/2 =(3/5)r$. This implies that we can find a line segments of length $\leq (3/5)r$ joining $q_1$ to $q_2$, and a new box $B_1$ satisfying condition 1 in the statement of the proposition.

If, say, $q_1\in s_1\cup s_3^1$, we place a rectangular region of width $r/20$ along $s_1\cup s_3^1$, outside $B$, and extend it by $r/20$ onto the boundary of $R_1$ (see Figure \ref{fig: continuation}). Either every continuous path traversing this region from its end abutting $\partial R_1$ reaches
\[S_1= \overline{\mathcal{S}\cap p}\] before
\[S_2=\overline{\mathcal{S}\cap t_k},\] or there is a path in $\mathcal{R}$ contained in the region of $\ell_\infty$ diameter $r/2+2r/20=(3/5)r$ joining $S_1$ to $S_2$ inside $\mathcal{R}$. In the former case, we can use RSW to extend a closed dual arm from a landing site in $R_1$ to connect it around $q_1$ to $p$. In the latter case we can find a new box $B'$ as before.
\end{itemize}

\begin{figure}
\centering
\includegraphics[scale = 0.40]{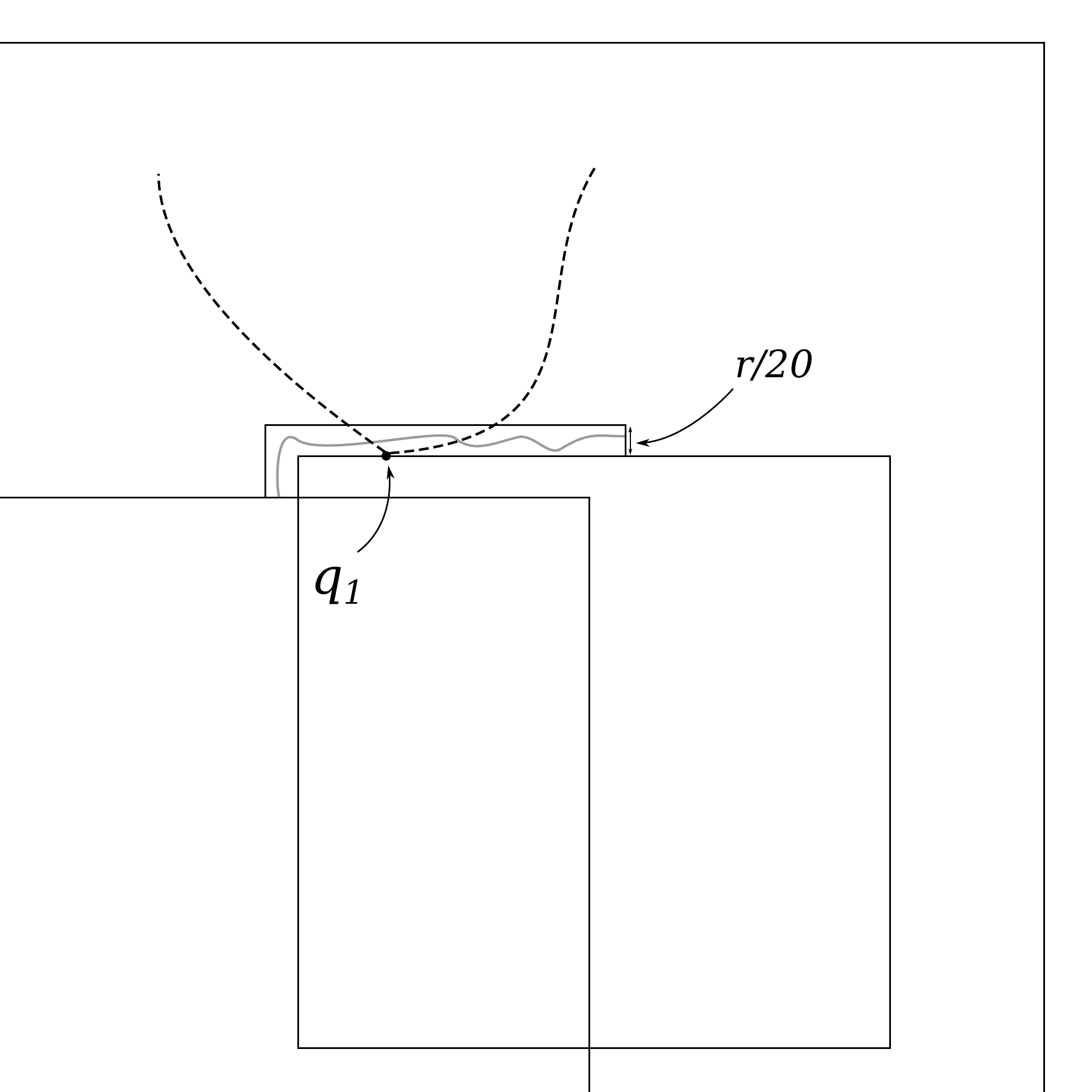}
\caption{Continuing an arm along $s_1\cup s_1^3$: either such a continuation is always possible with positive probability using RSW, or a smaller box can be found.}
\label{fig: continuation}
\end{figure}

To summarize, the previous alternative implies that if $s_1\cup s_2\cup s_3$ contains both $q_1$ and $q_2$, we can either extend arms from a landing sequence in $R_1\cap \partial \tilde{B}$ or we can find a box $B'$ in $R$ satisfying the conditions of the proposition.

It remains only to remark on a final, and somewhat degenerate case, when either $q_1$ or $q_2$ lies on $\partial B \setminus (\partial R_1 \cup s_1 \cup s_2 \cup s_3)$. This can only happen if part of $\partial B$ coincides with $B(e_k,5/4\cdot 2^{l-1})$. Since $R_1$ lies on the side of $L$ opposite $B(e_k,5/4\cdot 2^{l-1})$, at most half of any side of $\partial B$ can intersect $\partial B(e_k,(5/4)\cdot 2^{l-1})$. See Figure \ref{fig: degenerate} for an illustration. This is obvious for the sides perpendicular to $L$. For the remaining side of $\partial B$ to coincide with part of $\partial B(e_k,5/4\cdot 2^{l-1})$ while $q_1$ or $q_2$ lies in the intersection, $r$ has to be at least $(5/2)\cdot 2^{l-1}$. It follows that the length of the intersection is at most $(1/2)r$, and this last case can be treated like the second case in the dichotomy above.
\end{proof}
\end{prop}

\begin{figure}
\centering
\includegraphics[scale = 0.40]{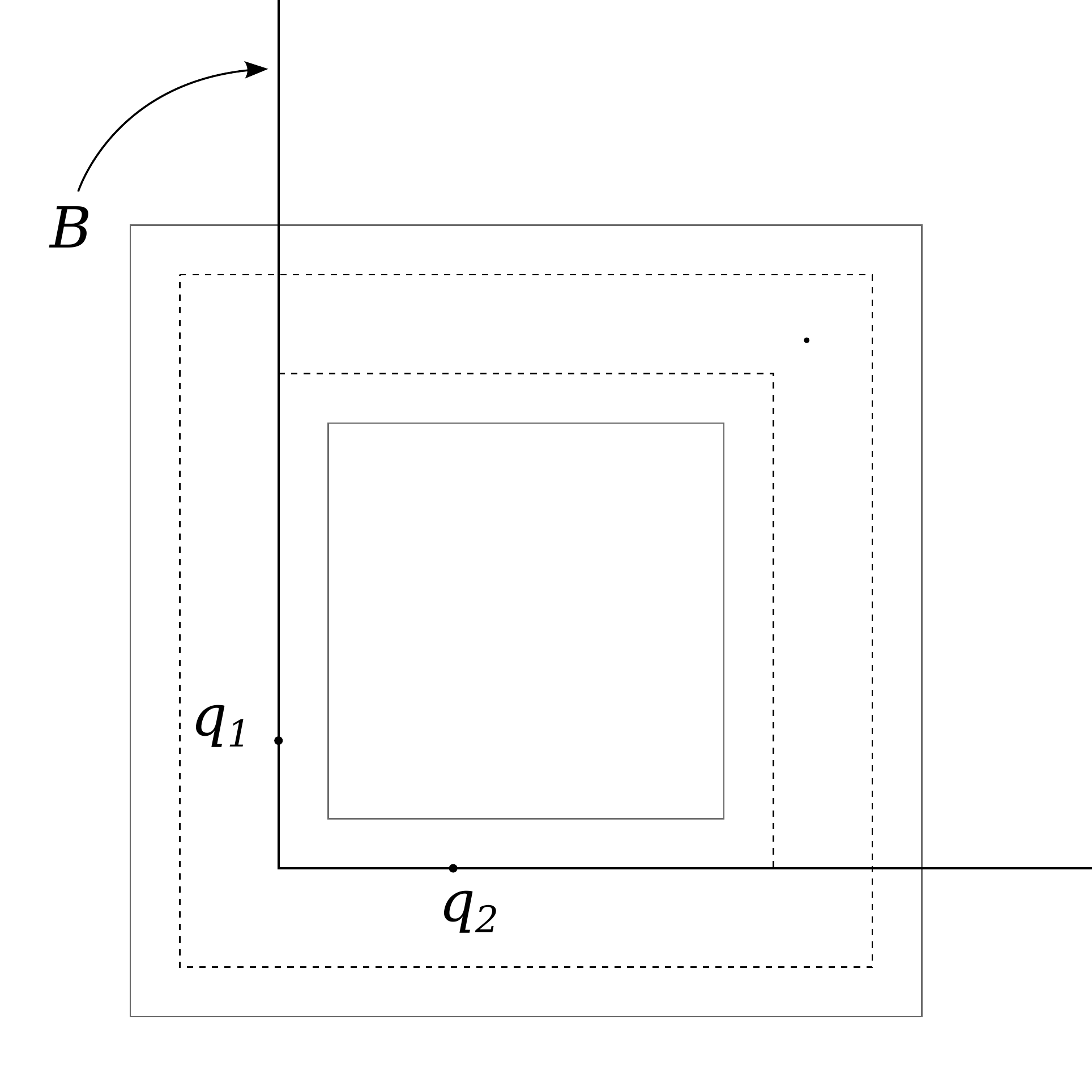}
\caption{A configuration where segments of two sides of $\partial B$ coincide with $\partial B(e_k,5/4\cdot 2^{l-1})$. The dotted annulus is $Ann_l'$.}
\label{fig: degenerate}
\end{figure}

We can now conclude the proof. On $\{D_k, T_k=t_k, P(T_k)=p\}$, let $F(2^l,t_k,p)$ be the event there are no more than $C\epsilon^{2\delta}n^{2}\pi_3(\epsilon^\delta n)$ edges connected to $t_k$ by two disjoint open paths and to $p$ by a closed dual path in the region below $t_k$ and to the right of $p$ in $Ann(e_k,2^l)$.
Then:
\begin{align*}
&\mathbf{P}(A_n, D_k, T_1=t_1,\ldots,T_k=t_k, P(T_k)=p)\\ 
\le &~\mathbf{P}(T_1=t_1,\ldots, T_k=t_k, D_k, P(T_k)=p, \cap_l F(2^{l},t_k,p))\\
\le &~(1-C)^{c\log_2 \frac{1}{\epsilon}}\mathbf{P}(T_1=t_1,\ldots,T_k=t_k, P(T_k)= p).
\end{align*}
Putting everything together, we find:
\[\mathbf{P}(A_n,D_k) \le \epsilon^{c'} + \left( C \log \frac{1}{\epsilon} \right)^k \epsilon^{c'}\]
for $c'>0$ small enough, which is \eqref{eqn: akdk}.

\section{Lemmas of a topological nature}
\label{sec: lemmata}
In this section, we complete the detour construction by providing proofs of the proposition and lemmas assumed in Sections \ref{sec: short} and \ref{sec: detours}.

First, recall Proposition \ref{prop: four}: If $\omega\in C_n$, then for distinct edges $e,f \in \gamma_n$, $\pi(e)$ and $\pi(f)$ are either equal or have no vertices in common. 

Regarding $\sigma_n$, we establish the following two properties:
\begin{enumerate}
\item (\emph{Lemma \ref{circuit}}) For $\omega \in C_n,~ \sigma_n$ is an open circuit in $A(n)$ surrounding the origin.
\item (\emph{Lemma \ref{lem: piempty}}) For $\omega \in C_n$, if $e \in \hat \gamma_n \setminus \hat{\Pi}$ then $\pi(e) = \emptyset$.
\end{enumerate}

\subsection{Proof of  Proposition \ref{prop: four}}
Because $\pi(e)$ was defined as the first element of $\mathcal{S}(e)$ in a deterministic ordering, we see that it will suffice to show: given $P(e) \in \mathcal{S}(e)$ and $P(f) \in \mathcal{S}(f)$, if
\begin{equation}
V(P(e)) \cap V(P(f)) \neq \emptyset, \label{eqn: intersection}
\end{equation}
(their vertex sets intersect) then 
\begin{equation}
P(e) \in \mathcal{S}(f) \text{ and } P(f) \in \mathcal{S}(e).
\end{equation}
Now, suppose the intersection of $\pi(e)$ and $\pi(f)$ is non-empty. Then $\pi(e)\in \mathcal{S}(e)\cap\mathcal{S}(f)$, so $\pi(e)=\pi(f)$.

Let $P(e)\in S(e)$. By extending the two open ends of the closed path in condition 4. to meet the midpoints of the edges $\{w_0,w_0-\mathbf{e}_1\}$ and $\{w_M, w_M+\mathbf{e}_1\}$, we can form a Jordan curve $\theta(e)$ (i.e. a continuous, self-avoiding closed curve) by traversing the closed dual path from $w_0+(1/2)(-\mathbf{e}_1+\mathbf{e}_2)$ to $w_M+(1/2)(\mathbf{e}_1+\mathbf{e}_2)$, traversing the path $Q(e)$ (listed as $Q$ in the definition of $\mathcal{S}(e)$), and returning to $w_0+(1/2)(-\mathbf{e}_1+\mathbf{e}_2)$. By the Jordan Curve Theorem, any connected set of which does not intersect $\theta(e)$ must lie completely on either side of $\theta(e)$.

We apply the preceding to $P(e)\setminus\{w_0(e),w_M(e)\}$ and $P(f)\setminus \{w_0(f),w_M(f)\}$. By assumption (Condition 1.) both sets lie in the exterior of $\gamma_n$, so neither can intersect $Q(e)$ or $Q(f)$. Since they consist of open edges, they also cannot intersect either the dual portion of $\theta(e)$ or that of $\theta(f)$. Thus, assuming \eqref{eqn: intersection}, we have
\begin{claim}
Except for the endpoints $w_0(e)$, $w_0(f)$, $w_M(e)$, $w_M(f)$, the paths  $P(e)$ and $P(f)$, considered as the union of their edges and vertices, lie in the same connected components of $\mathbb{R}^2\setminus \theta(e)$ and $\mathbb{R}^2\setminus \theta(f)$.
\end{claim}
We now assert that this implies:
\begin{claim}\label{coincidence_claim}
The  endpoints of $P(e)$ and $P(f)$ coincide.
\end{claim}

To see this, we start with the following  
\begin{claim}\label{first_claim}
Any vertex of $\gamma_n \setminus Q(e)$ lies in the component of $\theta(e)^c$ which does not contain $P(e)$.
\end{claim}
\begin{proof}
The vertices $w_1(e)$ and $w_1(e)-\mathbf{e}_1$ lie on opposite sides of $\theta(e)$, since they can be connected by a segment which intersects the curve exactly once. The vertex $w_0(e)-\mathbf{e}_1$ is in the same component as $w_1(e)-\mathbf{e}_1$, so it is in the component which does not contain $P(e)$. However each vertex of $\gamma_n \setminus Q(e)$ can be connected to $w_0(e)-\mathbf{e}_1$ by simply following $\gamma_n$ (without crossing $\theta(e)$).
%On the other hand, $w_1(e)$ and $w_M(e)$ can be connected to each other by a path which does not cross either the dual path or $Q(e)$: connect $w_1(e)$ to $w_0(e)+(1/10)\mathbf{e}_2$, and travel along each edge of $Q(e)$, at $\|\cdot\|_{\infty}$ distance $1/10$ to reach $w_{M-1}(e)+(1/10)\mathbf{e}_2$, and finally $w_M(e)$. This path never crosses the dual path because $Q(u)$ is open, and also does not cross $Q(e)$. Thus $w_1(e)$ and $w_M(e)$ are in the same component of $\theta(e)^c$, and the two other points are in the other component. Since the edge $\{w_0(e),w_0(e)-\mathbf{e}_1\}=\{w_0(e),w_1(e)-\mathbf{e}_1-\mathbf{e}_2\}$ is in $\gamma_n$ by assumption, any vertex in $\gamma_n\setminus Q(e)$ is path-connected to $w_0(e)$.
\end{proof} 

Claim \ref{first_claim} implies that $P(f)$, lying as it does in the same component as $P(e)$, must have both of its endpoints in $Q(e)$ (since they must be in $\gamma_n$). The argument is symmetric, so that both endpoints of $P(e)$ must lie in $Q(f)$, so that $w_0(e)=w_0(f)$ and $w_M(e)=w_M(f)$. At this point, we have established Claim \ref{coincidence_claim}.

The coincidence of the endpoints of $P(e)$ and $P(f)$ implies also $Q(e)=Q(f)$, so that $P(e)$, together with the dual path from Condition 4. in the definition of $\mathcal{S}(e)$, satisfies Conditions 1-5 defining $\mathcal{S}(f)$. We have proved Proposition 4.

\subsection{Proof of Lemma \ref{circuit}}
By Proposition \ref{prop: four}, the set $E(\Pi)$ is a disjoint union of paths $P(e)=\pi(e)$ for a finite collection of edges $e \in \sigma_n$. Our strategy will be to inductively replace each portion $Q(e)$ of $\gamma_n$ by the detour path $P(e)$ and show that at each stage, we still have a circuit around the origin. In other words, enumerating the paths $P_1, P_2, \ldots$, and $Q_1, Q_2, \ldots$, we replace $Q_1$ with $P_1$ to create $\gamma_n^{(1)}$ from the original circuit $\gamma_n$. Then we replace a portion of $\gamma_n^{(1)}$ (which is $Q_2$) with $P_2$ to create $\gamma_n^{(2)}$, and so on. 

Note that at stage $k$, the path $P_{k+1}$ satisfies the definition of $\epsilon$-shielded detour with $\gamma_n$ replaced by $\gamma_n^{(k)}$. Indeed, since all paths $Q$ are disjoint for paths in $\Pi$, points 2-5 are obvious. Furthermore, assuming that $P_{k+1}$ satisfies point 1 with $\gamma_n$ equal to $\gamma_n^{(k-1)}$, it must also satisfy it with $\gamma_n$ equal to $\gamma_n^{(k)}$. If this were not the case, then if we write $\sigma^o$ for a path $\sigma$ excluding its endpoints, we would have $P_{k+1}^o \subset \mathrm{ext}\gamma_n^{(k-1)}$ but $P_{k+1}^o \cap \mathrm{int}\gamma_n^{(k)} \neq \emptyset$. But $P_{k+1}^o$ must also contain a point of $\mathrm{ext}\gamma_n^{(k)}$ since the paths $Q_i$ are disjoint (choose a point near an endpoint of $Q_{k+1}$), so this implies that $P_{k+1}$ must cross from the interior to the exterior. It cannot cross $\gamma_n$, so it must cross one of the $P_i$'s for $i \neq k+1$. This is a contradiction, since the $P_i$'s are disjoint.

Therefore to prove Lemma~\ref{circuit}, it will suffice to show the following:
\begin{claim}\label{lemma2claim}
Let $\sigma$ be a self-avoiding circuit surrounding the origin, and let $P$ denote a self-avoiding path with endpoints in $\sigma$, such that $P$ satisfies Conditions 1-5 in the definition of $\pi(e)$, with $\gamma_n$ replaced by $\sigma$. Then 
\[(\sigma \setminus Q) \cup P\]
is a circuit in $A(n)$ surrounding the origin. Here $Q$ is defined relative to $\sigma$ and $P$ as in Condition 3. in the definition of $\pi(e)$.
\end{claim}
We first show: 
\begin{claim} \label{alternative}
\[Q^o\subset \mathrm{int} (P\cup (\gamma_n\setminus Q)).\]
\end{claim}
The connected set $Q^o$ lies either in $\mathrm{int} (P\cup (\gamma_n\setminus Q))$ or in $\mathrm{ext} (P\cup (\gamma_n\setminus Q))$. It will suffice to exclude the latter case. For this, we will use a simple intermediate result.
\begin{lma}[The ABC lemma] \label{ABClemma} Let $A(t)$, $B(t)$ and $C(t)$, $t\in [0,1]$ be three self-avoiding, continuous curves in $\mathbb{R}^2$. Denote their images by $A=A([0,1])$, $B=B([0,1])$, $C=C([0,1])$. Suppose
\[A(0)=B(0)=C(0)=a\]
and
\[A(1)=B(1)=C(1)=b\]
with $a\neq b$, and
\[A\cap B = A\cap C = B\cap C = \{a,b\}.\]
Form the three Jordan curves $A\cup B$, $A \cup C$ and $B\cup C$, and suppose 
\[C\setminus\{a,b\} \subset \mathrm{int}(A\cup B).\]
Then 
\[\mathrm{int} (B\cup C) \subset \mathrm{int}(A\cup B).\]
\begin{proof}
Any continuous path from a point in $\mathrm{int} (B\cup C)$ to infinity must cross either $B$ or $C$. If it does not cross $B$, it must cross $C$ at a point in $\mathrm{int} (A\cup B)$, after which it must cross $A$, and the lemma follows.
\end{proof}
\end{lma}

Let us now return to the proof of Lemma \ref{circuit}. We assume
\begin{equation}\label{drpepper}
Q^o \subset \mathrm{ext} (P \cup (\gamma_n\setminus Q)).
\end{equation}
Because $P$ does not cross $\gamma_n$, two possibilities arise: either 
\begin{equation}
\label{orangina}
(\gamma_n\setminus Q)^o \subset \mathrm{int}(P\cup Q),
\end{equation}
or 
\begin{equation}\label{cherrycoke}
(\gamma_n\setminus Q)^o \subset \mathrm{ext}(P\cup Q).
\end{equation}
In case \eqref{orangina} holds, we can apply the ``ABC Lemma'' \ref{ABClemma} to find $0 \in \mathrm{int}\gamma_n \subset P\cup Q$, which is not possible by Condition 3 in the definition of $P$.

Now, assume \eqref{cherrycoke} holds. $P$ is contained the boundary of both $\mathrm{int}(P\cup Q)$ and $\mathrm{int}(P\cup (\gamma_n\setminus Q))$. Recalling that we have also \eqref{drpepper}, we see that these two Jordan domains are disjoint. Thus, every point in a sufficiently small neighborhood of a point of $P^o$ lies in exactly one of these two domains (if it is not in $P$). It follows that any curve joining a point in $P^o$ to infinity, intersecting $P^o$ only at its starting point, must intersect either $\gamma_n\setminus Q$ or $Q$, so 
\[P^o \subset \mathrm{int} \gamma_n,\]
a contradiction to Condition 2. in the definition of $\pi(e)$. At this point, we have established Claim \ref{alternative}. Applying Lemma \ref{ABClemma}, find  
\[0 \in \mathrm{int} \gamma_n \subset \mathrm{int} (P\cup (\gamma_n \setminus Q)),\]
which is Claim \ref{lemma2claim}.

\subsection{Proof of Lemma \ref{lem: piempty}}
By Proposition~\ref{prop: four}, the set $E(\Pi)$ is a disjoint union of sets of the form $\pi(e)$, where $e$ ranges over a finite collection $S$ of edges in $\hat{\Pi}$. If $e\in \hat \gamma_n\setminus \hat{\Pi}$ and $\pi(e)\neq \emptyset$ then, again by Proposition~\ref{prop: four}, 
\[\pi(f)\cap \pi(e) =\emptyset\] 
for all $f \in S$. This contradicts the maximality of $\Pi$, provided we show
\begin{claim}
For all $f\in S$, 
\[\hat{\pi}(e)\cap \hat{\pi}(f)=\emptyset.\]
\end{claim}
This is because if the intersection is non-empty, then $\hat{\pi}(e)$ and $\hat{\pi}(f)$ must share a common segment, which forces 
\begin{equation}
\pi(e)\cap \pi(f)\neq \emptyset.\label{neinter}
\end{equation}
 Indeed, if $\hat{\pi}(e)$ and $\hat{\pi}(f)$ coincide, then the two initial and two final vertices of $\pi(e)$ and $\pi(f)$ coincide by Condition 2. in the definition of $\mathcal{S}(e)$, and otherwise one endpoint of $\hat{\pi}(e)$ must be equal to some non-endpoint vertex of $\hat{\pi}(f)$. By the Jordan Curve Theorem, $\pi(e)$ must then intersect $\pi(f)$. Given \eqref{neinter}, Proposition~\ref{prop: four} now implies $\pi(e)=\pi(f)$, which contradicts the assumption on $e$.

\vskip 0.3in 

\acknow \  We wish to thank Michael Aizenman, with whom we had many inspiring discussions in the early stages of this project. In particular, he suggested to us that $\mathbf{Var}(L_n)\sim (\mathbf{E} L_n)^2$, which led us to the proof of \eqref{eqn: aizenman-sugg}. The realization that paths confined to small regions are short later turned out to be a crucial point.

We thank Eviatar Procaccia for pointing out that the work of Grimmett and Marstrand \cite{grimmettmarstrand} addresses the problem of the chemical distance in the supercritical phase. We are especially indebted to Art\"em Sapozhnikov for his detailed comments on this work, and for pointing out a mistake in an earlier version of our argument. This work was completed during the Sherman Memorial Conference at Indiana University, Bloomington, in May 2015. 

%This event provided an ideal environment for work.

\end{document}